\documentclass[a4wide,11pt,reqno]{article}
\usepackage{bm}
\usepackage{mathrsfs}
\usepackage{fullpage}
\usepackage{enumerate}
\usepackage[table]{xcolor}
\usepackage[utf8]{inputenc}
\inputencoding{latin1}
\usepackage{amssymb, amsmath, amsfonts}
\usepackage{mathrsfs}
\usepackage[colorlinks]{hyperref}
\usepackage{epsfig,amsbsy,amsthm} 
\usepackage{tikz}
\usetikzlibrary{patterns,calc,decorations.markings}
\usepackage{graphics}
\usepackage{caption}
\usepackage{subcaption}
\usepackage{float,epsfig}
\usepackage{fixmath}
\usepackage{graphics}
\usepackage{pgfpages}
\usepackage{appendix}
\renewcommand{\arraystretch}{3}
\numberwithin{equation}{section}  
\usepackage{graphicx}
\usepackage{pst-all}
\usepackage{calligra}
\usepackage{tikz}
\usepackage{calligra}
\usepackage{color}
\usepackage{tikz}
\usepackage{cleveref}
\usepackage[utf8]{inputenc}
\usepackage{algorithm}
\usepackage{multirow}
\usepackage{array}
\newcolumntype{H}{>{\setbox0=\hbox\bgroup}c<{\egroup}@{}}
\usepackage{accents}
\DeclareMathAlphabet{\mathpzc}{OT1}{pzc}{m}{it}
\DeclareMathAlphabet{\mathcalligra}{T1}{calligra}{m}{n}
\DeclareMathAlphabet{\mathpzc}{OT1}{pzc}{m}{it}
\DeclareMathAlphabet{\mathcalligra}{T1}{calligra}{m}{n}
\DeclareMathAlphabet{\mathpzc}{OT1}{pzc}{m}{it}
\DeclareMathAlphabet{\mathcalligra}{T1}{calligra}{m}{n}
\begin{document}
\newtheorem{theorem}{\bf Theorem}[section]
\newtheorem{proposition}[theorem]{\bf Proposition}
\newtheorem{definition}{\bf Definition}[section]
\newtheorem{corollary}[theorem]{\bf Corollary}
\newtheorem{exam}[theorem]{\bf Example}
\newtheorem{remark}[theorem]{\bf Remark}
\newtheorem{lemma}[theorem]{\bf Lemma}
\newtheorem{assum}[theorem]{\bf Assumption}
\newcommand{\von}{\vskip 1ex}
\newcommand{\vone}{\vskip 2ex}
\newcommand{\vtwo}{\vskip 4ex}
\newcommand{\ds}{\displaystyle}
\def \noin{\noindent}
\newcommand{\be}{\begin{equation}}
\newcommand{\ee}{\end{equation}}
\newcommand{\beno}{\begin{equation*}}
\newcommand{\eeno}{\end{equation*}}
\newcommand{\ba}{\begin{align}}
\newcommand{\ea}{\end{align}}
\newcommand{\bano}{\begin{align*}}
\newcommand{\eano}{\end{align*}}
\newcommand{\bea}{\begin{eqnarray}}
\newcommand{\eea}{\end{eqnarray}}
\newcommand{\beano}{\begin{eqnarray*}}
\newcommand{\eeano}{\end{eqnarray*}}
\newcommand{\chatv}{\hatv{c}}
\def \noin{\noindent}
\def\arraystretch{1.3}
\def \tcK{{\tilde {\mathcal K}}}    
\def \O{{\Omega}}
\def \cT{{\mathcal T}}
\def \cV{{\mathcal V}}
\def \cE{{\mathcal E}}
\def \R{{\mathbb R}}
\def \V{{\mathbb V}}
\def \S{{\mathbb S}}
\def \N{{\mathbb N}}
\def \Z{{\mathbb Z}}
\def \Mc{{\mathcal M}}
\def \Cc{{\mathcal C}}
\def \Rc{{\mathcal R}}
\def \Ec{{\mathcal E}}
\def \Gc{{\mathcal G}}
\def \Tc{{\mathcal T}}
\def \Qc{{\mathcal Q}}
\def \Ic{{\mathcal I}}
\def \Pc{{\mathcal P}}
\def \Oc{{\mathcal O}}
\def \Uc{{\mathcal U}}
\def \Yc{{\mathcal Y}}
\def \Ac{{\mathcal A}}
\def \Bc{{\mathcal B}}
\def \k{\mathpzc{k}}
\def \Rp{\mathpzc{R}}
\def \Os{\mathscr{O}}
\def \Js{\mathscr{J}}
\def \Es{\mathscr{E}}
\def \Fs{\mathscr{F}}
\def \Qs{\mathscr{Q}}
\def \Ss{\mathscr{S}}
\def \Cs{\mathscr{C}}
\def \Ds{\mathscr{D}}
\def \Ms{\mathscr{M}}
\def \Ts{\mathscr{T}}
\def \LL{L^{\infty}(L^{2}(\Omega))}
\def \LH{L^{2}(0,T;H^{1}(\Omega))}
\def \B {\mathrm{BDF}}
\def \el {\mathrm{el}}
\def \re {\mathrm{re}}
\def \e {\mathrm{e}}
\def \div {\mathrm{div}}
\def \CN {\mathrm{CN}}
\def \Rs   {\mathbf{R}_{{\mathrm es}}}
\def \Rb {\mathbf{R}}
\def \Jb {\mathbf{J}}
\def  \apos {\emph{a posteriori~}}

\def\mean#1{\left\{\hskip -5pt\left\{#1\right\}\hskip -5pt\right\}}
\def\jump#1{\left[\hskip -3.5pt\left[#1\right]\hskip -3.5pt\right]}
\def\smean#1{\{\hskip -3pt\{#1\}\hskip -3pt\}}
\def\sjump#1{[\hskip -1.5pt[#1]\hskip -1.5pt]}
\def\jumptwo{\jump{\frac{\p^2 u_h}{\p n^2}}}
\title{An $hp$-adaptive discontinuous Galerkin discretization of a static anti-plane shear crack model}
\author{Ram Manohar\thanks{Department of Mathematics \& Statistics, Texas A\&M University-Corpus Christi, TX- 78412, USA \tt{(ram.manohar@tamucc.edu)}.} 
~~ and ~~S. M. Mallikarjunaiah\thanks{Department of Mathematics \& Statistics, Texas A\&M University-Corpus Christi, TX- 78412, USA \tt{(M.Muddamallappa@tamucc.edu)}.}}

\date{}
\maketitle
\textbf{Abstract.}{{We propose an $hp$-adaptive discontinuous Galerkin finite element method (DGFEM) to approximate the solution of a static crack boundary value problem. The mathematical model describes the behavior of a geometrically linear strain-limiting elastic body. The compatibility condition for the physical variables, along with a specific algebraically nonlinear constitutive relationship, leads to a second-order quasi-linear elliptic boundary value problem. We demonstrate the existence of a unique discrete solution using Ritz representation theory across the entire range of modeling parameters. Additionally, we derive a priori error estimates for the DGFEM, which are computable and, importantly, expressed in terms of natural energy and $L^2$-norms. Numerical examples showcase the performance of the proposed method in the context of a manufactured solution and a non-convex domain containing an edge crack. }
~\\ 
		
	
\vspace{.1in}

\section{Introduction}
\subsection{Formulation of the model problem.}
Let  $\Omega \subset \mathbb{R}^{d}$, $d\geq 2$, be the bounded polygonal domain with Lipschitz boundary $\Gamma:=\partial \Omega$. Assume that the $(d-1)$-dimensional measure of boundary $\Gamma$ is positive. Let $\Omega$ be occupied by an elastic body whose response is defined by a special nonlinear constitutive relation between Cauchy stress tensor $\boldsymbol{T} \colon \Omega \to \mathbb{R}^{d \times d}_{sym}$ and linearized elasticity tensor $\boldsymbol{\epsilon} \colon \Omega \to \mathbb{R}^{d \times d}_{sym}$. The constitutive class of relations within the setting of \textit{strain-limiting theories of elasticity} is of the form \cite{rajagopal2007elasticity}: 
\begin{equation}\label{model1}
\mathcal{F}(\boldsymbol{T}, \; \boldsymbol{B}) = \boldsymbol{0},
\end{equation}
where $\mathcal{F} \colon \mathbb{R}^{d \times d}_{sym} \times \mathbb{R}^{d \times d}_{sym} \to \mathbb{R}^{d \times d}_{sym}$ is  a nonlinear, tensor-valued function. By performing a linearization under the assumption of small displacement gradients, one can arrive at nonlinear constitutive relations such as \cite{mallikarjunaiah2015direct}: 
\begin{equation}\label{model2}
\boldsymbol{\epsilon} :=\mathcal{F}(\boldsymbol{T}).
\end{equation}
The system of equations governing the material behavior is: 
\begin{align}\label{model_1}
\begin{cases}
\div \, \boldsymbol{T} = \boldsymbol{f}, \quad \boldsymbol{T}=\boldsymbol{T}^T & \mbox{in} \quad \Omega,  \\
\boldsymbol{\epsilon} :=\mathcal{F}(\boldsymbol{T})  & \mbox{in} \quad \Omega. 
\end{cases}
\end{align}
The first part of the first equation in Eq-\eqref{model_1} is the \textit{balance of linear momentum}, and the second part of the first equation implies that the stress tensor is \textit{symmetric} under no body couples. For the case of anti-plane shear loading, the displacement vector has only one nonzero component (i.e., in $z$-direction). Hence, stress and strains are planar with only a nonzero third component. Let $u(x,y)$ be the Airy's stress function which satisfies
\begin{equation}\label{Tcompo}
\boldsymbol{T}_{13} = \dfrac{\partial u}{\partial y}, \quad \boldsymbol{T}_{23} = - \dfrac{\partial u}{\partial x},
\end{equation}
therefore, the equilibrium equation is automatically satisfied. Then, the compatibility condition is reduced to 
\begin{equation}\label{sc}
\dfrac{ \partial \boldsymbol{\epsilon}_{13} }{\partial y} - \dfrac{ \partial \boldsymbol{\epsilon}_{23} }{\partial x} =0. 
\end{equation}
The current study is devoted to analyzing a special subclass of the above general class of models 
In the current study, we consider a material model in which the linearized strain is a nonlinear function of stress via
\begin{equation}\label{model_2}
\boldsymbol{\epsilon} := \dfrac{\boldsymbol{T}}{2\mu \left( 1 + \beta^{\alpha} \, |\boldsymbol{T}|^{\alpha}     \right)^{1/\alpha}}. 
\end{equation}
$\beta \geq 0$ and $\alpha >0$ are the modeling parameters, and $\mu>0$ is the shear modulus. A similar constitutive class has been analyzed in previous studies \cite{mallikarjunaiah2015direct,vasilyeva2024generalized,kulvait2013,gou2023MMS,HCY_SMM_MMS2022,gou2023computational,yoon2022finite,yoon2021quasi}. Then, substituting \eqref{Tcompo} into \eqref{model_2} and then using the components of the strains in the compatibility condition \eqref{sc} yield a second-order quasi-linear elliptic partial differential equation.  Then, to study the crack-tip fields in an isotropic, homogeneous,  strain-limiting solid under anti-pane shear loading (or Model-III), we need to solve the following boundary value problem: 

To find the Airy stress potential funciton $u(\boldsymbol{x}) \colon \Omega \to \mathbb{R}$ that satisfies the following boundary value problem:
\begin{subequations}
	\begin{eqnarray}
		-\nabla \cdot \Big(\mathscr{G}(x,|\nabla u|)\nabla u\Big)&=&f \quad ~~~~~\text { in }~ \Omega, \label{contmodel}\\
		u &=& g \quad ~~~~ \text { on }~ \Gamma, \label{dbdry}	
	\end{eqnarray}
\end{subequations}
for fixed $\mu > 0$, $\alpha > 0$, and the function $\mathscr{G}(x, |\nabla u|)$ is defined as:
\begin{equation}
\mathscr{G}(x, |\nabla u|):=\frac{1}{ 2 \mu ( 1+\beta^\alpha |\nabla u |^\alpha )^{1/ \alpha}}.  \label{defH}
\end{equation}
We will later specify the ranges for $\alpha$ and $\beta$. Additionally, we assume that there exist positive constants $C_{B_1}$ and $C_{B_2}$ such that, for every $s\in \mathbb{R}_{+}\cup \{0\}$:
\begin{eqnarray}
|\mathscr{G}(s)|\leq C_{B_1} \quad \text{and}
\quad |\bar{\mathscr{G}}_s(\xi)|\leq C_{B_2}. \label{hbdd}
\end{eqnarray}
One can obtain an estimate based on the Mean Value Property 
\begin{eqnarray}
\mathscr{G}(\zeta)-\mathscr{G}(s)~=~(\zeta-s)\int_0^1\mathscr{G}_s(t\zeta+(1-t)s)\,dt:= (\zeta-s)\bar{\mathscr{G}}_s(\zeta).\nonumber
\end{eqnarray}
Thus, we obtain
\begin{equation*}
\big|\mathscr{G}(\zeta)-\mathscr{G}(s)\big|~\leq~C_{B_2}|\zeta-s|.    
\end{equation*}
This property will be used extensively in the analysis, even though it is not explicitly retained in all steps. 

\subsection{Background and Motivation.}
 Here, we briefly review existing research on quasi-linear problems and relevant sources (cf., books \cite{Amann1993, bernardi2003, brenner2008, ciar78, Ciarlet1991, lasis2003, riviere2008discontinuous} and articles \cite{Ainsworthkay1999, Ainsworthkay2000, andre1996, babuska1987, Castillo2000, Feistauer1993, Hlava1994, houston2005, Douglas1975, Perugia2001, Poussin1994, Riviere2001}). Finite element methods (FEM) have garnered considerable attention in both science and engineering due to their ability to handle complex geometries and domains. One of the FEM techniques, the discontinuous Galerkin (DG) method, is particularly effective in solving partial differential equations with discontinuous coefficients, as it approximates and discretizes boundary value problems while controlling the effects of discontinuities through penalty parameters.

Arnold, Douglas, and Wheeler introduced the symmetric interior penalty discontinuous Galerkin (SIPG) method for elliptic and parabolic problems, based on Nitsche's symmetric formulation \cite{Arnold1982, Douglasjr1976, Wheeler1978}. Although SIPG techniques are adjoint-consistent, their stabilizing parameters depend on problem-specific coefficient bounds and constants from inverse inequalities. Oden, Babuska, and Baumann later developed a nonsymmetric interior penalty discontinuous Galerkin (NIPG) technique, offering unconditional stability for the penalty parameter choice \cite{odenbabuska1988}. NIPG methods, explored by Riviere, Wheeler, and Girault \cite{Riviere2001} and Houston, Schwab, and Siili \cite{houstonschwab2002}, have drawn interest for their ability to solve a variety of partial differential equations.

For linear self-adjoint elliptic equations, optimal a priori error estimates using the broken $H^1$-norm have been established for SIPG and NIPG methods \cite{arnoldbrezzi2002, Riviere2001}. However, there are few studies on DG approximations for nonlinear elliptic problems, with exceptions such as \cite{gudi2007discontinuous, gudi2008hp, houston2005}. In \cite{houston2005}, a family of DG methods is applied to quasi-linear elliptic problems involving Lipschitz-continuous, strongly monotone differential operators.

Numerous DG formulations for elliptic problems exist \cite{arnoldbrezzi2002, gudi2008hp, gudi2007discontinuous, houston2005, riviere2008discontinuous}. The local discontinuous Galerkin (LDG) method, initially designed for first-order hyperbolic problems, has been extended to elliptic problems through mixed DG formulations \cite{Castillo2000}. Studies like \cite{Bustinza2004} have used the LDG method for quasilinear elliptic problems with mixed boundary conditions and uniformly monotone nonlinearities, demonstrating the existence of approximate solutions and providing a priori error estimates. In \cite{houston2005}, a one-parameter DG family was used to address quasilinear problems with monotone operators, yielding optimal error estimates in $h$ but suboptimal ones in $p$ for the broken Sobolev $H^1$-norm.

The problem considered in this work frequently arises in material science, engineering, and geophysics to forecast and examine fracture behavior under out-of-plane shear forces. The static anti-plane shear crack model is the primary approach for studying shear fracture in materials.  This model assists in failure prediction and provides information for the design of more robust building components and materials by utilizing techniques such as stress intensity factors and fracture toughness (see, for example, \cite{irwin1957analysis, lawn1993fracture, anderson2005fracture, scholz2019mechanics}).

\subsection{Our Contribution.}
Using a DG finite element approach, our study investigates the static anti-plane shear crack model \eqref{contmodel}-\eqref{dbdry}. We discretize the model problem using the DG method, allowing for hanging nodes in the domain discretization. Invoking Ritz representation theory, we establish the existence of a unique discrete solution and derive an optimal a priori error estimate in the broken Sobolev space. The linearized form of the model is handled using the Abuin-Nitsche procedure to yield the $L^2$-norm error estimate, which shows the optimal order convergence. Consequently, the optimal order a priori error estimate for the energy norm is also established. The numerical results confirm the theoretical predictions and demonstrate the singular behavior of the stresses at the crack. However, crack-tip strains' growth is slower than the linearized model.

\noindent
The entire paper is arranged as follows: Section 2 introduces the preliminary function spaces, Section 3 presents the DG finite element formulation, and Section 4 contains the error analysis. Section 5 concludes with numerical experiments verifying our analysis and a summary of the findings.

\section{Precursory}
A few essential assets that are crucial to our upcoming study are included in this section. In this work, we employ the usual notation from Lebesgue and Sobolev space theory \cite{adams1975}. The Sobolev space of order $(k,p)$, abbreviated by $W^k_p(\Omega)$, is defined by  
\begin{equation*}
 W^k_p(\Omega)~:=\big\{\phi\in L^p(\Omega):\quad D^\iota \phi \in L^p(\Omega), |\iota|\leq k\big\}.
\end{equation*}
for $1\leq p \leq \infty$, equipped with inner product and the norm
\begin{equation*}
(\phi, \psi)_{k,p,\Omega}~=~  \sum_{\iota \leq k}\int_\Omega  D^\iota \phi \cdot D^\iota \psi\, dx,   
\end{equation*}
and 
\begin{equation*}
\|\phi\|_{W^k_p(\Omega)}~=~ \Big(\sum_{\iota \leq k}\int_\Omega  |D^\iota \phi|^p\, dx\Big)^{1/p},   
\end{equation*}
respectively. For $p=\infty$, the norm is given by
\begin{equation*}
\|\phi\|_{W^k_\infty(\Omega)}~=~ ess\sup_{\iota \leq k}\|D^\iota \phi\|_{L^\infty(\Omega)}.   
\end{equation*}
Moreover, we write $W^k_2(\Omega):=H^k(\Omega)$, for $p=2$, and the norm is denoted by $\|\cdot\|_{H^k(\Omega)}$. Let us define the function space $\mathcal{S}$ and the test function space $T$ for our problem, by
\begin{equation*}
\mathcal{S}~:=~\big\{ \phi \in H^2(\Omega)\cap W^{1,\infty}(\Omega):~ \phi|_{\Gamma}=g\big\},
\end{equation*}
 and  $T~:=~H^1_0(\Omega),$ respectively. Hence, the variational formulation for the problem \eqref{contmodel}-\eqref{dbdry} is read as: Find  $u\in\mathcal{S}$ such that 
\begin{equation}
\mathcal{A}(u; u, \phi)~=~ \mathcal{L}(\phi), \quad \text{for all}\quad \phi\in T. \label{abstLform2.9}
\end{equation}
However, the semilinear $\mathcal{A}$ and the linear forms are given by
\begin{eqnarray}
 \mathcal{A}(z; u, \phi)&=:&\int_{\Omega}\mathscr{G}(x, |\nabla z|)\nabla u\cdot \nabla \phi\,dx, \quad \forall \phi\in T, \label{variationalform2.8} 
\end{eqnarray}
and 
\begin{equation}
 \mathcal{L}(\phi)=:~\int_{\Omega}f \phi\,dx,\quad \forall \phi\in T,\label{Lform2.9}
\end{equation}
respectively. Additionally, we consider the function  $f \in L^2(\Omega),$ and $\; g \in H^{3/2} (\Gamma)$, respectively. Then, the problem   \eqref{abstLform2.9} has a unique weak solution $u \in \mathcal{S}$. Then, there is a positive constant $C_{R_1}$ such that the following regularity results holds,
\begin{equation}
  \|u\|_{H^2(\Omega)}  \leq~ C_{R_1} \big( \|f\|_{L^2(\Omega)}+\|g\|_{H^{\frac{3}{2}}(\Gamma)}\big).\label{ctsubd}
\end{equation}
For the estimation of the nonlinear term, we define $v \longmapsto \mathcal{J}(x,v),\,$ $x\in \bar{\Omega},$ such that
\begin{equation*}
\mathcal{J}(x, v)=: \frac{v}{2\,\mu \,( 1+\beta^\alpha |v|^\alpha )^{1/ \alpha}}, \quad \forall \, v \in (W^{1,\infty}(\Omega))^d, 
\end{equation*}
One may observed that $\mathcal{J}(x, v)=\mathscr{G}(x, |v|)v$. Throughout the paper, we write $\mathscr{G}(|v|)$ instead of $\mathscr{G}(x, |v|)$ for simplicity.

For all $x\in \bar{\Omega}$, $v_1,\,v_2 \in (W^{1,\infty}(\Omega))^d$ with $v_1\geq v_2$, and $t\in [0,1]$, then, from the mean value theorem, we have 
\begin{eqnarray}
\mathcal{J}(x, v_1)-\mathcal{J}(x, v_2)&=&\int_0^1\frac{1}{dt}\mathcal{J}(tv_1+(1-t)v_2)\,dt\nonumber\\
&=&\frac{1}{2\,\mu} \,\int_0^1\frac{1}{dt}\Big(\frac{(tv_1+(1-t)v_2)}{( 1+\beta^\alpha |tv_1+(1-t)v_2|^\alpha )^{1/ \alpha}}\Big)\,dt\nonumber\\
&=&\frac{1}{2\,\mu} \int_0^1\frac{(v_1-v_2)}{( 1+\beta^\alpha |tv_1+(1-t)v_2|^\alpha )^{1/ \alpha}}\,dt,  \nonumber
\end{eqnarray}
and hence,
\begin{eqnarray}
\frac{1}{2\,\mu\,( 1+\beta^\alpha (|v_1|+|v_2|)^\alpha )^{1/ \alpha}}\times(v_1-v_2)\leq~ \mathcal{J}(x, v_1)-\mathcal{J}(x, v_2)
\leq~\frac{1}{2\,\mu}\times(v_1-v_2). \label{2.5inq} 
\end{eqnarray}
Since $|tv_1+(1-t)v_2|\leq |v_1|+|v_2|$ and $( 1+\beta^\alpha (|v_1|+|v_2|)^\alpha )^{1/ \alpha}\geq 1$. 
Thus, we have 
\begin{equation}
C_{M_1} (v_1-v_2)\leq~\mathcal{J}(x, v_1)-\mathcal{J}(x, v_2)\leq~ C_{M_2} (v_1-v_2)
\end{equation}
with the precise constants $C_{M_1}=1/[2 \mu( 1+\beta^\alpha (|v_1|+|v_2|)^\alpha )^{1/ \alpha}]$ and $C_{M_2}=1/2\mu$. Then, the functional $\mathcal{J}(\cdot, \cdot)$ satisfies the following inequalities 
\begin{equation}\label{lpsctsinq}
|\mathcal{J}(x, v_1)-\mathcal{J}(x, v_2)|\leq C_{M_3} |v_1-v_2|, 
\end{equation}
and 
\begin{equation}\label{monoinq}
(\mathcal{J}(x, v_1)-\mathcal{J}(x, v_2), v_1-v_2)\geq C_{M_4}  |v_1-v_2|^2,
\end{equation}
where $C_{M_3}$ and $C_{M_4}$ are the positive constant.
The inequality  \eqref{lpsctsinq} refers to the Lipschtiz continuity of $\mathcal{J}(\cdot, \cdot)$, while the inequality \eqref{monoinq} is known as the monotonicity.

The following section concerns the discontinuous Galerking setting of the problem \eqref{contmodel}--\eqref{dbdry}.
\section{Discretization}
This section is devoted to the discontinuous Galerkin finite elements approximation of the problem \eqref{contmodel}--\eqref{dbdry}. 
Let $\mathcal{T}_{h}$ represent the partitions of $\bar{\Omega}$ into the disjoint open elements $\tau_i$, where $i\in [1:N_h]$, and $\bar{\Omega}=\bigcup_{i=1}^{N_h} \bar{\tau}_i$ for all $i$. Only $d$-simplex or $d$-parallelepiped components are permitted. Keeping in mind that $\mathcal{T}_h$ may be regular or 1-irregular, that is, each face of $\tau_i \in \mathcal{T}_h$ permits a single hanging node, which is the barycentre of the face. It is likely because $\mathcal{T}_h$ is a shape-regular family of partitions (see. \cite[Braess, 1997, pp. 61, 113, and Remark 2.2, pages. 114]{braess1997}). Moreover, we make the assumption that $\mathtt{F}_{i}(\hat{\tau}_i)=\tau_i,\; \forall \tau_i \in \mathcal{T}_h$ is an affine image of a fixed component $\hat{\tau}_i$ for each $\tau_i \in \mathcal{T}_h$. In this case, $\hat{\tau}_i$ indicates either the open unit simplex or the open unit hypercube in $\mathbb{R}^d$.
Further, the collection of all inner and boundary edges/faces are identified by $\mathscr{E}_{int,h}$ and $\mathscr{E}_{bd,h}$, respectively. 
 It follows that the collection of all edges/faces $\mathscr{E}_{h}$ is such that $\mathscr{E}_{h}=\mathscr{E}_{int, h} \cup \mathcal{E}_{bd,h}$. We assume that any edge/face $e_i\in \mathscr{E}_h$ that precisely lies on the boundary $\Gamma$ will be included in the boundary $\Gamma$ decomposition. For each element $\tau_i$, let $h_{\tau_i}=diam(\tau_i)$ represent its diameter, and $h_{e_i}$ represent the length (diameter) of the edge (face) $e_i$ according to $d=2\, \text{or}\, d\geq 3$. Take for granted that $h=\max\{h_{\tau_i}: \tau_i \in \mathcal{T}_h\}$. We set $h_{\tau_i}=h_i$ for ease.

First, we present the finite element space and then go on to additional analysis. The set of polynomials with a total degree of $\mathtt{n}$ (a non-negative integer) on $\hat{\tau}$ is represented by $\mathcal{P}_{\mathtt{n}}(\hat{\tau})$. The set of all tensor-product polynomials on $\hat{\tau}$ of degree $\mathtt{n}$ in each coordinate direction is denoted by $\mathcal{Q}_{\mathtt{n}}(\hat{\tau})$ in the unit hypercube case. Additionally, for each $\tau_i\in \mathcal{T}_h$, let $n_i$ and $m_i$ be two non-negative numbers that define the local Sobolev index and the local polynomial degree, respectively. The vectors, using $n_i,\; m_i$ and $\mathtt{F}_i$, are now defined as
$$\mathtt{n}:=\{n_i: \hat{\tau}_i \in \mathcal{T}_h\}, \quad \mathtt{m}:=\{m_i:\hat{\tau}_i \in \mathcal{T}_h\} \quad \text{and} \quad \mathbf{F}:=\{\mathtt{F}_i:~ \hat{\tau}_i \in \mathcal{T}_h\},$$  respectively. Further, let $\mathcal{R}_{n_i}$ represent a collection of polynomials, such as $\mathcal{P}_{n_i}(\hat{\tau}_i)$ or $\mathcal{Q}_{n_i}(\hat{\tau}_i)$, then the discontinuous finite element space is defined by 
\begin{equation*}
\mathcal{S}^{\mathtt{n}}_h(\mathcal{T}_h):=\{\phi \in L^2(\Omega): ~ \phi|_{\tau_i}*\mathtt{F}_{i} \in \mathcal{R}_{n_i}(\hat{\tau}_i), \quad \forall \tau_i \in \mathcal{T}_h\},
\end{equation*}
and $\mathscr{B}^{\mathtt{m}}_\mathtt{n}(\Omega; \mathcal{T}_h)$ identifies the broken Sobolev space with composite index $\mathtt{m}$, as
\begin{equation*}
\mathscr{B}^{\mathtt{m}}_\mathtt{n}(\Omega; \mathcal{T}_h):=\big\{\phi \in L^2(\Omega): ~ \phi|_{\tau_i} \in H^{m_i}(\hat{\tau_i}), \quad \forall \tau_i \in \mathcal{T}_h\big\} 
\end{equation*}
equipped with the broken Sobolev norm and the semi-norm as 
\begin{equation*}
\|\phi\|^2_{\mathtt{m}, \mathcal{T}_h}:=\sum_{\tau \in \mathcal{T}_h} \|\phi\|^2_{H^{m_i}(\tau_i)} \quad \text{and} \quad |\phi|^2_{\mathtt{m}, \mathcal{T}_h}:=\sum_{\tau_i \in \mathcal{T}_h} |\phi|^2_{H^{m_i}(\tau_i)},
\end{equation*}
respectively. For error analysis, we define the discontinuous finite element space 
\begin{equation*}
\accentset{\ast}{\mathcal{S}}^n_h(\mathcal{T}_h)= \{\phi \in L^2(\Omega): ~ \phi|_{\tau_i}*\mathtt{F}_{i} \in \mathcal{R}_{n_i}(\hat{\tau}_i), \quad \forall \tau_i \in \mathcal{T}_h\}.
\end{equation*}
For all $\tau_i \in \mathcal{T}_h$, we adopt the notation $\mathscr{B}^{m}_\mathtt{n}(\Omega; \mathcal{T}_h)$ in case of $m_i= m$, while the norm and the semi-norm will be identified by $\|\cdot\|_{m, \mathcal{T}_h}$ and $|\cdot|_{m, \mathcal{T}_h}$, respectively.
 \begin{figure}
   \begin{minipage}[h]{0.45\linewidth}
\begin{tikzpicture}
 \filldraw[fill=brown!100,draw=black, thick] {(0,0)node[below]{A}--(6,0)node[below]{B}--(6,6)node[above]{C}--(0,6)node[above]{D}--(0,0)};
     \filldraw[fill=gray!100, draw=black, thick] {(0,0)--(1,0)--(1,1)--(0,0)};
     \filldraw[fill=gray!80, draw=black, thick] {(1,1)--(2,1)--(2,2)--(1,1)};
     \filldraw[fill=gray!100, draw=black, thick] {(2,2)--(3,2)--(3,3)--(2,2)};
     \filldraw[fill=gray!80, draw=black, thick] {(3,3)--(6,3)--(6,6)--(3,3)};
     \filldraw[fill=gray!100, draw=black, thick] {(3,3)--(6,3)--(6,0)--(3,3)};
      \filldraw[fill=gray!80, draw=black, thick] {(3,3)--(3,4)--(2,4)--(3,3)};
      \filldraw[fill=gray!100, draw=black, thick] {(2,4)--(2,5)--(1,5)--(2,4)};
      \filldraw[fill=gray!80, draw=black, thick] {(1,5)--(1,6)--(0,6)--(1,5)};
       \filldraw[fill=gray!100, draw=black, thick] {(1,0)--(3,0)--(3,2)--(1,0)};
        \filldraw[fill=gray!80, draw=black, thick] {(3,0)--(3,3)--(4.5,1.5)--(3,0)};
     \filldraw[fill=gray!80, draw=black, thick] {(1,6)--(3,6)--(3,4)--(1,6)};
        \filldraw[fill=gray!100, draw=black, thick] {(3,6)--(3,3)--(4.5,4.5)--(3,6)};
     \draw [thick](0,0)--(6,6);
     \draw [thick](0,1)--(2,1);
    \draw [thick](0,2)--(3,2);
     \draw [thick](0,4)--(3,4);
     \draw [thick](0,5)--(2,5);
    \draw [thick](0,3)--(6,3);
    \draw[thick] {(6,0)--(0,6)};
    \draw[thick] {(3,0)--(3,6)};
    \draw[thick] {(1,0)--(1,6)};
     \draw[thick] {(2,1)--(2,5)};
       \draw[thick] {(0,4)--(2,2)};
       \draw[thick] {(2.5,2.5)--(0,5)};  
       \draw[thick] {(0,1)--(2.5,3.5)};
       \draw[thick] {(0,2)--(2,4)};
       \draw[thick] {(0,3)--(1.5,4.5)};
       \draw[thick] {(0,4)--(1,5)};
        \draw[thick] {(0,5)--(0.5,5.5)}; 
       \draw[thick] {(0,3)--(1.5,1.5)};
       \draw[thick] {(0,2)--(1,1)};
        \draw[thick] {(0,1)--(0.5,0.5)};
       \draw[thick] {(1,6)--(3,4)};
         \draw[thick] {(3,6)--(4.5,4.5)};
    \draw[thick] {(1,0)--(3,2)};
       \draw[thick] {(3,0)--(4.5,1.5)};
    \draw (2.2,1.2) node[above, thick]{\Large $e_{i}$};
    \draw (2.3,4.3) node[very thick, black]{\Large $\tau_{i}$};
    \draw (5.2,3.9) node[very thick, black]{Hanging};
    \draw (5.2,3.4) node[very thick, black]{Nodes};
    \foreach \x/\y in {4.3/3.9}\filldraw[red] (\x,\y) circle (0.9mm);
    \draw (5.2,2.5) node[very thick, black]{Corner};
    \draw (5.2,2.0) node[very thick, black]{Nodes};
    \foreach \x/\y in {4.3/2.5}\filldraw[thick] (\x,\y) circle (0.9mm); 
       \foreach \x/\y in {0/0,1/0,3/0,6/0,0.5/0.5,
       0.5/1.5,1.5/1.5, 0.5/2.5,1.5/2.5,2.5/2.5,
       0.5/3.5,1.5/3.5, 0.5/4.5,   0/1,0/2,0/3,0/4,0/5,0/6,   1/1,1/2,1/3,1/4,1/5,1/6,
        2/1,2/2,2/3,2/4,2/5,  3/2,3/3,3/4,3/6,   
        6/6}\fill (\x,\y) circle (0.9mm);
    \foreach \x/\y in {0.5/5.5, 2/5,3/4,4.5/1.5, 4.5/4.5,
    1.5/4.5,2.5/3.5,2.5/2.5,0.5/0.5,1.5/1.5,2/1,3/2}\filldraw[red] (\x,\y) circle (0.9mm); 
    \draw (2.8,-1.0) node[thick]{(a) Hanging node in trianglulation};
    \end{tikzpicture}
  \end{minipage}
  \vspace{0.00mm}
\begin{minipage}[h]{0.3\linewidth} %
\centering
 \begin{tikzpicture}
    \filldraw[fill=brown!100, draw=black, thick]  {(0,0)node[below]{E} --(8,0)node[below]{F}--(8,6)node[above]{G}--(0,6)node[above]{H}--(0,0)};
     \filldraw[fill=gray!80,draw=black, thick] {(2,0)--(4,0)--(4,2)--(2,2)--(2,0)};
      \filldraw[fill=gray!80,draw=black, thick] {(4,2)--(6,2)--(6,4)--(4,4)--(4,2)};
      \filldraw[fill=gray!100,draw=black, thick] {(6,4)--(8,4)--(8,6)--(6,6)--(6,4)};
    \filldraw[fill=gray!100,draw=black, thick] {(2,2)--(4,2)--(4,4)--(2,4)--(2,2)};
     \filldraw[fill=gray!80,draw=black, thick] {(0,2)--(2,2)--(2,4)--(0,4)--(0,2)};
      \filldraw[fill=gray!100,draw=black, thick] {(2,4)--(0,4)--(0,6)--(2,6)--(2,4)};
    \draw[ thick] {(0,4)--(8,4)};
    \draw[ thick] {(3,4)--(3,6)};
    \draw[ thick] {(0,1)--(2,1)};
    \draw[ thick] {(1,0)--(1,2)};
    \draw[thick] {(4,2)--(4,4)};
    \draw[thick] {(2,5)--(6,5)};
    \draw[ thick] {(0,4)--(4,4)};
    \draw[ thick] {(5,4)--(5,6)};
    \draw[thick] {(0, 2)--(8, 2)};
    \draw[thick] {(6, 0)--(6, 6)};
    \draw[thick] {(4,0)--(4,6)};
    \draw[thick] {(2,0)--(2,6)};
    \draw (1.7,2.9) node[very thick, black]{\Large $e_{i}$};
    \draw (3,3)  node[very thick, black]{\Large $\tau_{i}$};
    \draw (7.15,1.2) node[very thick, black]{Corner};
    \draw (7.15,0.7) node[very thick, black]{Nodes};
    \draw (7.15,3.1) node[very thick, black]{Hanging};
    \draw (7.12,2.6) node[very thick, black]{Nodes};
    \foreach \x/\y in {0/0,0/1, 0/6, 8/0,0/4,4/0,4/4,0/2,2/0,2/4,8/4,4/2,3/4,3/5,3/6,2/5,2/6,5/5,5/6,5/4,4/5,4/6,6/0,6/2,6/4,6/6,1/1,1/2,2/2,1/0}\fill (\x,\y) circle (0.9mm);
   \foreach \x/\y in { 6.3/1.2}\fill (\x,\y) circle (0.9mm);
    \foreach \x/\y in {6.25/3.1}\filldraw[red] (\x,\y) circle (0.9mm); 
    \foreach \x/\y in {2/1,3/4,1/2,2/5,5/4,6/5}\filldraw[red] (\x,\y) circle (0.9mm);   
    \draw (4.06,-1.0) node[thick]{(b) Hanging node in Qudrilateral};
    \end{tikzpicture}
\end{minipage}
\caption{Red color dotes denote the hanging nodes in both families of subdivisions $\mathcal{T}_h$.}
\label{mshfig}
\end{figure}
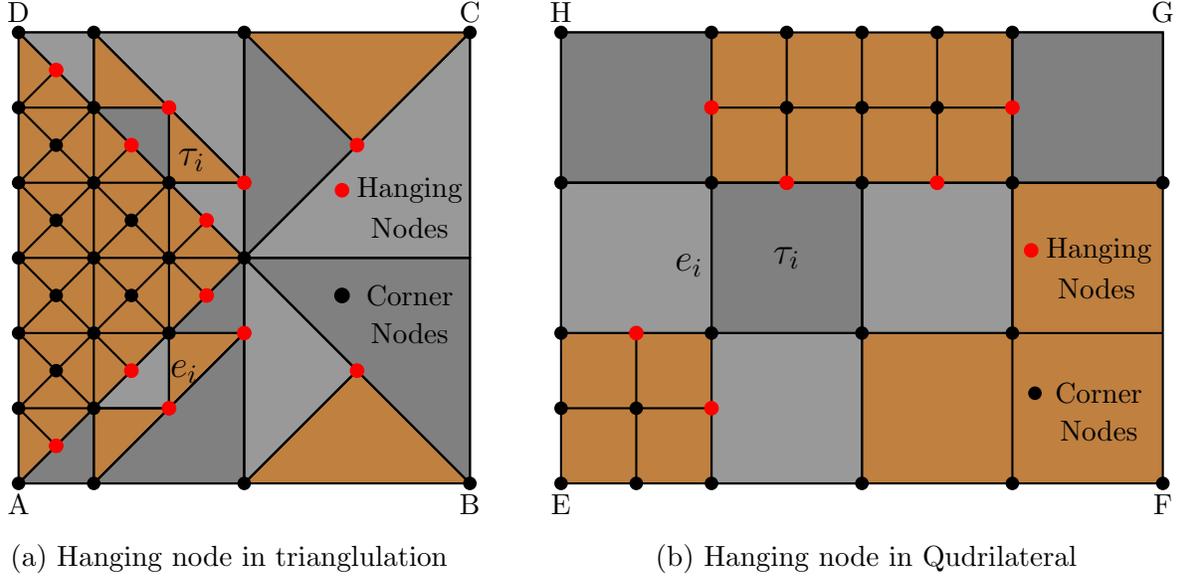
Note that the discretization $\mathcal{T}_h$ and the discontinuous finite element space $\mathcal{S}^\mathtt{n}_h$ must satisfy the following properties:
\begin{itemize} \label{Thp1p2}
    \item[($A_1$)] The bounded local variation property is satisfied by the family of subdivisions of $\mathcal{T}_h$. That is, there exists a positive constant $\delta$, independent of  $h_i$ and $h_j$, such that
    $$1/\delta\leq h_i/h_j\leq \delta$$ for any $\tau_i,\, \tau_j \in \mathcal{T}_h$ if $\partial\tau_i \cap \partial \tau_j$ has positive measure. 
    \item[($A_2$)] For a  given  pair of $\tau_i,\, \tau_j$ in $\mathcal{T}_h$ that share a $(d-1)$-dimensional face and the polynomial degree vector $\mathtt{n}$, then there is a positive constant $\varrho$ that is independent of both $n_i$ and $n_j$,  is assumed to satisfy a constrained local variation property  $$1/\varrho \leq n_i/n_j \leq \varrho$$ with $n_i \geq 1$, for each $\tau_i \in \mathcal{T}_h$.
\end{itemize}
Additionally, assume that $n_{ij}$ denote the degree of the polynomial that is confined to $e_{ij}$, namely, a common edge of $\tau_i$ and $\tau_j$, which means that $e_{ij}=\partial \tau_i \cap \partial \tau_j$, and defined as $n_{ij}=(n_i+n_j)/2$. While $n_i$ is the degree of the polynomial on the edge $e_i$ if $e_i=\partial \tau_i \cap \Gamma$.

We now define the jump and mean of $\phi\in \mathscr{B}^1_\mathtt{n}(\Omega; \mathcal{T}_h)$ across an element edge/face $e_{ij}\in \mathscr{E}_{int,h}$, which is a common edge/face of $\tau_i$ and $\tau_j$ with $i>j$ as
$$\sjump{\phi}=:\phi|_{e_{ij}\cap \partial\tau_i}-\phi|_{e_{ij}\cap \partial\tau_j}$$ and $$\smean{\phi}=:\frac{1}{2}\{\phi|_{e_{ij}\cap\partial\tau_i}+\phi|_{e_{ij}\cap \partial\tau_j}\},$$
respectively. If $\e_i \in \partial \tau_i \cap \Gamma$. that is, $e_i \subset \mathscr{E}_{bd,h}\cup \mathscr{E}_{d,h}$, then the mean and the jump on $e_i$ are defined as 
$$\sjump{\phi}|_{e_{i}}=:\phi|_{\partial \tau_{i}\cap \Gamma} \quad \text{and} \quad \smean{\phi}|_{e_{i}}=:\phi|_{\partial\tau_{i}\cap \Gamma},$$
respectively. Thus, DG-approximation of \eqref{contmodel}-\eqref{dbdry} is read as: To find $u_h^{dg}\in \mathcal{S}^{\mathtt{n}}_h(\mathcal{T}_h)$ such that 
\begin{equation}
\mathcal{A}_h(u_h^{dg};u_h^{dg}, \phi_h)= \mathcal{L}_h(\phi_h),\quad \forall \phi_h \in \mathcal{S}^{\mathtt{n}}_h(\mathcal{T}_h).\label{dgformulation}
\end{equation}
Here, the semilinear form  $\mathcal{A}_h(\cdot,\cdot,\cdot)$ is given by
\begin{eqnarray}
\mathcal{A}_h(u_h^{dg}; u_h^{dg},\phi_h)&=:& \sum_{\tau_i \in \mathcal{T}_h}\int_{\tau_i}\mathscr{G}(|\nabla u_h^{dg}|)\nabla u_h^{dg}\cdot \nabla \phi_h\,dx-\sum_{e_i \in \mathscr{E}_{int, h}} \int_{e_i} \smean{\mathscr{G}(|\nabla u_h^{dg}|)\nabla u_h^{dg} \cdot \bar{\bf n}} \sjump{\phi_h}\,ds\nonumber\\
 &&- \sum_{{e_i} \in \mathscr{E}_{int, h}} \int_{e_i} \smean{\mathscr{G}(|\nabla u_h^{dg}|)\nabla \phi_h \cdot \bar{\bf n}}\sjump{u_h^{dg}}ds+ \sum_{{e_i} \in \mathscr{E}_{h}} \int_{e_i} \sigma_i\frac{n_i^2}{|e_i|^\gamma}\sjump{u_h^{dg}} \sjump{\phi_h}\,ds,\nonumber\\ \label{semilnrform3.1}
\end{eqnarray}
and 
\begin{equation}
 \mathcal{L}_h(\phi_h)=:~\sum_{\tau_i \in \mathcal{T}_h}\int_{\tau_i}f \phi_h\,dx +\sum_{e_i \in \mathscr{E}_{bd,h}} \int_{e_i} \sigma_i\frac{n_i^2}{|e_i|^\gamma} \, g \phi_h\,ds,\nonumber \label{lnrform3.2}
\end{equation}
in which $\bar{\bf n} $ is defined as the unit outward normal vector such that $\bar{\bf n} =( \nu_1, \nu_2, \ldots, \nu_d )$, and $ \sigma|_{e_i} :=\sigma_i,$ $n_i$ and $\gamma$ are the positive constants and independent of $h$.

From the equivalence relation between the model problem \eqref{contmodel}-\eqref{dbdry} and the DG-formulation \eqref{dgformulation}, the solution $u\in H^2(\Omega)$ of the problem \eqref{contmodel}-\eqref{dbdry} satisfies the Eq. \eqref{dgformulation}, i.e., $\mathcal{A}_h(u; u, \phi_h)=\mathcal{L}_h(\phi_h)$, since $u\in H^1(\Omega)\cap H^2(\Omega,\mathcal{T}_h)$. Moreover, by recalling the consistency with $\phi_h \in \mathcal{S}^{\mathtt{n}}_h(\mathcal{T}_h) \subset H^2(\Omega; \mathcal{T}_h)$, one can easily obtain the following condition 
\begin{equation}
\mathcal{A}(u; u, \phi_h)~=~\mathcal{A}_h(u_h^{dg}; u_h^{dg},\phi_h) \quad \text{for all}\;\;  \phi_h \in \mathcal{S}^{\mathtt{n}}_h(\mathcal{T}_h). \label{orthogotype}
\end{equation}
Here $u_h^{dg}$ is the solution of \eqref{dgformulation}.

We are going to accumulate some fundamental approximation properties (see \cite{babuska1987}), trace inequalities (see \cite[Lemma 2.1]{Riviere2001}, \& \cite[Appendix A.2]{oden1998}), and inverse inequalities (see \cite[Theorem 6.1]{bernardi2003} \& \cite[pp. 6]{lasis2003}) that will be helpful in our upcoming analysis. 
We refer \cite{babuska1987, bernardi2003, lasis2003, Riviere2001, oden1998} for a description of the proof of that. The proofs of the following lemmas are skipped over here. Moreover, for $ \tau_i \in \mathcal{T}_h$, with the degree of the polynomial $n_i$ on $\tau_i$, define $\mathcal{I}_hu$ as $$\mathcal{I}_hu|_{\tau_i}=u_{n_i}^{h_i},\quad \text{for} \quad u|_{\tau_i} \in H^{m_i}(\tau_i),\, m_i \geq 2.$$ 
\begin{lemma}[\cite{babuska1987}] \label{apprxlm3.2}
For $\phi|_{\tau_i} \in H^{m_i}(\tau_i), \quad \tau_i \in \mathcal{T}_h$ with $m_i \geq 0$, then there are two positive constants $C_{ap, 1}$ and $C_{ap,2}$ independent of the mesh parameters and a sequence $\mathcal{I}_h\phi|_{\tau_i}=\phi^{h_i}_{n_i} \in \mathcal{R}_{n_i}(\tau_i)$ with $n_i=1,\,2,\, \ldots,$ such that 
\begin{enumerate}
\item[(i)] for any $0\leq j \leq m_i$,
\begin{equation}\label{appforelmnts}
\|\phi-\mathcal{I}_h\phi\|_{H^{j}(\tau_i)}\leq~C_{ap, 1} \frac{h_i^{{q_i}-j}}{n_i^{{m_i}-j}}\|\phi\|_{H^{m_i}(\tau_i)},
\end{equation}
\item[(ii)] for any $j+\frac{1}{2} < m_i$,
\begin{equation}\label{appforedge}
\|\phi-\mathcal{I}_h\phi\|_{H^{j}(e_i)}\leq~C_{ap,2} \frac{h_i^{{q_i}-j-\frac{1}{2}}}{n_i^{{m_i}-j-\frac{1}{2}}}\|\phi\|_{H^{m_i}(\tau_i)},
\end{equation}
\end{enumerate}
where $1\leq q_i\leq \min\{m_i, n_i+1\}$ with $n_i \geq 1$.
\end{lemma}
\begin{lemma}[\cite{Riviere2001}]\label{t2lemma4.3}
 Let $\phi_h\in \mathcal{R}_{n_i}(\tau_i)$ and $r=0 \;\text{or}\; 1$, then there is a positive constant $C_{tr,1},$ such that the following inequality
 \begin{equation}
\|\nabla^r \phi_h\|_{L^2(e_i)} \leq~C_{tr,1}\frac{n_i}{\sqrt{h_i}}\,\|\nabla^r \phi_h\|_{L^2(\tau_i)}, \label{tr1invh}
 \end{equation}
 holds.
\end{lemma}
\begin{lemma}[\cite{oden1998}]\label{tracelemma4.2}
Let $\phi \in H^{r+1}(\tau_i),$ $\tau_i \in \mathcal{T}_h$. Then, there is a positive constant $C_{tr,2},$ such that
\begin{equation}\label{2tr3.8}
 \|\phi\|^l_{W^r_l(e_i)} \leq~C_{tr,2}~\big\{h^{-1}_i \|\phi\|^l_{W^r_l(\tau_i)}+\|\phi\|^{l-1}_{W^r_{2l-2}(\tau_i)} \|\nabla^{r+1}\phi\|_{L^2(\tau_i)} \big\},  
\end{equation}
where $r=0,\,1$ and $l=2,\,4$.
\end{lemma}
\begin{lemma}[\cite{bernardi2003, lasis2003}]
Let $\phi_h\in \mathcal{R}_{n_i}(\tau_i)$ and $l \geq 1$, then there is a positive constant $C_{Inv},$ such that the following inequality
 \begin{equation}
|\phi_h|_{H^l(\tau_i)} \leq~C_{Inv}\frac{n_i^2}{h_i}\, | \phi_h|_{H^{l-1}(\tau_i)}, \label{tr1invh3.8}
 \end{equation}
 holds.
\end{lemma}
\subsection{{Existence of a unique discrete solution $u_h^{dg}$ in $\mathcal{S}^{\mathtt{n}}_h(\mathcal{T}_h)$}} Let us first define the inner product and norm on $\mathcal{S}^{\mathtt{n}}_h(\mathcal{T}_h)$, for all $\phi,\; \psi \in \mathcal{S}^{\mathtt{n}}_h(\mathcal{T}_h)$, as
\begin{equation*}
(\phi, \psi)_{\mathcal{E}}:=\sum_{\tau_i \in \mathcal{T}_h}\int_{\tau_i} \nabla \phi \cdot \nabla \psi\,dx 
+ \sum_{{e_i} \in \mathscr{E}_{h}} \int_{e_i} \sigma_i\frac{n_i^2}{|e_i|^\gamma}\sjump{u} \sjump{\phi}\,ds,
\end{equation*}
and
\begin{equation*}
\|\phi\|_{\mathcal{E}}:=\Big(\sum_{\tau \in \mathcal{T}_h}\int_\tau | \nabla \phi|^2\,dx 
+ \sum_{{e_i} \in \mathscr{E}_{h}} \int_{e_i} \sigma_i\frac{n_i^2}{|e_i|^\gamma}|\sjump{\phi}|^2\,ds\Big)^{\frac{1}{2}},
\end{equation*}
respectively.

Before embarking on the existence of a unique discrete discontinuous Galerkin solution using the theory of monotone operators (cf. \cite{necas1986}). We first prove the Lipschitz-continuity of the semi-linear form $\mathcal{A}_h$ in the following lemma.
\begin{lemma}\label{lpscntofbilinear}
Assume that inequalities \eqref{lpsctsinq}--\eqref{monoinq} and the assumptions $A_1-A_2$ (Sect. \ref{Thp1p2}) hold. For $u_1,\,u_2 \in \mathcal{S}^{\mathtt{n}}_h(\mathcal{T}_h)$, the semi-linear form $\mathcal{A}_h$ is Lipschitz-continuous in the first argument, in the sense that 
\begin{equation}
|\mathcal{A}_h(u_1; u_1, \phi)- \mathcal{A}_h(u_2; u_2, \phi)|\leq C_{31} \|u_1-u_2\|_{\mathcal{E}}\|\phi\|_{\mathcal{E}}, \quad \text{for all} \quad \phi \in \mathcal{S}^{\mathtt{n}}_h(\mathcal{T}_h),
\end{equation}
where $C_{31}=\sqrt{2} \times \,\sqrt{\Big(}\max\big\{1, C^2_{M_3}\big\}+\frac{C^2_{M_3}C_{tr,1}^2\,  (1+\varrho^2\, \delta)}{2\,\sigma_i\,|e_i|^{1-\gamma}}\Big)$, $i\in [1:N_h]$ and $\gamma \geq 1$.
\end{lemma}
\begin{proof}
For any $u_1,\,u_2\in \mathcal{S}^{\mathtt{n}}_h(\mathcal{T}_h)$, and the use of Eq. \eqref{semilnrform3.1} to have 
\begin{eqnarray}
&&\big|\mathcal{A}_h(u_1;u_1, \phi)-\mathcal{A}_h(u_2;u_2, \phi)\big|
\leq~ \sum_{\tau_i \in \mathcal{T}_h}\int_{\tau}\big|\mathscr{G}(|\nabla u_1|)\nabla u_1-\mathscr{G}(|\nabla u_2|)\nabla u_2\big|\big|\nabla \phi \big|\,dx\nonumber\\
&&~~~~+\sum_{e_i \in \mathscr{E}_{int, h}} \int_{e_i} \big|\smean{\mathscr{G}(|\nabla u_1|)\nabla u_1 \cdot \bar{\bf n}}-\smean{\mathscr{G}(|\nabla u_2|)\nabla u_2 \cdot \bar{\bf n}}\big| \big|\sjump{\phi}\big|\,ds\nonumber\\
&&~~~~+\sum_{e_i \in \mathscr{E}_{int, h}} \int_{e_i} \big|\smean{\mathscr{G}(|\nabla u_1|)\nabla \phi \cdot \bar{\bf n}}\sjump{u_1}-\smean{\mathscr{G}(|\nabla u_2|)\nabla \phi \cdot \bar{\bf n}}\sjump{u_2}\big|\,ds\nonumber\\
&&~~~~+ \sum_{e_i \in \mathscr{E}_{h}} \int_{e_i} \sigma_i\frac{n_i^2}{|e_i|^\gamma} \big|\sjump{u_1-u_2}\big| \big| \sjump{\phi}\big|\,ds~=: I_1+I_2+I_3+I_4. \quad (Say) \label{sumofbdd}
\end{eqnarray}
We now bind each of the terms $I_k,\;k=1,\,2,\;3,\;4,$ independently. Firstly, we will examine $I_1$ using the Cauchy-Schwartz inequality and the Lipschitz continuity  \eqref{lpsctsinq} with the fact $|\nabla \varphi \cdot \bar{\bf n}|\leq |\nabla \varphi |$, for any $\varphi \in \mathcal{S}^{\mathtt{n}}_h(\mathcal{T}_h)$, as
\begin{eqnarray}
I_1&=&\sum_{\tau_i \in \mathcal{T}_h}\int_{\tau_i}\big|\mathscr{G}(|\nabla u_1|)\nabla u_1-\mathscr{G}(|\nabla u_2|)\nabla u_2\big|\big|\nabla \phi \big|\,dx\nonumber\\
&\leq& C_{M_3}\,\Big(\sum_{\tau_i \in \mathcal{T}_h}\|\nabla (u_1-u_2)\|^2_{L^2(\tau_i)}\Big)^{\frac{1}{2}} \times \Big(\sum_{\tau_i \in \mathcal{T}_h}\|\nabla \phi\|^2_{L^2(\tau_i)}\Big)^{\frac{1}{2}}.\nonumber
\end{eqnarray}
The bound of the term $I_2$ is obtained by applying the Cauchy-Schwartz inequality, and hence 
\begin{eqnarray}
I_2 &\leq&\Big(\sum_{e_i \in \mathscr{E}_{int, h}} \|\smean{(\mathscr{G}(|\nabla u_1|)\nabla u_1-\mathscr{G}(|\nabla u_2|)\nabla u_2) \cdot \bar{\bf n})}\|^2_{L^2(e_i)}\Big)^{\frac{1}{2}}\times\Big(\sum_{e_i \in \mathscr{E}_{int, h}} \|\sjump{\phi}\|^2_{L^2(e_i)}\Big)^{\frac{1}{2}}. \nonumber \label{intdiri3.9}
\end{eqnarray}
To compute the term $\smean{(\mathscr{G}(|\nabla u_1|)\nabla u_1-\mathscr{G}(|\nabla u_2|)\nabla u_2) \cdot \bar{\bf n})}$ on the interior edges, i.e., $e_i\in \mathscr{E}_{int,h}$, we apply the trace inequality \eqref{tr1invh}, Lipschitz continuity \eqref{lpsctsinq} 
and the local bounded variation properties $A_1-A_2$ (Sect.  \ref{Thp1p2}), to have
\begin{eqnarray}
&&\sum_{e_i \in \mathscr{E}_{int,h}} \|\smean{(\mathscr{G}(|\nabla u_1|)\nabla u_1-\mathscr{G}(|\nabla u_2|)\nabla u_2) \cdot \bar{\bf n})}\|^2_{L^2(e_i)}\nonumber\\ &&~~~~~~~~\leq ~\sum_{\tau_i \in \mathscr{T}_{h}}\frac{C_{tr,1}^2C_{M_3}^2n^2_{i}(1+\varrho^2 \, \delta)}{2h_{i}} \times \|\nabla (u_1-u_2)\|^2_{L^2(\tau_i)}.\label{intedg3.10}
\end{eqnarray}
Utilizing the bound \eqref{intdiri3.9}, we get 
\begin{eqnarray}
 I_2~\leq~ \Big(\sum_{\tau_i \in \mathscr{T}_{h}}\frac{C_{tr,1}^2C_{M_3}^2(1+\varrho^2\, \delta)}{2\sigma_i|e_i|^{1-\gamma}}\,\|\nabla (u_1-u_2)\|^2_{L^2(\tau_i)}\Big)^{\frac{1}{2}}\times\Big(\sum_{e_i \in \mathscr{E}_{int, h}} \sigma_i\frac{n_i^2}{|e_i|^\gamma}\|\sjump{\phi}\|^2_{L^2(e_i)}\Big)^{\frac{1}{2}}.\nonumber 
\end{eqnarray}
Since $|e_i|\leq h_i,\; \forall \, i$.
By applying the Cauchy-Schwartz inequality to the component $I_3$, one may obtain
\begin{eqnarray}
I_3&=&\sum_{e_i \in \mathscr{E}_{int, h}} \int_{e_i} \big|\smean{\mathscr{G}(|\nabla u_1|)\nabla \phi \cdot \bar{\bf n}}\sjump{u_1}-\smean{\mathscr{G}(|\nabla u_2|)\nabla \phi \cdot \bar{\bf n}}\sjump{u_2}\big|\,ds\nonumber\\
&=& \sum_{e_i \in \mathscr{E}_{int, h}} \int_{e_i} \big| \smean{\nabla \phi \cdot\bar{\bf n}} \big|\, \big|\mathscr{G}(h^{-1}\sjump{u_1})\sjump{u_1}-\mathscr{G}(h^{-1} \sjump{u_2})\sjump{u_2}\big|\, ds \nonumber\\
&\leq& \Big(\sum_{e_i \in \mathscr{E}_{int, h}}C_{M_3}^2 \| \smean{\nabla \phi \cdot \bar{\bf n}}\|^2_{L^2(e_i)}\Big)^{\frac{1}{2}} \times \Big( \sum_{e_i \in \mathscr{E}_{int, h}} \|\sjump{u_1-u_2}\|^2_{L^2(e_i)} \Big)^{\frac{1}{2}}. \label{B-3-bd3.15}
\end{eqnarray}
Since the term $\mathscr{G}(h^{-1}|\sjump{\varphi}|)$ is resemble to $\mathscr{G}(|\nabla \varphi|)$, for any $\varphi \in \mathcal{S}^{\mathtt{n}}_h(\mathcal{T}_h)$. Evaluating the term $\smean{\nabla \phi \cdot \bar{\bf n}}$ on the interior edges, (say)  $e_{ij} \in \mathscr{E}_{int,h}$,  applying the trace inequality \eqref{tr1invh} to determine
\begin{eqnarray}
\|\smean{\nabla \phi \cdot \bar{\bf n}}\|^2_{L^2(e_{ij})}
&\leq&\frac{1}{2}\|\mathscr{G}(|\sjump{\phi}|)\nabla \phi \cdot \bar{\bf n}|_{\tau_i^e} \|^2_{L^2(e_{ij})}+\frac{1}{2}\|\mathscr{G}(|\sjump{\phi}|)\nabla \phi \cdot \bar{\bf n}|_{\tau_j^e}\|^2_{L^2(e_{ij})}\nonumber\\
&\leq&\frac{C_{tr,1}^2n_i^2(1+\varrho^2\,\delta)}{2h_i}\|\nabla \phi\|^2_{L^2(\tau_i)}.\label{b3bd3.1613}
\end{eqnarray}
From \eqref{B-3-bd3.15} and \eqref{b3bd3.1613}, we obtain 
\begin{eqnarray}
I_3&\leq& \Big( \sum_{\tau_i \in \mathcal{T}_h}\frac{C_{M_3}^2C_{tr,1}^2(1+\varrho^2\,\delta)}{2\sigma_i|e_i|^{1-\gamma}}\,\|\nabla \phi\|^2_{L^2(\tau_i)}\Big)^{\frac{1}{2}} \times \Big( \sum_{e_i \in \mathscr{E}_{int, h}} \sigma_i\frac{n_i^2}{|e_i|^\gamma}\, 
 \|\sjump{u_1-u_2}\|^2_{L^2(e_i)} \Big)^{\frac{1}{2}}. \label{B3bdd}\nonumber
\end{eqnarray}
The last term, $I_4$, is ultimately determined utilizing the Cauchy-Schwartz inequality as
\begin{eqnarray}
I_4~\leq~\Big(\sum_{e_i \in \mathscr{E}_h} \sigma_i\frac{n_i^2}{|e_i|^\gamma} \|\sjump{u_1-u_2}\|^2_{L^2(e_i)} \Big)^{\frac{1}{2}}\times \Big(\sum_{e_i \in \mathscr{E}_h} \sigma_i\frac{n_i^2}{|e_i|^\gamma}\|\sjump{\phi}\|^2_{L^2(e_i)} \Big)^{\frac{1}{2}}.
\end{eqnarray}
Combine all bounds $I_1-I_4$ with utilizing \eqref{sumofbdd}, and then adjusting the terms to have 

\begin{eqnarray}
&&\big|\mathcal{A}_h(u_1;u_1, \phi)-\mathcal{A}_h(u_2;u_2, \phi)\big| \leq~\sqrt{2}\,\Big(\sum_{\tau_i \in \mathscr{T}_{h}} \Big\{C_{M_3}^2+\frac{C_{tr,1}^2 C_{M_3}^2 (1+\varrho^2\, \delta)}{2\sigma_i|e_i|^{1-\gamma}} \Big\}\nonumber\\
&&~~~~~~~~~ \times \|\nabla (u_1-u_2)\|^2_{L^2(\tau_i)}+\sum_{e_i \in \mathscr{E}_h}\Big\{1+\frac{C_{M_3}^2C_{tr,1}^2  (1+\varrho^2\, \delta)}{2\sigma_i|e_i|^{1-\gamma}} \Big\}\, \sigma_i \frac{n_i^2}{|e_i|^\gamma} \|\sjump{u_1-u_2}\|^2_{L^2(e_i)}\Big)^{\frac{1}{2}}\nonumber\\
&&~~~~~~~~~\times\Big(\sum_{\tau_i \in \mathscr{T}_{h}}  \|\nabla \phi\|^2_{L^2(\tau_i)}+\sum_{e_i \in \mathscr{E}_{h}} \sigma_i \frac{n_i^2}{|e_i|^\gamma} \|\sjump{\phi}\|^2_{L^2(e_i)}\Big)^{\frac{1}{2}}.\nonumber
\end{eqnarray}
Implementing $C^2_{31}=2\times \max\Big\{ C_{M_3}^2+\frac{C_{tr,1}^2 C_{M_3}^2 (1+\varrho^2\, \delta)}{2\sigma_i|e_i|^{1-\gamma}},\, 1+\frac{C_{M_3}^2C_{tr,1}^2 (1+\varrho^2\, \delta)}{2\sigma_i|e_i|^{1-\gamma}}\Big\}$, this achieves the required inequality.
\end{proof}
The following lemma shows the monotonicity condition of the operator $\mathcal{A}_h(\cdot, \cdot)$.
\begin{lemma}\label{lm3.23mono}
Assume that the inequalities \eqref{lpsctsinq}--\eqref{monoinq} and the assumptions $A_1-A_2$ (Sect. \ref{Thp1p2}) hold. Further, let $\sigma_i>4\,|e_i|^{\gamma-1}C_{tr,1}^2C_{M_3}^2C_{M_4}^{-1}(1+\varrho^2\, \delta)$ with $\varrho,\, \delta,\,\mu>0$ and $\gamma \geq 1$, then the semi-linear form $\mathcal{A}_h(\cdot, \cdot)$ satisfies the strong monotonicity condition, for all $\phi_1,\,\phi_2 \in \mathcal{S}^{\mathtt{n}}_h(\mathcal{T}_h)$, in the sense that 
\begin{equation}
\mathcal{A}_h(\phi_1; \phi_1, \phi_1-\phi_2)- \mathcal{A}_h(\phi_2; \phi_2, \phi_1-\phi_2)~\geq~ C_{32}\,  \|\phi_1-\phi_2\|^2_{\mathcal{E}},\label{stmonocd3.19} 
\end{equation}
where $C_{32}=\min \big\{\frac{C_{M_4}}{2},\, 1-\frac{4\,C_{tr,1}^2C_{M_3}^2\,(1+\varrho^2\, \delta)}{C_{M_4}\sigma_i\,|e_i|^{1-\gamma}}\big\}$, for $i\in [1:N_h]$.
\end{lemma}
\begin{proof}
For $\phi_1,\,\phi_2 \in \mathcal{S}^{\mathtt{n}}_h(\mathcal{T}_h)$, we have
\begin{eqnarray}
&&\mathcal{A}_h(\phi_1; \phi_1, \phi_1-\phi_2)-\mathcal{A}_h(\phi_2; \phi_2, \phi_1-\phi_2)\nonumber\\
&&~~~= \sum_{\tau_i \in \mathcal{T}_h}\int_{\tau_i}\big(\mathscr{G}(|\nabla \phi_1|)\nabla \phi_1-\mathscr{G}(|\nabla \phi_2|)\nabla \phi_2\big)\cdot \nabla (\phi_1-\phi_2)\,dx\nonumber\\
&&~~~-\sum_{e_i \in \mathscr{E}_{int, h}} \int_{e_i} \big(\smean{\mathscr{G}(|\nabla \phi_1|)\nabla \phi_1 \cdot \bar{\bf n}}-\smean{\mathscr{G}(|\nabla \phi_2|)\nabla \phi_2 \cdot \bar{\bf n}}\big) \sjump{\phi_1-\phi_2}\,ds\nonumber\\
&&~~~- \sum_{e_i \in \mathscr{E}_{int, h}} \int_{e_i} \big(\smean{\mathscr{G}(|\nabla \phi_1|)\nabla (\phi_1-\phi_2) \cdot \bar{\bf n}}\sjump{\phi_1}-\smean{\mathscr{G}(|\nabla \phi_2|)\nabla (\phi_1-\phi_2) \cdot \bar{\bf n}}\sjump{\phi_2}\big)\,ds \nonumber\\
&&~~~+ \sum_{e_i \in \mathscr{E}_{h}} \int_{e_i} \sigma_i \,\frac{n_i^2}{|e_i|^\gamma}\, \sjump{\phi_1-\phi_2}\, \sjump{\phi_1-\phi_2}\, ds. \nonumber
\end{eqnarray}    
Utilizing the inequality \eqref{monoinq}, one  may get
\begin{eqnarray}
&&\mathcal{A}_h(\phi_1; \phi_1, \phi_1-\phi_2)-\mathcal{A}_h(\phi_2; \phi_2, \phi_1-\phi_2)
\geq ~C_{M_4}\sum_{\tau_i \in \mathcal{T}_h}\int_{\tau_i}|\nabla(\phi_1-\phi_2)|^2\,dx \nonumber\\
&&~~~-\sum_{e_i \in \mathscr{E}_{int, h}} \int_{e_i} \big(\smean{\mathscr{G}(|\nabla \phi_1|)\nabla \phi_1 \cdot \bar{\bf n}}-\smean{\mathscr{G}(|\nabla \phi_2|)\nabla \phi_2 \cdot \bar{\bf n}}\big) \sjump{\phi_1-\phi_2}\,ds\nonumber\\
&&~~~- \sum_{e_i \in \mathscr{E}_{int, h}} \int_{e_i} \big(\smean{\mathscr{G}(|\nabla \phi_1|)\nabla (\phi_1-\phi_2) \cdot \bar{\bf n}}\sjump{\phi_1}-\smean{\mathscr{G}(|\nabla \phi_2|)\nabla (\phi_1-\phi_2) \cdot \bar{\bf n}}\sjump{\phi_2}\big)\,ds 
\nonumber\\
&&~~~+ \sum_{e_i \in \mathscr{E}_{h}} \int_{e_i} \sigma_i\, \frac{n_i^2}{|e_i|^\gamma}\, |\sjump{\phi_1-\phi_2}|^2\,ds\nonumber\\ &&~~~:= I_5+I_6+I_7+I_8. \label{blbdd3.20}
\end{eqnarray}
We bind the terms $I_6$ and $I_7$ of  \eqref{blbdd3.20} on the right-hand side terms by making an argument very similar to the bound of $I_2$ and $I_3$ as in Lemma \ref{lpscntofbilinear}. Initially, we investigate the bound of the expression $I_6$ as
\begin{eqnarray}
I_6&=&\sum_{e_i \in \mathscr{E}_{int, h}} \int_{e_i} \big(\smean{\mathscr{G}(|\nabla \phi_1|)\nabla \phi_1 \cdot \bar{\bf n}}-\smean{\mathscr{G}(|\nabla \phi_2|)\nabla \phi_2 \cdot \bar{\bf n}}\big) \, \sjump{\phi_1-\phi_2}\,ds\nonumber\\
&\leq&\Big(\sum_{\tau_i \in \mathscr{T}_{h}}\frac{C_{tr,1}^2C_{M_3}^2\,(1+\varrho^2\, \delta)}{2\, \sigma_i\,|e_i|^{1-\gamma}}\,\|\nabla (\phi_1-\phi_2)\|^2_{L^2(\tau_i)}\Big)^{\frac{1}{2}}\Big(\sum_{e_i \in \mathscr{E}_{int, h}}\sigma_i\,\frac{n_i^2}{|e_i|^\gamma}\,\|\sjump{\phi_1-\phi_2}\|^2_{L^2(e_i)}\Big)^{\frac{1}{2}}.\nonumber
\end{eqnarray}
One may bound the term $I_7$ by arguing the same lines as $I_3$ to have
\begin{eqnarray}
I_7&=&\sum_{e_i \in \mathscr{E}_{int, h}} \int_{e_i} \big(\smean{\mathscr{G}(|\nabla \phi_1|)\nabla (\phi_1-\phi_2) \cdot \bar{\bf n}}\sjump{\phi_1}-\smean{\mathscr{G}(|\nabla \phi_2|)\nabla (\phi_1-\phi_2) \cdot \bar{\bf n}}\sjump{\phi_2}\big)\,ds \nonumber\\
&\leq&\Big( \sum_{\tau_i \in \mathcal{T}_h}\frac{C_{tr,1}^2C_{M_3}^2(1+\varrho^2\,\delta)}{2\, \sigma_i\, |e_i|^{1-\gamma}}\,\|\nabla (\phi_1-\phi_2)\|^2_{L^2(\tau_i)}\Big)^{\frac{1}{2}} \Big(\sum_{e_i \in \mathscr{E}_{int, h}}\sigma_i\,\frac{n_i^2}{|e_i|^\gamma}\,\|\sjump{\phi_1-\phi_2}\|^2_{L^2(e_i)}\Big)^{\frac{1}{2}}.\nonumber
\end{eqnarray}
Thus, Eq. \eqref{blbdd3.20} implies
\begin{eqnarray}
&&\mathcal{A}_h(\phi_1; \phi_1, \phi_1-\phi_2)-\mathcal{A}_h(\phi_2; \phi_2, \phi_1-\phi_2)
\geq C_{M_4}\sum_{\tau_i \in \mathcal{T}_h}\|\nabla (\phi_1-\phi_2)\|^2_{L^2(\tau_i)} \nonumber\\
&&~~~- \Big(\sum_{\tau_i \in \mathscr{T}_{h}}\frac{2C_{tr,1}^2C_{M_3}^2\,(1+\varrho^2\, \delta)}{\sigma_i\,|e_i|^{1-\gamma}}\,\|\nabla (\phi_1-\phi_2)\|^2_{L^2(\tau_i)}\Big)^{\frac{1}{2}}\Big(\sum_{e_i \in \mathscr{E}_{h}} \sigma_i\,\frac{n_i^2}{|e_i|^\gamma}\,\|\sjump{\phi_1-\phi_2}\|^2_{L^2(e_i)}\Big)^{\frac{1}{2}}\nonumber\\
&&~~~+ \sum_{e_i \in \mathscr{E}_{h}} \sigma_i\, \frac{n_i^2}{|e_i|^\gamma}\, \|\sjump{\phi_1-\phi_2}\|^2_{L^2(e_i)}. \nonumber
\end{eqnarray}
An application of the $\epsilon-$form of Young's inequality to the second term on the right-hand side, followed by the arrangement of the terms, one may obtain
\begin{eqnarray}
&&\mathcal{A}_h(\phi_1; \phi_1, \phi_1-\phi_2)-\mathcal{A}_h(\phi_2; \phi_2, \phi_1-\phi_2)\geq \sum_{\tau_i \in \mathcal{T}_h}C_{M_4}\,\|\nabla (\phi_1-\phi_2)\|^2_{L^2(\tau_i)} 
\nonumber\\
&&~~~- \sum_{\tau_i \in \mathscr{T}_{h}}\frac{4\,C_{tr,1}^2C_{M_3}^2\,(1+\varrho^2\, \delta)}{\epsilon\, \sigma_i\,|e_i|^{1-\gamma}}\,\|\nabla (\phi_1-\phi_2)\|^2_{L^2(\tau_i)}\nonumber\\
&&~~~ +\sum_{e_i \in \mathscr{E}_{h}} \sigma_i\,\frac{(1-\frac{\epsilon}{2})\, n_i^2}{|e_i|^\gamma}\, \|\sjump{\phi_1-\phi_2}\|^2_{L^2(e_i)}. \nonumber
\end{eqnarray}
Selecting $\epsilon=\frac{8\,C_{tr,1}^2C_{M_3}^2(1+\varrho^2\, \delta)}{C_{M_4}\,\sigma_i\,|e_i|^{1-\gamma}}$ and $C_{32}=\min \big\{C_{M_4}-\frac{4\,C_{tr,1}^2C_{M_3}^2\,(1+\varrho^2\, \delta)}{\epsilon\,\sigma_i\,|e_i|^{1-\gamma}},\quad  1-\frac{\epsilon}{2} \big\}$, this implies the desired inequality \eqref{stmonocd3.19}.
\end{proof}
As a consequence,  the finite dimensional space $\mathcal{S}^{\mathtt{n}}_h(\mathcal{T}_h)$ is a Hilbert space equipped with the norm $\| \cdot\|_{\mathcal{E}}$ induced by the inner product space $(\cdot,\cdot)_{\mathcal{E}}$, which is easily demonstrable. Furthermore, all norms are comparable in the finite-dimensional space. 

Let us define the norm  $|\|\cdot|\|$ on $\mathcal{S}^{\mathtt{n}}_h(\mathcal{T}_h)$ by
\begin{equation*}
|\|\psi |\|:=\Big(\sum_{\tau_i \in \mathcal{T}_h}\int_{\tau_i} (| \nabla \psi|^2+|\psi|^2)\,dx+ \sum_{{e_i} \in \mathscr{E}_{int, h}} \int_{e_i} \sigma_i\frac{n_i^2}{|e_i|^\gamma}|\sjump{\psi}|^2\,ds+\sum_{e\in \mathscr{E}_{bd,h}}\int_e  |\psi|^2\,ds\Big)^{\frac{1}{2}},
\end{equation*}
which is equivalent to the norm $\|\cdot\|_{\mathcal{E}}$. Hence, $\exists$ a constant $C_{33}>0$ such that 
\begin{equation}
 \|\psi\|_{\mathcal{E}}\leq |\|\psi |\|\leq C_{33}\|\psi\|_{\mathcal{E}} \quad \text{for all} \quad \phi \in \mathcal{S}^{\mathtt{n}}_h(\mathcal{T}_h).   \label{eqvinq}
\end{equation}
The operator $\mathcal{L}_h(\cdot)$ is linear. By recalling the Cauchy-Schwarz inequality, we find the bound of the operator $\mathcal{L}_h(\cdot)$ for any $\phi_h\in \mathcal{S}^n_h(\mathcal{T}_h)$ to have
\begin{eqnarray}
|\mathcal{L}_h(\phi_h)| &\leq& \sum_{\tau_i \in \mathcal{T}_h}\int_{\tau_i}|f \phi_h|\,dx +\sum_{e_i \in \mathscr{E}_{bd,h}} \int_{e_i} \sigma_i \,\frac{n_i^2}{|e_i|^\gamma} |g\; \phi_h|\,ds,
\nonumber\\
&\leq & \Big( \sum_{\tau_i \in \mathcal{T}_h}\int_{\tau_i}|f|^2\,dx+ \sum_{e_i \in \mathscr{E}_{bd,h}} \int_{e_i} C_{tr,1}^2\frac{n_i^2}{|e_i|}\, |g|^2\,ds\Big)^{\frac{1}{2}}\nonumber\\ 
&& \times \Big( \sum_{\tau_i \in \mathcal{T}_h}\int_{\tau_i}|\phi|^2\,dx+\sum_{e_i \in \mathscr{E}_{bd,h}} \int_{e_i} \sigma_i \,\frac{n_i^2}{|e_i|^\gamma}\,|\phi|^2\,ds\Big)^{\frac{1}{2}}\nonumber\\
&\leq & \Big( \sum_{\tau_i \in \mathcal{T}_h}\int_{\tau_i}|f|^2\,dx+ \sum_{e_i \in \mathscr{E}_{bd,h}} \int_{e_i} C_{tr,1}^2\frac{n_i^2}{|e_i|}\, |g|^2\,ds\Big)^{\frac{1}{2}}\times |\|\phi|\|.
\end{eqnarray}
By using the relation \eqref{eqvinq}, we arrive at
\begin{equation}
 |\mathcal{L}_h(\phi_h)| ~\leq~ C_{34} \| \phi_h\|_{\mathcal{E}},  \label{Op:L:bound}
\end{equation}
where
\begin{equation}
C_{34}=C_{33}\,\Big( \sum_{\tau_i \in \mathcal{T}_h}\int_{\tau_i}|f|^2\,dx+ \sum_{e_i \in \mathscr{E}_{bd,h}} \int_{e_i} C_{tr,1}^2\frac{n_i^2}{|e_i|}\,|g|^2\,ds\Big)^{\frac{1}{2}}.\nonumber
\end{equation}
This shows that the operator $\mathcal{L}_h(\cdot)$ is bounded for any $\phi_h\in \mathcal{S}^n_h(\mathcal{T}_h)$. Therefore,  $\mathcal{L}_h(\cdot)$ is a bounded linear operator. Hence, by utilizing  the Ritz representation theory, there exists unique $\mathrm{f}\in \mathcal{S}^n_h(\mathcal{T}_h)$ with $\mathrm{f}=\mathrm{f}(f,g)$ such that   
$$\mathcal{L}_h(\phi_h)=(\mathrm{f}, \phi_h),\quad \forall \phi_h\in \mathcal{S}^n_h(\mathcal{T}_h).$$
Define the linear functional $\mathscr{F}_{u_h}:\, \mathcal{S}^n_h(\mathcal{T}_h) \longmapsto\mathbb{R}$  such that, for any $u_h \in \mathcal{S}^n_h(\mathcal{T}_h)$,
$$\mathscr{F}_{u_h}(\phi_h)=\mathcal{A}_h(u_h; u_h, \phi_h),\quad \forall \phi_h \in \mathcal{S}^n_h(\mathcal{T}_h).$$
From the virtue of Lemma \ref{lpscntofbilinear} with $u_1=u_h,\, u_2=0$ and $\phi=\phi_h$, one can easily get the estimate  
$$|\mathscr{F}_{u_h}(\phi_h)|\leq C_{31}\|u_h\|_{\mathcal{E}}\|\phi_h\|_{\mathcal{E}},\quad \forall \phi_h\in \mathcal{S}^n_h(\mathcal{T}_h).$$
This implies that the $\mathscr{F}_{u_h}$ is a bounded linear functional. Hence, from  the Ritz representation theorem, we can extract an element $\mathcal{H}(u_h)$ from $ \mathcal{S}^n_h(\mathcal{T}_h)$ such that 
$$\mathscr{F}_{u_h}(\phi_h)=(\mathcal{H}(u_h), \phi_h)_{\mathcal{E}}, \quad \forall \phi_h\in \mathcal{S}^n_h(\mathcal{T}_h).$$
We now show that the operator $\mathcal{H}$ is a bijection from $\mathcal{S}^n_h(\mathcal{T}_h)$ to itself. 

Note that if the solution exists in the finite-dimensional space $\mathcal{S}^n_h(\mathcal{T}_h)$, then it is unique. From the operator theory, it is sufficient to show that the operator $\mathcal{H}$ is Lipschitz continuous and the strong monotonic, then it is bijective. Therefore, $\mathcal{H}$ has an inverse that is Lipschitz continuous. Let us examine each of these properties individually. 
We know that  
\begin{eqnarray}
\|\mathcal{H}(u_h)-\mathcal{H}(\tilde{u}_h)\|_{\mathcal{E}}=\sup_{\phi \in \mathcal{S}^n_h(\mathcal{T}_h)}\frac{|(\mathcal{H}(u_h)-\mathcal{H}(\tilde{u}_h), \phi_h)_{\mathcal{E}}|}{\|\phi\|_{\mathcal{E}}},\quad \forall u_h,\,\tilde{u}_h \in \mathcal{S}^n_h(\mathcal{T}_h).\label{hine}
\end{eqnarray}
With the help of Lemma \ref{lpscntofbilinear}, one may easily compute,  for $u_h,\,\tilde{u}_h,\,\phi_h \in \mathcal{S}^n_h(\mathcal{T}_h)$, 
\begin{eqnarray}
&&|(\mathcal{H}(u_h)-\mathcal{H}(\tilde{u}_h), \phi_h)_{\mathcal{E}}|~=~|\mathscr{F}_{u_h}(\phi_h)-\mathscr{F}_{\tilde{u}_h}(\phi_h)|\nonumber\\&&~~~~=|\tilde{\mathcal{A}}_h(u_h; u_h, \phi_h)-\tilde{\mathcal{A}}_h(\tilde{u}_h; \tilde{u}_h, \phi_h)|~\leq~ \tilde{C}_{31}\, \|u_h-\tilde{u}_h\|_{\mathcal{E}}\|\phi_h\|_{\mathcal{E}}. \label{hbd}
\end{eqnarray}
Inequalities \eqref{hine} and \eqref{hbd} yield the Lipschitz continuity of $\mathcal{H}$.

Next, we show the strong monotonicity property for the non-linear operator $\mathcal{H}$ by recalling the Lemma \eqref{lm3.23mono}. Thus, for $u_h,\,\tilde{u}_h \in \mathcal{S}^n_h(\mathcal{T}_h)$ and $\sigma_i>4\,|e_i|^{\gamma-1}C_{tr,1}^2C_{M_3}^2C_{M_4}^{-1}(1+\varrho^2\, \delta)$ with $\varrho,\, \delta,\,\mu>0$ and $\gamma \geq 1$, we have 
\begin{eqnarray}
&&|(\mathcal{H}(u_h)-\mathcal{H}(\tilde{u}_h), u_h-\tilde{u}_h)_{\mathcal{E}}|~=~|\mathscr{F}_{u_h}(u_h-\tilde{u}_h)-\mathscr{F}_{\tilde{u}_h}(u_h-\tilde{u}_h)|\nonumber\\
&&~~~=~|\tilde{\mathcal{A}}_h(u_h; u_h, u_h-\tilde{u}_h)-\tilde{\mathcal{A}}_h(\tilde{u}_h; \tilde{u}_h, u_h-\tilde{u}_h)|~\geq~ \tilde{C}_{32}\,\|u_h-\tilde{u}_h\|^2_{\mathcal{E}}.\nonumber 
\end{eqnarray}
This yields the strong monotonicity of $\mathcal{H}$ from $\mathcal{S}^n_h(\mathcal{T}_h)$ to itself. Since the operator $\mathcal{H}$ is a bijective, the inverse of $\mathcal{H}$ exists, and Lipschitz is continuous. Then, there exists $u_h^{dg} \in \mathcal{S}^n_h(\mathcal{T}_h)$ such that 
$$\mathcal{H}(u_h^{dg})=\mathrm{f}, \; \text{for any} \; \mathrm{f} \in \mathcal{S}^n_h(\mathcal{T}_h).$$
This implies that the DG problem \eqref{dgformulation} admits a unique solution in $\mathcal{S}^n_h(\mathcal{T}_h)$. 

Equivalently, we can argue that the discrete problem \eqref{dgformulation} admits the unique discrete solution $u_h^{dg} \in \mathcal{S}^n_h(\mathcal{T}_h)$. We assert the following theorem for the existence and unique solution for the discrete model \eqref{dgformulation}.
\begin{theorem} Assume that the Lemmas \ref{lpscntofbilinear} and \ref{lm3.23mono} are valid. Further, we assume that the condition $\sigma_i>4\,|e_i|^{\gamma-1}C_{tr,1}^2C_{M_3}^2C_{M_4}^{-1}(1 +\varrho^2\, \delta)$ with $\varrho,\, \delta,\,\mu>0$ and $\gamma \geq 1$, is satisfied, then the discrete problem \eqref{dgformulation} has a unique solution in $\mathcal{S}^n_h(\mathcal{T}_h)$.
\end{theorem}
The following section will concentrate on the convergence error analysis for the problem \eqref{contmodel}-\eqref{dbdry}. 
\section{Error Analysis}
This section addresses the $L^2$-norm and the energy-norm error estimates for the discrete model problem \eqref{dgformulation}. First, we determine the energy-norm apriori type error estimate. Then, using the energy error estimate, we will establish the $L^2-$ error estimates. We first deteriorate the error into two components, $\eta$ and $\xi$, by incorporating the auxiliary variable $\mathcal{I}_h u\in \mathcal{S}^n_h(\mathcal{T}_h)$ to get the energy error estimates. Consequently, we have
\begin{eqnarray}
u-u^{dg}_{h}=(u-\mathcal{I}_hu)+(\mathcal{I}_hu-u^{dg}_{h})=:\eta + \xi,  \label{erreq4.1}
\end{eqnarray}
where $\eta:=u-\mathcal{I}_h u$ and $\xi:=\mathcal{I}_hu-u^{dg}_{h}$. To obtain the main result of the error in the $\mathcal{E}$-energy norm, we employ the auxiliary error bounds of $\eta$ and $\xi$.

The intermediate error bound for $\eta$ in the $\mathcal{E}$-energy norm is provided in the following lemma.
\begin{lemma}[{Error bound for $\eta$}]\label{rhorerr}
Let $u\in \mathcal{S}$ be the solution of the continuous problem \eqref{contmodel}--\eqref{dbdry}. Furthermore, if all the requirements of the Lemmas \ref{lpscntofbilinear} -- \ref{lm3.23mono} are fulfilled. Then there exists a positive constant $C_{41}$ such that the following inequality holds,  for $\mathcal{I}_h u\in \mathcal{S}^n_h(\mathcal{T}_h)$, 
\begin{equation}
\|\eta\|_{\mathcal{E}}~\leq~ C_{41}\, \Big(\sum_{\tau_i\in \mathcal{T}_h} \frac{h_i^{{2q_i}-2}}{n_i^{2m_i-2}}\,\|u\|^2_{H^{m_i}(\tau_i)}\Big)^{\frac{1}{2}},\label{rhobound}
\end{equation}
where $C_{41}=C_{ap, 1}\,\max\big\{C_{tr,2}\sqrt{\big(3}+C(\delta, \varrho)\big),\;1\big\}$ and $1\leq q_i\leq \min\{m_i, n_i+1\}$ with $n_i \geq 1$, $m_i \geq 2$, for all $i\in [1:N_h]$.
\end{lemma}
\begin{proof}
Using the notion of the energy norm and invoking the approximation property \eqref{appforelmnts} in the first term, we obtain
\begin{eqnarray}
\|\eta\|^2_{\mathcal{E}}&=&\sum_{\tau_i \in \mathcal{T}_h}\int_{\tau_i} | \nabla (u-\mathcal{I}_h u)|^2\,dx 
+ \sum_{{e_i} \in \mathscr{E}_{h}} \int_{e_i} \sigma_i\frac{n_i^2}{|e_i|^\gamma}|\sjump{u-\mathcal{I}_h u}|^2\,ds\nonumber\\
&\leq& C^2_{ap,1} \sum_{\tau_i \in \mathcal{T}_h}\frac{h_i^{{2q_i}-2}}{n_i^{{2m_i}-2}}\|u\|^2_{H^{m_i}(\tau_i)} + \sum_{{e_i} \in \mathscr{E}_{h}} \int_{e_i} \sigma_i\frac{n_i^2}{|e_i|^\gamma}|\sjump{u-\mathcal{I}_h u}|^2\,ds.\label{rhbdeq4.4}
\end{eqnarray}
Now, the second term is evaluated on inner edges of $e_{ij} (=e_i)\subset\partial\tau_i\cap\partial\tau_j,\,\forall i, j$. Considering the bounded local variation properties $A_1-A_2$ (refer to Sec. \ref{Thp1p2}), the approximate property \eqref{appforedge}, and the trace Lemma \ref{tracelemma4.2}, one may deduce that
\begin{eqnarray}
\|\sjump{u-\mathcal{I}_h u}\|_{L^2(e_{ij})}^2&\leq& \|u-\mathcal{I}_h u|_{\tau^e_i}\|_{L^2(e_{i})}^2+\|u-\mathcal{I}_h u|_{\tau^e_j}\|_{L^2(e_{j})}^2\nonumber\\
&\leq& 2C^2_{tr,2}C^2_{ap,1}\big(1+C(\delta, \varrho)\big)\frac{h_i^{{2q_i}-1}}{n_i^{2m_i-1}}\,\|u\|^2_{H^{m_i}(\tau_i)}\nonumber
\end{eqnarray}
with considering the degree of the polynomial $n^{2m_i}_i\leq n_i^{2m_i-1}$, and the same holds true for $m_j$. By making the same arguments, one may bound the boundary term  as
\begin{eqnarray}
\|\sjump{u-\mathcal{I}_h u}\|_{L^2(e_{i})}^2&\leq&  2C^2_{tr,2}C^2_{ap,1}\frac{h_i^{{2q_i}-1}}{n_i^{2m_i-1}}\,\|u\|^2_{H^{m_i}(\tau_i)},\nonumber
\end{eqnarray}
and consequently, 
\begin{eqnarray}
\sum_{{e_i} \in \mathscr{E}_{h}} \int_{e_i} \sigma_i\frac{n_i^2}{|e_i|^\gamma}|\sjump{u-\mathcal{I}_h u}|^2\,ds~\leq~2C^2_{tr,2}C^2_{ap,1}\big(3+C(\delta, \varrho)\big)\sum_{\tau_i\in \mathcal{T}_h} \frac{h_i^{{2q_i}-1}}{n_i^{2m_i-1}}\,\|u\|^2_{H^{m_i}(\tau_i)}.\label{bdintdr4.4}
\end{eqnarray}
The required result is obtained by combining Eqs \eqref{rhbdeq4.4} and \eqref{bdintdr4.4}, and choosing the constant $C_{41}^2=C^2_{ap,1}\max\big\{C^2_{tr,2}\big(3+C(\delta, \varrho)\big),\;1\big\}$.
\end{proof}
We acquire the intermediate error bound for $\xi$ in the subsequent lemma.
\begin{lemma}[{Error bound for $\xi$}]\label{thetarerr}
Assume that  $u^{dg}_{h}\in \mathcal{S}^n_h(\mathcal{T}_h)$ be the solution of the discrete problem \eqref{dgformulation}. Furthermore, if all the requirements of the Lemmas \ref{lpscntofbilinear} -- \ref{lm3.23mono} are fulfilled, then there is a positive  $C_{42}$ such that, for $\mathcal{I}_h u\in \mathcal{S}^n_h(\mathcal{T}_h)$, we have 
\begin{equation}
\|\xi\|_{\mathcal{E}}\leq C_{42}\, \Big(\sum_{\tau_i \in \mathcal{T}_h} \big\{\frac{h_i^{{2q_i}-2}}{n_i^{2m_i-2}}+ \frac{h_i^{{2q_i}+\gamma-3}}{n_i^{{2m_i}-2}}+\frac{h_i^{{2q_i}-2}}{n_i^{{2m_i}-3}}\big\}\,\|u\|^2_{H^{m_i}(\tau_i)}\Big)^{\frac{1}{2}},\label{4.6thetabd} 
\end{equation}
where $C_{42}=C_{43}/C_{32}$ with $C_{43}=\max \{\sqrt{\big(}2C_{M_3}^2\, C^2_{ap,1}+C_{tr,1}^2\,C_{M_3}^2\,\sigma_i^{-1}(1+\varrho^2\, \delta)+C_{44}^2\big)\}$,\; $C_{44}=C_{ap,2}\,C_{tr,1}\sqrt{(}1+C(\delta, \varrho))\, \sqrt{(}1+\varrho^2\,\delta),$ and $1\leq q_i\leq \min\{m_i, n_i+1\}$ with $n_i \geq 1$, $i\in [1:N_h]$. 
\end{lemma}
\begin{proof} An application of monotonicity inequality  \eqref{stmonocd3.19} of Lemma \ref{lm3.23mono} with $\phi_1=I_hu $, $\phi_2=u_h^{dg}$, and then use of  \eqref{orthogotype} leads to
\begin{eqnarray}
C_{32}  \|\xi\|^2_{\mathcal{E}}&\leq& \mathcal{A}_h(\mathcal{I}_hu;\mathcal{I}_hu, \xi)- \mathcal{A}_h(u_h^{dg}; u_h^{dg}, \xi)\nonumber\\
&=& \mathcal{A}_h(\mathcal{I}_hu; \mathcal{I}_hu, \xi)- \mathcal{A}(u; u, \xi)+\big(\mathcal{L}_h(\xi)- \mathcal{A}_h(u_h^{dg}; u_h^{dg}, \xi) \big)\nonumber\\ 
&=& \mathcal{A}_h(\mathcal{I}_hu; \mathcal{I}_hu, \xi)- \mathcal{A}(u; u, \xi)\nonumber\\ 
&=&\sum_{\tau_i \in \mathcal{T}_h}\int_{\tau_i}\big(\mathscr{G}(|\nabla \mathcal{I}_hu|)\nabla \mathcal{I}_hu-\mathscr{G}(|\nabla u|)\nabla u\big)\cdot \nabla \xi\,dx\nonumber\\
&&+\sum_{e_i \in \mathscr{E}_{int, h}} \int_{e_i} \big(\smean{\mathscr{G}(|\nabla u|)\nabla u \cdot \bar{\bf n}}-\smean{\mathscr{G}(|\nabla \mathcal{I}_hu|)\nabla \mathcal{I}_hu \cdot \bar{\bf n}}\big) \sjump{\xi}\,ds\nonumber\\
&&+ \sum_{e_i \in \mathscr{E}_{int, h}} \int_{e_i} \big|\smean{\mathscr{G}(|\nabla u|)\nabla \xi \cdot \bar{\bf n}}\sjump{u}-\smean{\mathscr{G}(|\nabla \mathcal{I}_hu|)\nabla \xi \cdot \bar{\bf n}}\sjump{\mathcal{I}_hu}\big|\,ds\nonumber\\
&=:&II_1+II_2+II_3. \quad (Say)\label{boundinterror}
\end{eqnarray}
The integral terms $II_k$, where $k=1,\,2,\,3,\,4$, should be prominently bounded. Employing  the Cauchy-Schwarz inequality, inequality \eqref{lpsctsinq}, and then applying Lemma \ref{apprxlm3.2}, we get to  
\begin{eqnarray}
II_1&\leq&\sum_{\tau_i \in \mathcal{T}_h}\int_{\tau_i}\big|\mathscr{G}(|\nabla u|)\nabla u-\mathscr{G}(|\nabla \mathcal{I}_hu|)\nabla \mathcal{I}_hu\big|\big|\nabla \xi\big|\,dx \nonumber\\
&=&C_{M_3}\,\Big(\sum_{\tau_i \in \mathcal{T}_h}\|\nabla \eta \|^2_{L^2(\tau_i)}\Big)^{\frac{1}{2}}\times \Big(\sum_{\tau_i \in \mathcal{T}_h}\|\nabla \xi\|_{L^2(\tau_i)}^2\Big)^{\frac{1}{2}}. \nonumber\\ 
&\leq&C_{M_3}\,C_{ap, 1}\Big(\sum_{\tau_i \in \mathcal{T}_h}2\,\frac{h_i^{{2q_i}-2}}{n_i^{{2m_i}-2}}\|u\|^2_{H^{m_i}(\tau_i)}\Big)^{\frac{1}{2}}\times \Big(\sum_{\tau_i \in \mathcal{T}_h}\frac{1}{2}\|\nabla \xi\|_{L^2(\tau_i)}^2\Big)^{\frac{1}{2}}. \label{bdI1}
\end{eqnarray}
Implementing Lemma \ref{apprxlm3.2} and making an argument akin to that of Lemma \ref{lpscntofbilinear}, one can derive
\begin{eqnarray}
II_2&\leq& \sum_{e_i \in \mathscr{E}_{int, h}} \int_{e_i} \big(\smean{\mathscr{G}(|\nabla u|)\nabla u \cdot \bar{\bf n}}-\smean{\mathscr{G}(|\nabla \mathcal{I}_hu|)\nabla \mathcal{I}_hu \cdot \bar{\bf n}}\big) \sjump{\xi}\,ds\nonumber\\
&\leq&\Big(\sum_{\tau_i \in \mathscr{T}_{h}}\frac{C_{tr,1}^2C_{M_3}^2(1+\varrho^2\, \delta)}{\sigma_i}\,\frac{h_i^{{2q_i}-3+\gamma}}{n_i^{{2m_i}-2}}\;\|u\|^2_{H^{m_i}(\tau_i)}\Big)^{\frac{1}{2}}\times\Big(\sum_{e_i \in \mathscr{E}_{int, h}} \sigma_i\,\frac{n_i^2}{2|e_i|^\gamma}\|\sjump{\xi}\|^2_{L^2(e_i)}\Big)^{\frac{1}{2}}.\nonumber
\end{eqnarray}
The last term, $II_3$, can be bounded by arguing along the same lines as the bound of $I_3$ to have
\begin{eqnarray}
II_3 &=& \sum_{e_i \in \mathscr{E}_{int, h}} \int_{e_i} \big|\smean{\mathscr{G}(|\nabla u|)\nabla \xi \cdot \bar{\bf n}}\sjump{u}-\smean{\mathscr{G}(|\nabla \mathcal{I}_hu|)\nabla \xi \cdot \bar{\bf n}}\sjump{\mathcal{I}_hu}\big|\,ds \nonumber\\
&=&\sum_{e_i \in \mathscr{E}_{int, h}} \int_{e_i} \big| \smean{\nabla \xi \cdot\bar{\bf n}} \big|\, \big|\mathscr{G}(h^{-1}\sjump{u})\sjump{u}-\mathscr{G}(h^{-1} \sjump{\mathcal{I}_hu})\sjump{\mathcal{I}_hu}\big|\, ds \nonumber\\
&\leq& \Big( \sum_{\tau_i \in \mathcal{T}_h}\frac{C_{M_3}^2C_{tr,1}^2n_i^2(1+\varrho^2\,\delta)}{2h_i}\|\nabla \xi\|^2_{L^2(\tau_i)}\Big)^{\frac{1}{2}} \times \Big( \sum_{e_i \in \mathscr{E}_{int, h}} 
 \|\sjump{\eta}\|^2_{L^2(e_i)} \Big)^{\frac{1}{2}}.\label{4.8ets}
\end{eqnarray}
Using the approximation inequality \eqref{appforedge} and the assumptions $A_1-A_2$ (cf., Sec \ref{Thp1p2}), we next estimate the expression $\sjump{\eta}$ on the internal edges $e_{ij}(=e_i)\in \mathscr{E}_{int,h}$, as
\begin{eqnarray}
\|\sjump{\eta}\|^2_{L^2(e_{ij})}&\leq&\|\eta|_{\tau_i}\|^2_{L^2(e_{ij})}+\|\eta|_{\tau_j}\|^2_{L^2(e_{ij})}\nonumber\\
&\leq & C^2_{ap,2}\,(1+C(\delta, \varrho))\, \frac{h_i^{{2q_i}-1}}{n_i^{{2m_i}-1}}\,\|u\|^2_{H^{m_i}(\tau_i)}.\nonumber
\end{eqnarray}
Hence, the estimation of  $\sjump{\eta}$ on the edges $\mathscr{E}_{int, h}$ is provided by
\begin{eqnarray}
\sum_{e_i \in \mathscr{E}_{int,h}} \|\sjump{\eta}\|^2_{L^2(e_i)}~\leq~  \sum_{\tau_i \in \mathcal{T}_h}\,C^2_{ap,2}(1+C(\delta, \varrho))\,\frac{h_i^{{2q_i}-1}}{n_i^{{2m_i}-1}}\|u\|^2_{H^{m_i}(\tau_i)}.\label{4.10rhobd}
\end{eqnarray}
Incorporating the bound \eqref{4.10rhobd} in \eqref{4.8ets}, we arrive at
\begin{eqnarray}
II_3 &\leq&C_{44}\,\Big(\sum_{\tau_i \in \mathcal{T}_h} \frac{h_i^{{2q_i}-2}}{n_i^{2m_i-3}}\,\|u\|^2_{H^{m_i}(\tau_i)}\Big)^{\frac{1}{2}}\times \Big(\sum_{\tau_i \in \mathcal{T}_h}\frac{1}{2}\,\|\nabla \xi\|^2_{L^2(\tau_i)}\Big)^{\frac{1}{2}},\nonumber
\end{eqnarray}
where $C_{44}^2=C^2_{ap,2}\,C_{tr,1}^2\,(1+C(\delta, \varrho))\times(1+\varrho^2\,\delta)$. 

Then, by summing the integral bounds $II_1$--$II_3$ with \eqref{boundinterror}, and then some adjustment of the term leads to
\begin{eqnarray}
\|\xi\|^2_{\mathcal{E}}~\leq~ \frac{C_{43}}{C_{32}}\,\Big(\sum_{\tau_i \in \mathcal{T}_h} \big\{\frac{h_i^{{2q_i}-2}}{n_i^{2m_i-2}}+ \frac{h_i^{{2q_i}+\gamma-3}}{n_i^{{2m_i}-2}}+\frac{h_i^{{2q_i}-2}}{n_i^{{2m_i}-3}}\big\}\,\|u\|^2_{H^{m_i}(\tau_i)}\Big)^{\frac{1}{2}}\times \|\xi\|_{\mathcal{E}}.\nonumber 
\end{eqnarray}
By setting $C_{42}=\frac{C_{43}}{C_{32}}$, where $C_{43}^2=\max\big\{ 2\,C_{M_3}^2\,C^2_{ap,1}+C_{tr,1}^2\,C_{M_3}^2\,\sigma_i^{-1}(1+\varrho^2\, \delta)+C_{44}^2\big\}$, this accomplishes the remaining proof. 
\end{proof}
Finally, we obtain the primary error bound in the energy norm utilizing the auxiliary error bounds from Lemmas \ref{rhorerr} and \ref{thetarerr}.
\begin{theorem}[{Energy-norm error estimation}]\label{mainresults}
Assume that $u\in \mathcal{S}$ and $u^{dg}_{h}\in \mathcal{S}^n_h(\mathcal{T}_h)$ represent the solutions of the continuous problem \eqref{contmodel}--\eqref{dbdry} and the discrete problem \eqref{dgformulation}, respectively. Please assume that all of their conditions have been met by the Lemmas \ref{rhorerr} and \ref{thetarerr}.
Moreover, if $u|_{\tau_i} \in H^{m_i}(\tau_i),\, m_i \geq 2,\;\; \text{for all}\;\; \tau_i \in \mathcal{T}_h$, $i\in [1: N_h]$, then there is a positive constant $C_{45}$ such that the following inequality 
\begin{equation}
\|u-u_h^{dg}\|_{\mathcal{E}}~\leq~ C_{45}\, \Big(\sum_{\tau_i \in \mathcal{T}_h} \big\{\frac{h_i^{{2q_i}-2}}{n_i^{2m_i-2}}+ \frac{h_i^{{2q_i}+\gamma-3}}{n_i^{{2m_i}-2}}+\frac{h_i^{{2q_i}-2}}{n_i^{{2m_i}-3}}\big\}\,\|u\|^2_{H^{m_i}(\tau_i)}\Big)^{\frac{1}{2}}\label{4.10err} 
\end{equation}
 holds, where $1\leq q_i \leq \min\{m_i,\, n_i+1\}$, $n_i\geq 1$ and $C_{45}=2\,\max\{C_{41},\,C_{42}\}$. 
\end{theorem}
\begin{proof}
One may quickly determine the inequality \eqref{4.10err} by employing the triangle inequality in Eq. \eqref{erreq4.1} and then invoking Lemmas \ref{rhorerr}--\ref{thetarerr}.
\end{proof}
\begin{remark} It is observed that from Lemma \ref{thetarerr},
\begin{itemize}
\item[(i)] for the sufficient small $h$, $\mathtt{m}\geq 2$ and $\gamma=d-1$ with $d\geq2$,  Eq. \eqref{4.6thetabd} implies
$$\frac{h_i^{{2q_i}-2}}{n_i^{2m_i-2}}+ \frac{h_i^{{2q_i}+\gamma-3}}{n_i^{{2m_i}-2}}+\frac{h_i^{{2q_i}-2}}{n_i^{{2m_i}-3}} \lessapprox \frac{h_i^{{2q_i}-2}}{n_i^{2m_i-3}}\quad \text{for} \quad n \geq 1.$$ 
Moreover, it equivalent to $h_i^{{2q_i}-2}/n_i^{2m_i-3}$ for $\gamma=1$.
\item[(ii)] Thus, the error bound of $\xi$ in Lemma \ref{thetarerr} can be estimated as
\begin{equation}
\|\xi\|_{\mathcal{E}}\lessapprox C_{42}\, \Big(\sum_{\tau_i \in \mathcal{T}_h} \frac{h_i^{{2q_i}-2}}{n_i^{2m_i-3}}\,\|u\|^2_{H^{m_i}(\tau_i)}\Big)^{\frac{1}{2}}, \nonumber
\end{equation}
and hence, the error between $u$ and $u^{dg}_h$ in the $\mathcal{E}$-energy norm is defined as in Theorem \ref{mainresults} can be approximated by
\begin{equation}
\|u-u_h^{dg}\|_{\mathcal{E}}\lessapprox C_{45}\, \Big(\sum_{\tau_i \in \mathcal{T}_h} \frac{h_i^{{2q_i}-2}}{n_i^{2m_i-3}}\,\|u\|^2_{H^{m_i}(\tau_i)}\Big)^{\frac{1}{2}}.\label{4.11energyestimates}
\end{equation}
\item[(iii)] These estimates involving $h_i$, $n_i$, and $m_i$ should be interpreted as asymptotic estimates, showing the convergence rates in the $\mathcal{E}$-energy norm. That is, for all $n_i\geq 1$ and $m_i\geq 2$, the error $\|u-u_h^{dg}\|_{\mathcal{E}}\thicksim \mathcal{O}\big(h_i^{q_i-1}/n_i^{m_i-3/2}\big)$ as $h_i\rightarrow 0$ with $1\leq q_i \leq \min\{m_i,\, n_i+1\}$ for all $i$, which is optimal in $h$ and mildly suboptimal in $\mathtt{n}$ (total degree of polynomial). 
\end{itemize}
\end{remark}
\noindent
Next, we will use the energy-norm error estimates to obtain the $L^2-$norm error. 
For this, we gather some approximation properties before diving into the error analysis (cf., \cite{babuska1987,Riviere2001}, [\cite{Riviere2001}, p. 908-913]).  
\begin{lemma}[\cite{babuska1987,Riviere2001}] \label{ctsappprop}
Let $\mathcal{T}_h$ be a regular subdivision. For a given $\phi\in \mathscr{B}^{\mathtt{m}}_\mathtt{n}(\Omega; \mathcal{T}_h),\; \mathtt{m}\geq 2,$ we assume that there exists a continuous interpolant $\Pi_h$ with $\Pi_h \phi \in \accentset{\ast}{\mathcal{S}}^n_h \cap C^0(\bar{\Omega})$ and 
$\Pi_h \phi|_{e}=\phi|_e$,\, $e\in \mathcal{E}_h$, and the constants $\hat{C}_{ap,i}>0,\, i=1,\,2,\,3,$ depending on $\mathtt{m}$ but independent of $\mathtt{n},\,h$, such that 
\begin{itemize}
\item[(i)] for any $0\leq j\leq \mathtt{m}$ and $0\leq j\leq 2$,
\begin{equation}\label{capp1forelmnts}
\|\phi-\Pi_h \phi\|_{H^{j}(\tau_i)}\leq~\tilde{C}_{ap,1}\, \frac{h_i^{q-j}}{\mathtt{n}^{\mathtt{m}-j-\eta}}\,\Big(\sum_{\tau_j\in \accentset{\ast}{\tau_i}} \|\phi\|^2_{H^\mathtt{m}(\tau_j)}\Big)^{\frac{1}{2}},
\end{equation}
where $\eta=0$ if $j=0,\,1,$ and $\eta=1$ if $j=2$,
\item[(ii)] for  $j+\frac{1}{2}<\mathtt{m}$ and $j=0,\,1$,
\begin{equation}\label{capp2forelmnts}
\|\phi-\Pi_h \phi\|_{H^{j}(e_i)}\leq~\tilde{C}_{ap,2}\, \frac{h_i^{q-j-\frac{1}{2}}}{\mathtt{n}^{\mathtt{m}-j-\frac{1}{2}-\tilde{\eta}}}\,\Big(\sum_{\tau_j\in \accentset{\ast}{\tau_i}} \|\phi\|^2_{H^\mathtt{m}(\tau_j)}\Big)^{\frac{1}{2}},
\end{equation}
where $\tilde{\eta}=0$ if $j=0,$ and $\tilde{\eta}=\frac{1}{2}$ if $j=1$,
\item[(iii)] for  $0\leq j\leq \mathtt{m}-1+\frac{2}{r}$ and $j=0,\,1$,
\begin{equation}\label{capp3forelmnts}
\|\phi-\Pi_h \phi\|_{W^{j}_r(K_i)}\leq~\tilde{C}_{ap,3}\, \frac{h_i^{q-j-1+\frac{2}{r}}}{\mathtt{n}^{\mathtt{m}-j-1+\frac{2}{r}}}\,\Big(\sum_{\tau_j\in \accentset{\ast}{\tau_i}} \|\phi\|^2_{H^\mathtt{m}(\tau_j)}\Big)^{\frac{1}{2}},
\end{equation}
where $q=\min\{\mathtt{m},\mathtt{n}+1\}$. However,  $\accentset{\ast}{\tau_i}:=\big\{\tau_j:~ meas(\partial \tau_i \cap \partial \tau_j)\, \text{is positive} \big\}$, and $e_i$ denotes the edge of $\partial \tau_i$.
\end{itemize}
\end{lemma}
We begin by defining the linearized version of the model problem \eqref{contmodel}, as
 \begin{eqnarray}
  \mathcal{J}\Psi~=:-\nabla\cdot \big(\mathscr{G}(|\nabla u|)\, \nabla \Psi+\mathscr{G}_u(|\nabla u|)\, \nabla u\, \Psi\big) 
  \end{eqnarray} 
  Then, the dual problem with Dirichlet-boundary can be read as Let $\Psi\in H^2(\Omega)\cap H_0^1(\Omega)$ be the solution of the problem which satisfies   
\begin{subequations}\label{adjpb4.14}
\begin{eqnarray}
 -\nabla\cdot \big(\mathscr{G}(|\nabla u|)\, \nabla \Psi\big)+\mathscr{G}_u(|\nabla u|)\, \nabla u\cdot \nabla\Psi &=& G,\quad \text {in}~ \Omega \label{adjpb4.14a}\\
\Psi ~=~ 0 \quad \text { on }~ \Gamma, \label{diribdry4.14b}
\end{eqnarray}
\end{subequations}
where $G\in L^2(\Omega)$. The solution $\Psi\in H^2(\Omega)$ complies with the continuous dependency  (see Gilbarg and Trudinger \cite[Lemma 9.17]{gilbarg1983}). Consequently, we have
\begin{equation}
\|\Psi\|_{H^2(\Omega)}~\leq~ R_{C_2} \|G\|_{L^2(\Omega)},
 \label{regularity}   
\end{equation}
where $R_{C_2}$ known as the regularity constant.

We now study the $L^2$-norm error estimates for the problem \eqref{dgformulation} in the subsequent theorem. Since the duality argument (Aubin-Nitsche technique) is an excellent tool for analyzing $L^2$-norm error estimates, it is a well-developed technique for the continuous Galerkin case; let's apply the following approach to the discontinuous Galerkin finite element method. 
\begin{theorem}[$L^2$-norm error estimate]\label{L2ErrEstmatesThm4.6}
Assume that $u\in \mathcal{S}$ and $u^{dg}_{h}\in \mathcal{S}^n_h(\mathcal{T}_h)$ represent the solutions to the continuous problem \eqref{contmodel}--\eqref{dbdry} and the discrete problem \eqref{dgformulation}, respectively.  We further suppose that all the conditions of Theorem \ref{mainresults} are satisfied. Then, there is a positive constant $C_{47}$, that is independent of the mesh parameters, such that the following estimate
\begin{eqnarray}
&&\|u-u_h^{dg}\|_{L^2(\Omega)}~\leq~ C_{47}\times \Big[\Big(\frac{h^q}{\mathtt{n}^{\mathtt{m}-1/2}}+
\frac{h^{q+\gamma/2-1/2}}{\mathtt{n}^{\mathtt{m}}}\Big) \times \|u\|_{H^{\mathtt{m}}(\Omega, \mathcal{T}_h)}\nonumber\\ 
&&\hspace{3cm}+ \Big( \frac{h^{2q-2}}{\mathtt{n}^{2\mathtt{m}-3}}+\frac{h^{2q-1}}{\mathtt{n}^{2\mathtt{m}-2}}+\frac{h^{2q+\gamma/2-2}}{\mathtt{n}^{2\mathtt{m}-3/2}}+\frac{h^{2q+\gamma/2-5/2}}{\mathtt{n}^{2\mathtt{m}-2}}\Big)\times\|u\|^2_{H^{\mathtt{m}}(\Omega, \mathcal{T}_h)}
\Big] \label{4.18l2est} 
\end{eqnarray}
holds with $\gamma \geq 1$, where the constant $C_{47}=C_{45}C_{46}\max\{1, C_{45}\}$, $C_{45}$ and $C_{46}$ are stated as \eqref{const} and \eqref{4.11energyestimates}, respectively.
\end{theorem}
\begin{proof}
Utilizing $G=u-u_h^{dg}$ in \eqref{adjpb4.14}, we obtain 
\begin{equation}
\|u-u_h^{dg}\|^2_{L^2(\Omega)}~=~\int_\Omega  (u-u_h^{dg})(-\nabla\cdot \big(\mathscr{G}(|\nabla u|)\, \nabla \Psi\big)+\mathscr{G}_u(|\nabla u|)\, \nabla u\cdot \nabla\Psi)\,dx.\nonumber
\end{equation}
With setting $\phi=u-u_h^{dg}$, and use integration by parts to integrate it element-by-element, and then use of  Eq. \eqref{orthogotype}, we arrive at
\begin{eqnarray}
&&\|\phi\|^2_{L^2(\Omega)}~=~\sum_{\tau_i\in \mathcal{T}_h}\int_{\tau_i} \mathscr{G}(|\nabla u|)\nabla \phi\cdot \nabla \Psi \,dx-\sum_{\tau_i\in \mathcal{T}_h}\int_{\partial \tau_i}  \phi\,\mathscr{G}(|\nabla u|)\nabla \Psi\cdot\bar{\bf n}\,ds\nonumber\\ 
&&~~~+ \sum_{\tau_i\in \mathcal{T}_h}\int_{\tau_i} \phi\,\mathscr{G}_u(|\nabla u|)\,\nabla u \cdot \nabla \Psi\,dx,\nonumber\\
&&~~~=\sum_{\tau_i\in \mathcal{T}_h}\int_{\tau_i} \mathscr{G}(|\nabla u|)\nabla u\cdot \nabla (\Psi-\Pi_h\Psi) \,dx-\sum_{\tau_i\in \mathcal{T}_h}\int_{\tau_i} \mathscr{G}(|\nabla u|)\nabla u_h^{dg}\cdot \nabla \Psi \,dx\nonumber\\
&&~~~+\sum_{\tau_i\in \mathcal{T}_h}\int_{\tau_i} \mathscr{G}(|\nabla u|)\nabla u\cdot \nabla \Pi_h\Psi \,dx+ \sum_{\tau_i\in \mathcal{T}_h}\int_{\tau_i} \phi\,\mathscr{G}_u(|\nabla u|)\,\nabla u \cdot \nabla \Psi\,dx\nonumber\\
&&~~~-\sum_{e_i\in \mathscr{E}_{h}}\int_{\e_i}  \phi\,\mathscr{G}(|\nabla u|)\nabla \Psi\cdot\bar{\bf n}\,ds\nonumber
\end{eqnarray}
\begin{eqnarray}
&&\|\phi\|^2_{L^2(\Omega)}~=~\sum_{\tau_i\in \mathcal{T}_h}\int_{\tau_i} \mathscr{G}(|\nabla u|)\nabla u\cdot \nabla (\Psi-\Pi_h\Psi) \,dx-\sum_{\tau_i\in \mathcal{T}_h}\int_{\tau_i} \mathscr{G}(|\nabla u|)\nabla u_h^{dg}\cdot \nabla \Psi \,dx\nonumber\\
&&~~~+ \sum_{\tau_i\in \mathcal{T}_h}\int_{\tau_i} \phi\,\mathscr{G}_u(|\nabla u|)\,\nabla u \cdot \nabla \Psi\,dx+\sum_{\tau_i\in \mathcal{T}_h}\int_{\tau_i} \mathscr{G}(|\nabla u_h^{dg}|)\nabla u_h^{dg}\cdot \nabla \Pi_h\Psi \,dx\nonumber\\&&~- \sum_{e_i\in \mathscr{E}_{int, h}}\int_{\e_i}\smean{\mathscr{G}(|\nabla u_h^{dg}|)\nabla u_h^{dg}\cdot\bar{\bf n}}\, \sjump{\Pi_h\Psi}\,ds\nonumber\\
&&~~~- \sum_{e_i\in \mathscr{E}_{int, h}}\int_{\e_i}\smean{\mathscr{G}(|\nabla u_h^{dg}|)\nabla \Pi_h\Psi\cdot\bar{\bf n}}\, \sjump{u_h^{dg}}\,ds\nonumber\\
&&~~~-\sum_{e_i\in \mathscr{E}_{h}}\int_{\e_i} \smean{\mathscr{G}(|\nabla u|)\nabla \Psi\cdot\bar{\bf n}}\, \sjump{\phi}\,ds
\nonumber 
\end{eqnarray}
Since $\Psi \in H^2(\Omega) \cap H^1_0(\Omega) $. By implementing the jump of $u$, $\Psi$ and $\Pi_h\Psi$ on the edges $e\in \mathscr{E}_h$ are zero. and adjusting the terms accordingly, then we achieved
\begin{eqnarray}
&&\|\phi\|^2_{L^2(\Omega)}~=~\sum_{\tau_i\in \mathcal{T}_h}\int_{\tau_i} [\mathscr{G}(|\nabla u|)-\mathscr{G}(|\nabla u_h^{dg}|)]\,\nabla u\cdot \nabla (\Psi-\Pi_h\Psi) \,dx\nonumber\\
&&~~~~-\sum_{\tau_i\in \mathcal{T}_h}\int_{\tau_i} (\mathscr{G}(|\nabla u|)-\mathscr{G}(|\nabla u_h^{dg}|))\nabla \phi \cdot \nabla (\Psi-\Pi_h\Psi) \,dx\nonumber\\
&&~~~~+\sum_{\tau_i\in \mathcal{T}_h}\int_{\tau_i} \mathscr{G}(|\nabla u|)\nabla \phi \cdot \nabla (\Psi-\Pi_h\Psi) \,dx+\sum_{\tau_i\in \mathcal{T}_h}\int_{\tau_i} [\mathscr{G}(|\nabla u|)-\mathscr{G}(|\nabla u_h^{dg}|)]\nabla \phi \cdot \nabla \Psi \,dx\nonumber\\
&&~~~~+ \sum_{\tau_i\in \mathcal{T}_h}\int_{\tau_i} \big\{\phi\,\mathscr{G}_u(|\nabla u|)-[\mathscr{G}(|\nabla u|)-\mathscr{G}(|\nabla u_h^{dg}|)]\big\}\,\nabla u \cdot \nabla \Psi\,dx\nonumber
\end{eqnarray}
\begin{eqnarray}
&&~~~~+\sum_{e_i\in \mathscr{E}_{int,h}}\int_{\e_i}\smean{[\mathscr{G}(|\nabla u|)-\mathscr{G}(|\nabla u_h^{dg}|)]\nabla (\psi-\Pi_h\Psi)\cdot\bar{\bf n}}\, \sjump{\phi}\,ds\nonumber\\
&&~~~~-\sum_{e_i\in \mathscr{E}_{int,h}}\int_{\e_i}\smean{[\mathscr{G}(|\nabla u|)-\mathscr{G}(|\nabla u_h^{dg}|)]\nabla \psi \cdot\bar{\bf n}}\, \sjump{\phi}\,ds\nonumber\\
&&~~~~-\sum_{e_i\in \mathscr{E}_{h}}\int_{\e_i}\smean{\mathscr{G}(|\nabla u|)\nabla (\Psi-\Pi_h\Psi) \cdot\bar{\bf n}}\, \sjump{\phi}\,ds
~=:~ \sum_{k=1}^{8}J_k. \label{L^2erreq}
\end{eqnarray}
We now individually bound each of the terms $J_k,\, k=1,\ldots,\,13$. We first consider $J_1$, and then using the H\"{o}lder's inequality, we have  
\begin{eqnarray}
|J_1|&\leq &\sum_{\tau_i\in \mathcal{T}_h}\int_{\tau_i}\big| [\mathscr{G}(|\nabla u|)-\mathscr{G}(|\nabla u_h^{dg}|)]\,\nabla u\cdot \nabla (\Psi-\Pi_h\Psi)  \big|\,dx\nonumber\\
&\leq &\sum_{\tau_i\in \mathcal{T}_h}\big|\mathscr{G}_u \big|\,\|\nabla (u- u_h^{dg})\|_{L^4(\tau_i)}\|\nabla u\|_{L^4(\tau_i)}\|\nabla (\Psi-\Pi_h\Psi)\|_{L^2(\tau_i)}\nonumber\\
&\leq &\sum_{\tau_i\in \mathcal{T}_h}\big|\mathscr{G}_u \big|\,\|u- u_h^{dg}\|_{W^1_4(\tau_i)}\| u\|_{W^1_4(\tau_i)}\|\Psi-\Pi_h\Psi\|_{H^1(\tau_i)}.\nonumber
\end{eqnarray}
From Eq. \eqref{hbdd} and the embedding of $H^{1}(\Omega)$, i.e., $H^{1}(\Omega) \hookrightarrow W^1_4(\Omega)$, we arrive at
\begin{eqnarray}
|J_1|&\leq &C_{B_2}C_{emb}^2\,\sum_{\tau_i\in \mathcal{T}_h}\|u- u_h^{dg}\|_{H^{1}(\tau_i)}\| u\|_{H^{1}(\tau_i)}\|\Psi-\Pi_h\Psi\|_{H^1(\tau_i)}.\nonumber
\end{eqnarray}
Using the bounds \eqref{ctsubd}, \eqref{regularity}, and the approximation inequality \eqref{capp1forelmnts} of Lemma \ref{ctsappprop} with $\mathtt{n}=2,\,q=2,\, j=1,\, \eta=0$, to obtain 
\begin{eqnarray}
|J_1|&\leq & C_{B_2}M_3R_{C_2}C_{emb}^2C_{33}\tilde{C}_{ap,1}\,\frac{h}{\mathtt{n}}\,\|\phi\|_{L^2(\Omega)}\times \|u- u_h^{dg}\|_{\mathcal{E}},\nonumber
\end{eqnarray}
where $M_3=C_{R_1} \big\{ \|f\|_{L^2(\Omega)}+\|g\|_{H^{\frac{3}{2}}(\Gamma)}\big\}$. \\ \noindent
Arguing similarly to above, we bound the terms $J_2$ and $J_3$ using the H\"{o}lder's inequality and the approximation Lemma \ref{ctsappprop} to have 
\begin{eqnarray}
|J_2|&\leq &\sum_{\tau_i\in \mathcal{T}_h}\int_{\tau_i} (\mathscr{G}(|\nabla u|)-\mathscr{G}(|\nabla u_h^{dg}|))\nabla \phi \cdot \nabla (\Psi-\Pi_h\Psi) \,dx\nonumber\\
&\leq &C_{B_2}C_{emb}^2\,\sum_{\tau_i\in \mathcal{T}_h}\|u- u_h^{dg}\|_{H^{1}(\tau_i)}\| \phi \|_{H^{1}(\tau_i)}\|\Psi-\Pi_h\Psi\|_{H^1(\tau_i)} \nonumber\\
&\leq &C_{B_2} R_{C_2}C_{emb}^2C_{33}^2\tilde{C}_{ap,1}\,\frac{h}{\mathtt{n}}\,\|\phi\|_{L^2(\Omega)}\times \|u- u_h^{dg}\|^2_{\mathcal{E}},\nonumber
\end{eqnarray}
and the estimate of $J_3$, as
\begin{eqnarray}
|J_3|&\leq & \sum_{\tau_i\in \mathcal{T}_h}\int_{\tau_i} \big|\mathscr{G}(|\nabla u|)\nabla \phi \cdot \nabla (\Psi-\Pi_h\Psi)\big|\,dx  \nonumber\\
&\leq & C_{B_1}\,\sum_{\tau_i\in \mathcal{T}_h} \|\phi\|_{H^1(\tau_i)} \|\Psi-\Pi_h\Psi\|_{H^1(\tau_i)} \nonumber\\
&\leq & C_{B_1}\tilde{C}_{ap,1} C_{33} \sum_{\tau_i\in \mathcal{T}_h} \,\frac{h}{\mathtt{n}} \|\Psi\|_{H^2(\tau_i)} \times \|u-u_h^{dg}\|_{\mathcal{E}} \nonumber\\
&\leq & C_{B_1}R_{C_2}\tilde{C}_{ap,1} C_{33} \,\frac{h}{\mathtt{n}} \,\|\phi\|_{L^2(\Omega)}\times \|u-u_h^{dg}\|_{\mathcal{E}}. 
\end{eqnarray}
Similarly, we estimate the fourth term $J_4$. Consider 
\begin{eqnarray}
|J_4| &\leq& \sum_{\tau_i\in \mathcal{T}_h}\int_{\tau_i}\big| [\mathscr{G}(|\nabla u|)-\mathscr{G}(|\nabla u_h^{dg}|)]\nabla \phi \cdot \nabla \Psi \big|\,dx\nonumber\\
&\leq&\sum_{\tau_i\in \mathcal{T}_h}\int_{\tau_i}\big|\mathscr{G}_u\big|\,\big||\nabla (u- u_h^{dg})|\nabla \phi \cdot \nabla \Psi \big|\,dx\nonumber\\
&\leq& C_{B_2}\,\sum_{\tau_i\in \mathcal{T}_h}\,\|\nabla (u- u_h^{dg})\|_{L^4(\tau_i)} \|\nabla \phi\|_{L^4(\tau_i)}\|\nabla \Psi \|_{L^2(\tau_i)}\nonumber\\
&\leq& C_{B_2}\,\sum_{\tau_i\in \mathcal{T}_h}\,\|u-u_h^{dg}\|_{W^1_4(\tau_i)}^2 \|\Psi \|_{H^1(\tau_i)}.\nonumber
\end{eqnarray}
Using the embedding result and inequality \eqref{regularity}, we get 
\begin{eqnarray}
|J_4| ~\leq~ C_{B_2}R_{C_2}C_{33}^2C_{emb}^2\, \| \phi\|_{L^2(\Omega)} \times \|u-u_h^{dg}\|_{\mathcal{E}}^2.\nonumber
\end{eqnarray}
Let us compute the fifth term $J_5$ by applying the H\"{o}lder's inequality and the embedding results to have 
\begin{eqnarray}
 |J_5|&\leq&  \sum_{\tau_i\in \mathcal{T}_h}\int_{\tau_i} \big| \big\{\phi\,\mathscr{G}_u(|\nabla u|)-[\mathscr{G}(|\nabla u|)-\mathscr{G}(|\nabla u_h^{dg}|)]\big\}\,\nabla u \cdot \nabla \Psi \big|\,dx\nonumber\\
 &\leq&  \sum_{\tau_i\in \mathcal{T}_h}\int_{\tau_i}  \big\{\big|\phi\,\mathscr{G}_u(|\nabla u|\big|+\big|\mathscr{G}_u\big|\,\big||\nabla u|-|\nabla u_h^{dg}|)\big|\big\}\,\big|\nabla u \cdot \nabla \Psi \big|\,dx\nonumber\\
 &\leq& C_{B_2} \,\sum_{\tau_i\in \mathcal{T}_h}\big( \|\phi\|_{L^2(\tau_i)}+\|\nabla \phi\|_{L^2(\tau_i)}\big)\,\|\nabla u\|_{L^4(\tau_i)} \|\nabla \Psi \|_{L^4(\tau_i)}\nonumber\\
 &\leq& C_{B_2} \,\sum_{\tau_i\in \mathcal{T}_h} \|\phi\|_{H^1(\tau_i)}\|u\|_{W^1_4(\tau_i)} \|\Psi \|_{W^1_4(\tau_i)}.\nonumber
\end{eqnarray}
Utilization of the embedding $H^2(\Omega)\hookrightarrow W^1_4(\Omega)$, inequalities \eqref{ctsubd} and \eqref{regularity} yields  
\begin{eqnarray}
|J_5|&\leq& C_{B_2}M_3 C_{33}C_{emb}^2 \,\|\phi\|_{L^2(\Omega)} \times \|u-u_h^{dg}\|_{\mathcal{E}}.\nonumber
\end{eqnarray}
We now bound the terms $J_6$ and $J_7$ by applying the H\"{o}lder's inequality, trace Lemma \eqref{tracelemma4.2}, and the Sobolev embedding theorem,
\begin{eqnarray}
|J_6|&\leq&\sum_{e_i\in \mathscr{E}_{int,h}}\int_{\e_i}\smean{[\mathscr{G}(|\nabla u|)-\mathscr{G}(|\nabla u_h^{dg}|)]\nabla (\psi-\Pi_h\Psi)\cdot\bar{\bf n}}\, \sjump{\phi}\nonumber\\
&\leq& C_{B_2} \tilde{C}_{ap,2}\sum_{e_i\in \mathscr{E}_{int,h}} \sqrt{\frac{(1+\delta_i\varrho_i)}{\sigma_i}}\, \frac{|e_i|^{\gamma/2}\,h^{1/2}}{\mathtt{n}^{3/2}} \|\phi\|_{W^{1,4}(e_i)} \|\Psi\|_{H^{2}(\Omega)} \|u-u_h^{dg}\|_{\mathcal{E}}\nonumber\\
&\leq& C_{B_2}R_{C_2}C_{emb} \tilde{C}_{ap,2}\Check{C}(\sigma,\varrho,\delta) \frac{h^{\gamma/2}}{\mathtt{n}^{3/2}} \|\phi\|_{L^{2}(\Omega}\times \|u-u_h^{dg}\|^2_{\mathcal{E}},
\end{eqnarray}
where $\Check{C}(\sigma,\varrho,\delta)=\max_i\big\{\sqrt{\frac{(1+\delta_i\varrho_i}{\sigma_i}}\big\}$, $i \in[ 1:N_h]$. Moreover, the term $J_9$ is bounded by
\begin{eqnarray}
|J_7|&\leq&\sum_{e_i\in \mathscr{E}_{int,h}}\int_{\e_i}\smean{[\mathscr{G}(|\nabla u|)-\mathscr{G}(|\nabla u_h^{dg}|)]\nabla \psi \cdot\bar{\bf n}}\, \sjump{\phi}\,ds\nonumber\\
&\leq& C_{B_2} \sum_{e_i\in \mathscr{E}_{int,h}} \frac{|e_i|^{\gamma/2}}{\sigma_i\,\mathtt{n}} \|\phi\|_{W^{1,4}(e_i)} \|\Psi\|_{H^{2}(\Omega)} \|\phi\|_{\mathcal{E}}\nonumber\\
&\leq& C_{B_2}R_{C_2}C_{emb}\Check{C}(\sigma,\varrho,\delta) \frac{h^{\gamma/2-1/2}}{\mathtt{n}} \|\phi\|_{L^{2}(\Omega}\times\|u-u_h^{dg}\|^2_{\mathcal{E}}.\nonumber
\end{eqnarray}

Now, finally, we estimate the last term $J_{8}$ as 
\begin{eqnarray}
|J_{8}|&\leq& \sum_{e_i\in \mathscr{E}_h}\int_{e_i}\big|\smean{\mathscr{G}(|\nabla u|)\nabla (\Psi-\Pi_h\Psi) \cdot\bar{\bf n}}\, \sjump{\phi}\big|\,ds\nonumber\\
&\leq&C_{B_1} \tilde{C}_{ap,2} \sum_{e_i\in \mathscr{E}_h} \sqrt{\frac{(3+\delta_i\varrho_i)}{2\sigma_i}}\frac{|e_i|^{\gamma/2}\, h_i^{1/2}}{n^{3/2}_i} 
\|\Psi\|_{H^{2}(\Omega)}\times\Big(\int_{e_i}\sigma_i\,\frac{n_i^2}{|e_i|^\gamma}\, |\sjump{\phi}|^2\,ds\Big)^{1/2}\nonumber\\
&\leq&C_{B_1} R_{C_2} \tilde{C}_{ap,2}\bar{C}(\sigma,\delta, \varrho)\frac{h^{\gamma/2+1/2}}{\mathtt{n}^{3/2}} 
\|\phi\|_{L^{2}(\Omega)}\times \|u-u_h^{dg}\|_{\mathcal{E}},
\end{eqnarray}
where $\bar{C}(\sigma,\delta, \varrho)=\max_i\big\{\sqrt{\frac{(3+\delta_i\varrho_i)}{2\sigma_i}}\big\}$, $i \in[ 1:N_h]$. 
We now combine all the bounds $J_1-J_{8}$, to have
\begin{eqnarray}
\|\phi\|_{L^2(\Omega)}&\leq& C_{46}\times \Big[\Big(\frac{h}{\mathtt{n}}+
\frac{h^{\gamma/2+1/2}}{\mathtt{n}^{3/2}}\Big)\times \|u- u_h^{dg}\|_{\mathcal{E}} \nonumber\\ &&+ \Big( 1+\frac{h}{\mathtt{n}}+\frac{h^{\gamma/2}}{\mathtt{n}^{3/2}}+\frac{h^{\gamma/2-1/2}}{\mathtt{n}}\Big) \times \|u-u_h^{dg}\|^2_{\mathcal{E}}
\Big].\nonumber 
\end{eqnarray}
where 
\begin{eqnarray}
C_{46}&=&\max\big\{R_{C_2}\tilde{C}_{ap,1}\big(C_{B_2}M_3C_{emb}^2C_{33}+C_{B_1} C_{33}\big),
\,C_{B_2} R_{C_2}C_{emb}^2C_{33}^2\tilde{C}_{ap,1},  \nonumber\\ && C_{B_2}C_{33}C_{emb}^2(R_{C_2}C_{33}+M_3),\,
\,C_{B_2}R_{C_2}C_{emb} \tilde{C}_{ap,2}\Check{C}(\sigma,\varrho,\delta),\,  C_{B_2}R_{C_2}C_{emb}\Check{C}(\sigma,\varrho,\delta), \nonumber\\ && C_{B_1} R_{C_2} \tilde{C}_{ap,2}\bar{C}(\sigma,\delta, \varrho)
\big\}.\label{const}
\end{eqnarray}
I am using the energy error estimate \eqref{4.11energyestimates} and setting $\phi=u-u_h^{dg}$ completes the rest of the proof of the theorem.
\end{proof}
In the following remark, let us conclude the above Theorem \ref{L2ErrEstmatesThm4.6}.
\begin{remark} One may observed that from Eq. \ref{4.18l2est} that 
\begin{itemize}
\item[(i)] for the sufficient small $h$, $\mathtt{m}\geq 2$ and $\gamma=d-1$ with $d\geq2$, the terms 
\begin{equation*}
\frac{h^q}{\mathtt{n}^{\mathtt{m}-1/2}}+
\frac{h^{q+\gamma/2-1/2}}{\mathtt{n}^{\mathtt{m}}}\lessapprox \frac{h^{q}}{\mathtt{n}^{\mathtt{m}-1/2}},\quad \mathtt{n} \geq 1.   
\end{equation*}
and
\begin{equation*}
 \frac{h^{2q-2}}{\mathtt{n}^{2\mathtt{m}-3}}+\frac{h^{2q-1}}{\mathtt{n}^{2\mathtt{m}-2}}+\frac{h^{2q+\gamma/2-2}}{\mathtt{n}^{2\mathtt{m}-3/2}}+\frac{h^{2q+\gamma/2-5/2}}{\mathtt{n}^{2\mathtt{m}-2}} \lessapprox \frac{h^{2q-2}}{\mathtt{n}^{2\mathtt{m}-3}}, \quad \mathtt{n} \geq 1.   %
\end{equation*}
\item[(ii)] Furthermore, the $L^2$-norm error in Theorem \ref{L2ErrEstmatesThm4.6} can be estimated by 
\begin{equation}
\|u-u_h^{dg}\|_{L^2(\Omega)}\lessapprox C_{47}\,\frac{h^{q}}{\mathtt{n}^{\mathtt{m}-1/2}}\, \|u\|_{H^{\mathtt{m}}(\Omega, \mathcal{T}_h)}.\label{4.11energyestimates}
\end{equation}
Since $\frac{h^{2q-2}}{\mathtt{n}^{2\mathtt{m}-3}} \leq \frac{h^{q}}{\mathtt{n}^{\mathtt{m}-1/2}}$ for $\mathtt{n} \geq 1$, $\mathtt{m} \geq 2$.
\item[(iii)] For all $\mathtt{n}\geq 1$ and $\mathtt{m}\geq 2$, the $L^2$-norm convergence is given by  $\mathcal{O}\big(h^{q}/\mathtt{n}^{\mathtt{m}-1/2}\big)$, which is optimal in $h$ and sub-optimal in $\mathtt{n}$ (total degree of polynomial). 
\end{itemize}
\end{remark}
The following section confirms the error analysis we established in the previous section.
\section{Numerical Assessments}
In this section, we illustrate the theoretical findings with two examples. We set up a simple elastic unit-square domain with a manufactured solution in the first example. In contrast, the second example is considered a domain under anti-plane shear loading containing a single crack in two-dimensional. The numerical results are obtained by a computational code developed using \textsf{FreeFEM$++$} as a frontend wrapper \cite{MR3043640}. 
The computational domain is discretized using shape-regular triangles.  The material parameter (shear modulus $\mu$) and $\alpha$ are kept as $1.0$ unless stated otherwise. We will compute the errors in the $L^2$-norm and $\mathcal{E}$-energy norm. Further, the order of convergence is presented for both $L^2$ and energy norms. The following algorithm illustrates the steps taken in the entire computational procedure. 
\begin{table}[h!] \label{Algorithm}
\begin{center}
\addtolength{\tabcolsep}{10pt}
\begin{tabular}{ l l}
 \hline
 \multicolumn{2}{l}{\cellcolor{red!15}\texttt{Algorithm $5.1$:\;\;Adaptive algorithm for nonlinear strain-limiting model}} \\ 
 \hline 
 \multicolumn{2}{l}{Given tolerance $(tol)$, the parameters $\alpha,\, \beta,\, \mu$ and the source function $f$} \vspace{0.2cm} \\
 \hline    \vspace{0.3cm}
 \textbf{Step-I}\cellcolor{brown!8}&\cellcolor{brown!8} Initialize the mesh $\mathcal{T}_h=\mathcal{T}_h^n$ with mesh size $k=2^n$  \vspace*{0.2cm}\\
 
\textbf{Step-II}\cellcolor{brown!8} &\cellcolor{brown!8} Compute $u_{h, n}^{dg}$ by solving the discrete model problem \eqref{dgformulation} on $\mathcal{T}_h^n$   \\ 

\cellcolor{brown!8}  &\cellcolor{brown!8} \hspace{0.5cm}  Compute the {\it \textbf{ err} } and absolute value of the gradient, i.e., $|\nabla u_{h, n}^{dg}|$ \vspace*{0.2cm} \\

\textbf{Step-III} \cellcolor{brown!8}&\cellcolor{brown!8} {\bf While}\;  $ {\it \textbf{ err} } \geq tol$,\; {\bf do}  \\ 

\cellcolor{brown!8}&\cellcolor{brown!8} \hspace{0.5cm} Compute $u_{h, n}^{dg}$ by solving the discrete model problem \eqref{dgformulation} on $\mathcal{T}_h^n$\\

\cellcolor{brown!8}  &\cellcolor{brown!8} \hspace{0.5cm}  using Picards-iteration method \\

\cellcolor{brown!8}  &\cellcolor{brown!8} \hspace{0.5cm}  Compute the {\it \textbf{ err} } and the absolute value of the gradient \\
            
\cellcolor{brown!8}&\cellcolor{brown!8} \hspace{1 cm} {\bf If}\;  $ {\it \textbf{ err} } < tol$,\; {\bf do}  \\ 
\cellcolor{brown!8}&\cellcolor{brown!8} \hspace{1.5 cm}  \textbf{Break}; \\
\cellcolor{brown!8}&\cellcolor{brown!8} \hspace{1 cm} {\bf End If} \\
\cellcolor{brown!8}&\cellcolor{brown!8} \hspace{0.5cm}  Compute $L^2$ \& $\mathcal{E}$-errors,  stress and strain  $T_{13},\, T_{23},\, \epsilon_{13}$ and \, $\epsilon_{13}$ \\
\cellcolor{brown!8}&\cellcolor{brown!8} {\bf End While} \vspace{0.2cm}\\

\textbf{Step-IV}\cellcolor{brown!8}&\cellcolor{brown!8} \texttt{Go To {\bf Step-I} for refining the mesh}\\ 
\cellcolor{brown!8}&\cellcolor{brown!8} \hspace{0.5cm}  Set $n=n+1$, mesh size $k=2^{n+1}$ and transfer the solution $u_{h, n}^{dg}$\\ 

\cellcolor{brown!8}&\cellcolor{brown!8} \hspace{0.5cm}  Repeat the steps\\
 \hline
\end{tabular}
\end{center}
\end{table}
Additionally, the percentage relative error in the $L^2$ and the energy $\mathcal{E}$ are calculated by using the following  formulas 
\begin{equation*}
\% ~L^2: R. E. = \frac{\|e\|_{L^2}}{\|u\|_{L^2}}\times 100, \quad \text{and} \quad \% ~\mathcal{E}: R. E. = \frac{\|e\|_{\mathcal{E}}}{\|u\|_{\mathcal{E}}}\times 100,
\end{equation*}
and the experimental order of convergence (EOC) by
\begin{equation*}
EOC = \frac{log(E_{i})-log(E_{i-1})}{log(h_i)-log(h_{i-1})},
\end{equation*}
where $E_{i}$'s and $h_{i}$'s denote the errors and the corresponding mesh size.  
\begin{exam}[A simple elastic domain] \label{exm1}
Let $\Omega$ be a unit-square domain, i.e., $\Omega:=[0,1] \times [0,1]$. Consider the manufactured solution 
$$u(x,y) = x^{5/2}(1-x)y^{5/2}(1-y)$$
of a special type of sub-class of the theoretical model \eqref{contmodel}-\eqref{dbdry}. That is,
\begin{subequations}
\begin{eqnarray}
- \nabla \cdot \left(\frac{\nabla u }{2\left(1 + |\nabla u|\right)} \right)&=&f \quad ~~~~~\text { in } ~~ \Omega, \label{contmodelnu}\\
u &=& 0 ~ \quad ~~~~ \text { on } ~~ \Gamma \label{dbdrynu}
\end{eqnarray}
\end{subequations}
with $\Gamma:= \bar{\Gamma}_1\cup \bar{\Gamma}_2\cup \bar{\Gamma}_3 \cup \bar{\Gamma}_4$
\begin{figure}[h!]
\centering
\begin{tikzpicture}[scale=1.25]
\filldraw[draw=black, thick] (0,0) -- (4,0) -- (4,4) -- (0,4) -- (0,0);
\shade[inner color=red! 25, outer color=blue!25] (0,0) rectangle (4,4);
\node at (-0.35,-0.24)   {$(0,0)$};
\node at (4.24,-0.24)   {$(1,0)$};
\node at (-0.35,4.24)   {$(0,1)$};
\node at (4.25, 4.25)   {$(1,1)$};
\node at (-0.3, 2)   {$\Gamma_{4}$};
\node at (4.4, 2)   {$\Gamma_{2}$};
\node at (2, -0.26)   {$\Gamma_{1}$};
\node at (2, 4.3)   {$\Gamma_{3}$};
    \def\xOrigin{5.5}
    \def\yOrigin{0.5} 
    \draw[->, color=red, thick] (\xOrigin, \yOrigin) -- (\xOrigin + 1, \yOrigin) node[right] {$x$};
    \draw[->, color=blue, thick] (\xOrigin, \yOrigin) -- (\xOrigin, \yOrigin + 1) node[above] {$y$};
\end{tikzpicture}
\caption{A domain and the boundary indicators.}
\label{fig:h-conv}
\end{figure}
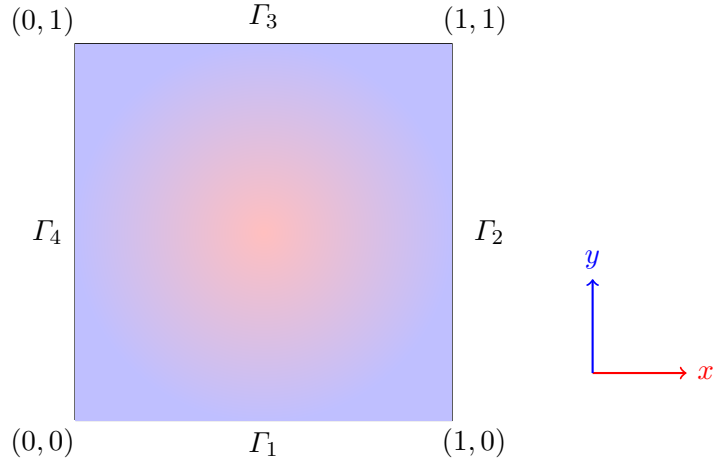
\end{exam}

We compute the source function $f(x,y)$ by using the exact solution $u(x,y)$ in the problem  \eqref{contmodelnu}. We set up the DGFEM to approximate the solution to the above boundary value problem \eqref{contmodelnu}-\eqref{dbdrynu}. The computational mesh with $128 \times 128$ elements and the corresponding surface mesh is depicted in n Figure \ref{Ex1unifsfmesh}. Figure \ref{Ex1exactcgsolalpha} shows the profiles of the exact solution (left) and the solution (right) obtained from the proposed numerical method. The DGFEM computations are done for the fixed penalty parameter $\sigma = 100$  on uniform meshes. We have used the piece-wise $P_2$ discontinuous elements in DGFEM. The nonlinear problem is linearized using Picard's algorithm.

For the $L^2$-norm and the energy norm $\mathcal{E}$, we studied the errors and their corresponding experimental order of convergence (EOC) using the $h$-version refinement, respectively.  Please refer to Tables \ref{table1examp1p1dc}-\ref{tab2ex1p3dc} for the convergence rates of the DGFEM solution.   From the Tables \ref{table1examp1p1dc}-\ref{tab2ex1p3dc}, one may observe that the obtained EOCs are optimal using $h$-refinement. The Figure \ref{Ex1L2H1eoc} shows the convergence plot for the $h$-refinement. Note that the left picture in Figure \ref{Ex1L2H1eoc} shows the $L^2$-norm convergence; however, the energy $\mathcal{E}$-norm convergence is reflected in the right picture of Figure \ref{Ex1L2H1eoc}. 

Further,  the convergence results for the $p$-version are presented in Table \ref{pversionref}. It is clear that the $L^2$-norm and energy norm $\mathcal{E}$ errors are decreasing, while the degree of the polynomials ($n_i$) is increasing. We have used the discontinuous polynomials P1dc,\, P2dc,\, P3dc, and P4dc of order $1,\,2,\,3$ and $4$, respectively. The mesh size is kept fixed in all computations, $h=1/30$. The table \ref{pversionref} contains the percentage of the relative error in the $L^2$ and the energy norm $\mathcal{E}$, where $e=u-u_h^{dg}$. Figure \ref{pversion} indicates that the percentage errors decrease with increasing polynomial order.  

Moreover, the Figure \ref{hpversion} (left) shows the log-log plot of $L^2$ and energy $\mathcal{E}$ errors in the $hp-$refinement, while the right plot depicts the \% relative errors for both the refinements. For this numerical test, we simultaneously allow refinements ($h$ and $p$).  The results are taken at the various mesh size ($h$) $1/5$, $1/10$, $1/15$ and $1/20$, respectively, however, only up to $4^{th}$ order polynomials are considered.  Table \ref{hptable} shows that the $hp$-version convergence is faster than the $p$-version. Since \% relative errors reach zero only at the mesh size  $20\times 20$, one can see in Figure \ref{hpversion}.
\begin{figure}[t!] \label{table1examp1p2dc}
\centering
\begin{subfigure}[b]{0.5\textwidth}
\centering
\includegraphics[width=0.99\linewidth]{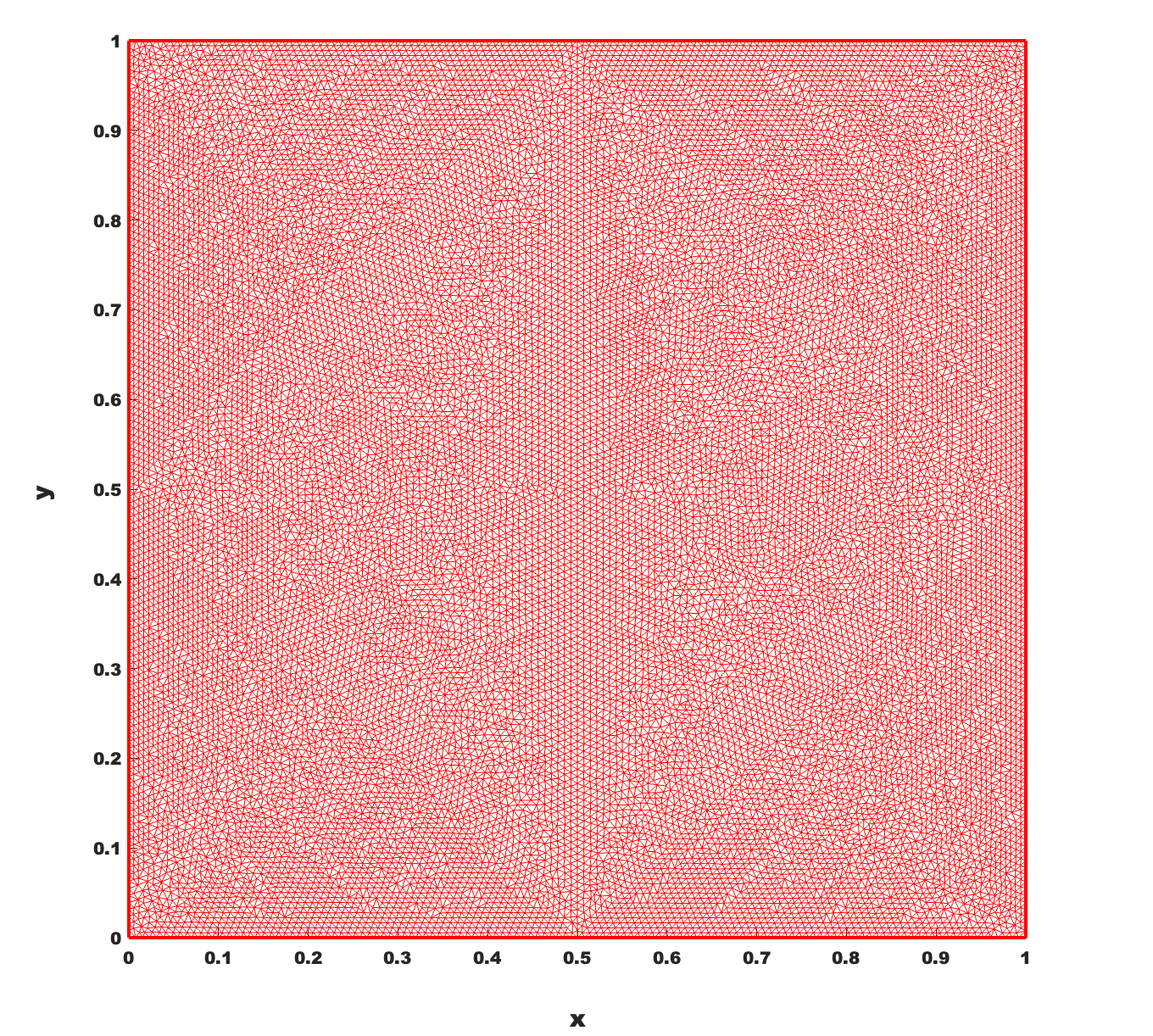} 
\end{subfigure}
\hfill
\begin{subfigure}[b]{0.49\textwidth}
\centering
\includegraphics[width=0.99\linewidth]{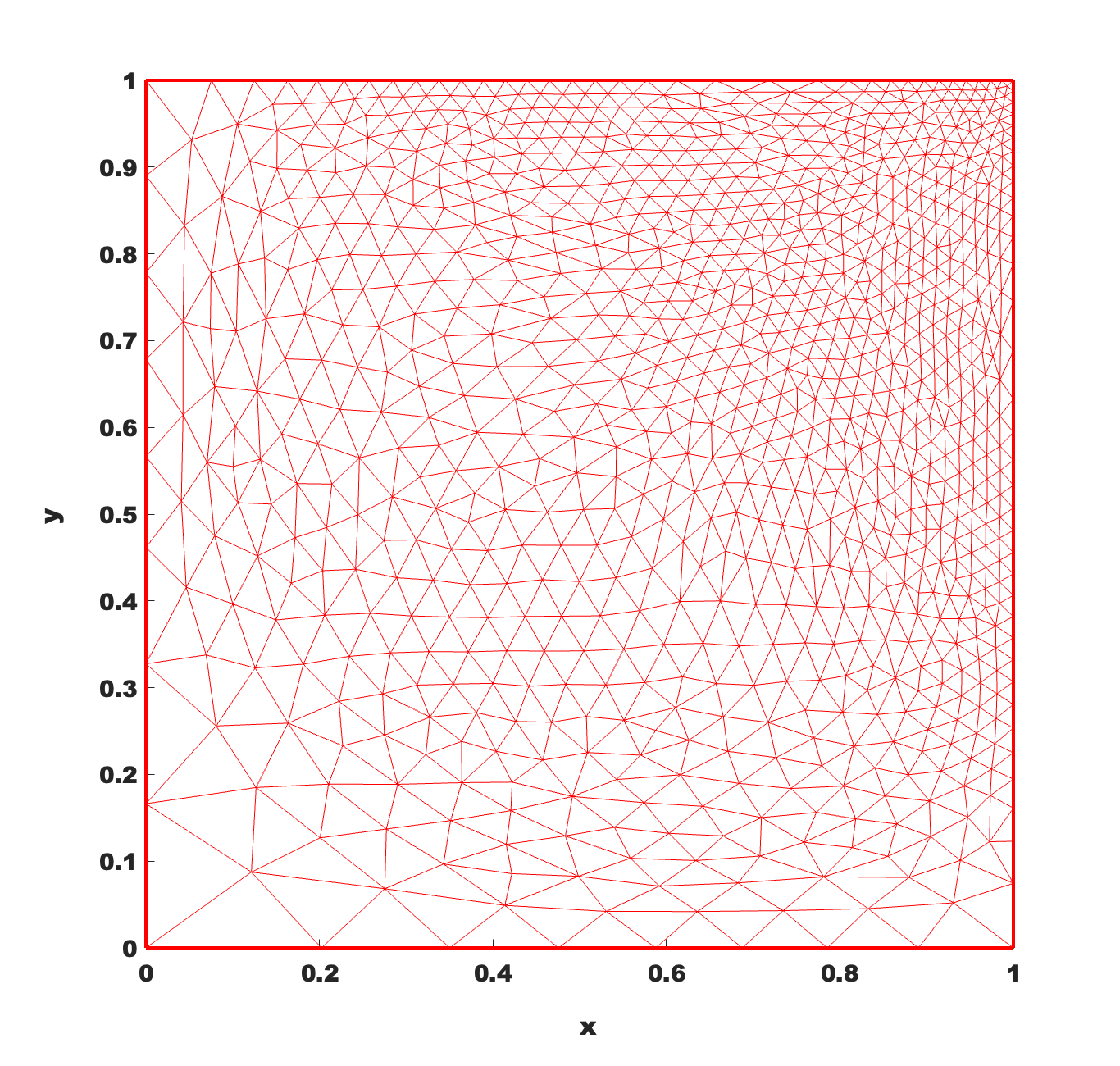}
\end{subfigure}
\caption{Uniform mesh with mesh size $128\times 128$ and the corresponding adaptive meshes.} \label{Ex1unifsfmesh}
\centering
\begin{subfigure}[b]{0.5\textwidth}
\centering
\includegraphics[width=0.99\linewidth]{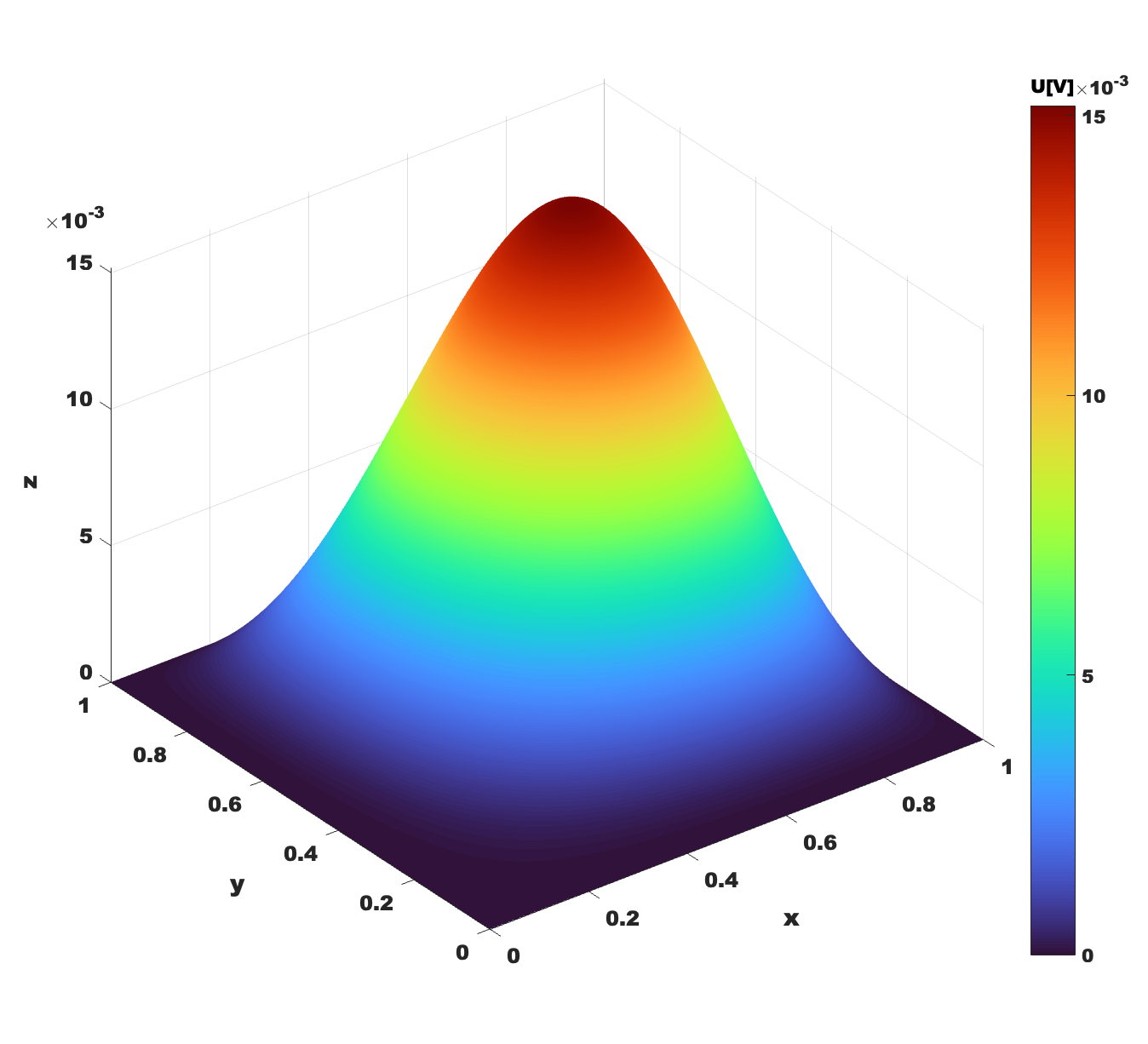} 
\end{subfigure}
\hfill
\begin{subfigure}[b]{0.49\textwidth}
\centering
\includegraphics[width=0.99\linewidth]{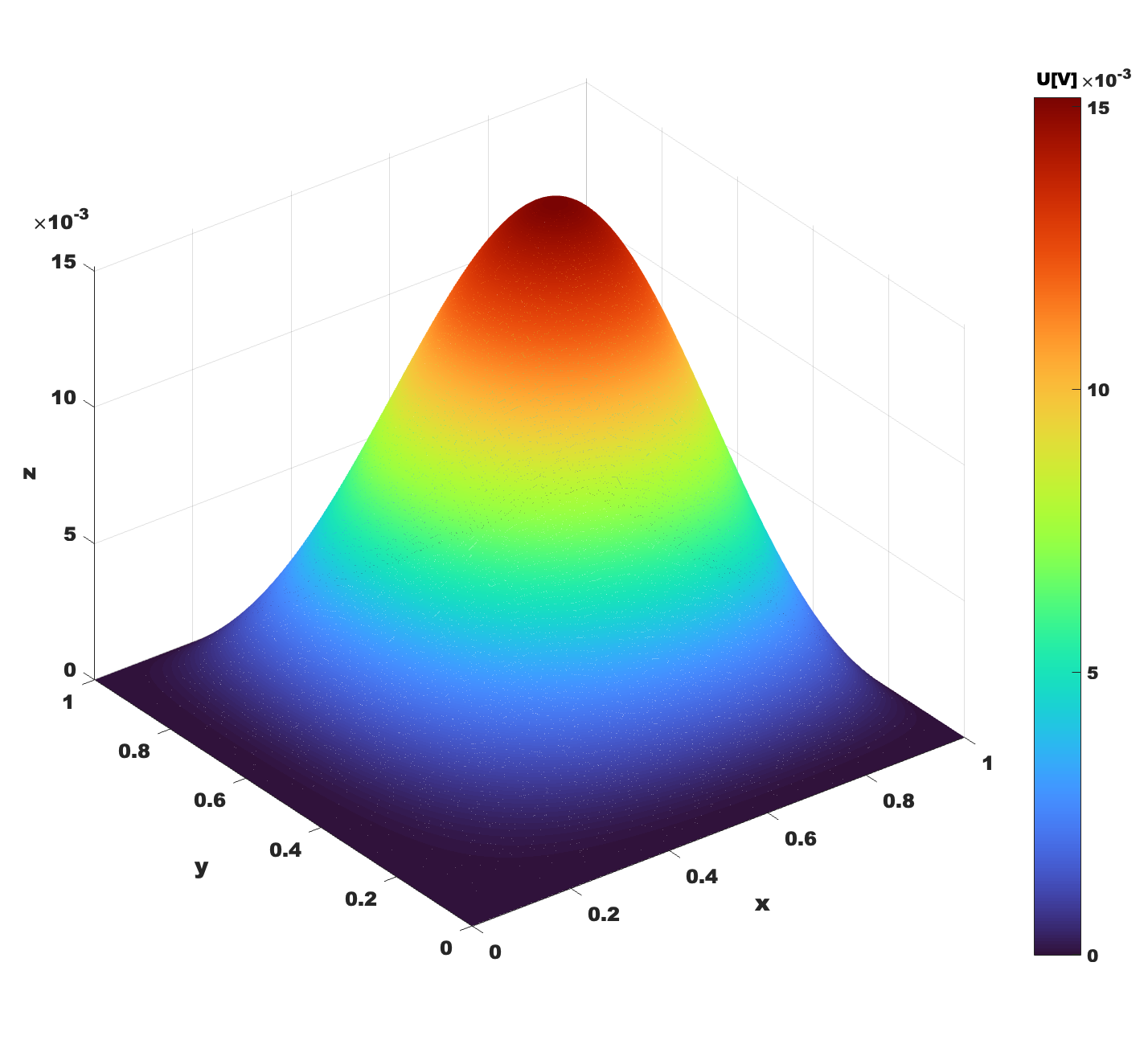}
\end{subfigure}
\caption{Profiles of Exact solution and the approximated DG  solution on the uniform mesh for the penalty $\sigma_i=10000$ and  fixed parameters $\alpha=1.0,\, \beta=1.0$ and $\mu=1.0$.} \label{Ex1exactcgsolalpha}
\end{figure}
\begin{table}[H]
\begin{center}
\caption{$L^2$-Error and $\mathcal{E}$-Error and their EOCs for P1dc-Elements} \label{table1examp1p1dc} 
\addtolength{\tabcolsep}{10.5pt}
\begin{tabular}{ |c||c|c|c| c|}
 \hline
 \multicolumn{5}{|c|}{\cellcolor{red!15} Parameters: $\alpha=1.0,\, \beta=1.0,\, \mu=0.5$} \\
 \hline
\cellcolor{brown!8} Mesh Size & $L^2$-Error & \cellcolor{brown!8} EOC &  $\mathcal{E}$-Error & \cellcolor{brown!8} EOC \\
 \hline
\cellcolor{brown!8}  $4\times 4$   & 2.3746e-02 &  \cellcolor{brown!8} --    &  3.3881e-02    & \cellcolor{brown!8} --  \\

\cellcolor{brown!8}  $8\times 8$   & 4.9324e-03 &  \cellcolor{brown!8} 2.2673   &  1.7014e-02    & \cellcolor{brown!8} 0.9937 \\

\cellcolor{brown!8} $16\times 16$ &  1.0958e-03  & \cellcolor{brown!8}  2.1703   &   8.4519e-03       & \cellcolor{brown!8} 1.0094 \\

\cellcolor{brown!8} $32\times 32$ & 2.4976e-04 & 
\cellcolor{brown!8}  2.1333 & 4.1834e-03   &     \cellcolor{brown!8}  1.0146  \\

\cellcolor{brown!8} $64\times 64$ &  5.9817e-05   & \cellcolor{brown!8}  2.0619   & 2.0285e-03   & \cellcolor{brown!8}  1.0443\\

\cellcolor{brown!8} $128\times 128$& 1.4763e-05  & \cellcolor{brown!8}  2.0186     &  9.6630e-04 &     \cellcolor{brown!8} 1.0699 \\
\hline
\end{tabular}
\end{center}
\end{table} 
\begin{table}[H]
\begin{center}
\caption{$L^2$-Error and $\mathcal{E}$-Error and their EOCs for P2dc-Elements.} \label{tab1ex1}
\addtolength{\tabcolsep}{13.5pt}
\begin{tabular}{ |c||c|c|c| c|}
 \hline
 \multicolumn{5}{|c|}{\cellcolor{red!15} Parameters: $\alpha=1.0,\, \beta=1.0,\, \mu=0.5$} \\
 \hline
\cellcolor{brown!8} Mesh Size & $L^2$-Error & \cellcolor{brown!8} EOC &  $\mathcal{E}$-Error & \cellcolor{brown!8} EOC \\
 \hline
 \cellcolor{brown!8}  $4\times 4$   &  3.5072e-05 &  \cellcolor{brown!8} --   &  5.1831e-04   & \cellcolor{brown!8} --  \\
 
\cellcolor{brown!8}  $8\times 8$ & 3.9039e-06     &  \cellcolor{brown!8} 3.1673   &  1.5306e-04   & \cellcolor{brown!8} 1.7597 \\

\cellcolor{brown!8} $16\times 16$ & 4.5434e-07   & \cellcolor{brown!8} 3.1031   &  4.3642e-05  & \cellcolor{brown!8} 1.8103  \\

\cellcolor{brown!8} $32\times 32$ & 5.6623e-08   &\cellcolor{brown!8}   3.0043      &  1.1306e-05  &     \cellcolor{brown!8}  1.9487  \\

\cellcolor{brown!8} $64\times 64$ &  7.1381e-09   & \cellcolor{brown!8}  2.9877   &  2.8369e-06  & \cellcolor{brown!8} 1.9946 \\

\cellcolor{brown!8} $128\times 128$&  9.0387e-10   & \cellcolor{brown!8}  2.9872     &  7.0218e-07  &     \cellcolor{brown!8} 2.0144 \\
\hline
\end{tabular}
\end{center}
\begin{center}
\caption{$L^2$-Error and $\mathcal{E}$-Error and their EOCs for P3dc-Elements}\label{tab2ex1p3dc}
\addtolength{\tabcolsep}{13.5pt}
\begin{tabular}{ |c||c|c|c| c|}
 \hline
 \multicolumn{5}{|c|}{\cellcolor{red!15} Parameters: $\alpha=1.0,\, \beta=1.0,\, \mu=0.5$} \\
 \hline
\cellcolor{brown!8} Mesh Size & $L^2$-Error & \cellcolor{brown!8} EOC &  $\mathcal{E}$-Error & \cellcolor{brown!8} EOC \\
 \hline
\cellcolor{brown!8}  $4\times 4$   &  4.3190e-06 &  \cellcolor{brown!8} --    &   2.4363e-05  & \cellcolor{brown!8} -- \\
 
\cellcolor{brown!8}  $8\times 8$   & 3.1804e-08  &  \cellcolor{brown!8} 3.7634   &  2.5915e-06   & \cellcolor{brown!8} 3.2328\\

\cellcolor{brown!8} $16\times 16$ &  2.2998e-08   & \cellcolor{brown!8}  3.7896  &  2.9123e-07  & \cellcolor{brown!8}  3.1536 \\

\cellcolor{brown!8} $32\times 32$ & 1.4669e-09   & \cellcolor{brown!8}  3.9707    &  3.5124e-08  &   \cellcolor{brown!8}  3.0516 \\

\cellcolor{brown!8} $64\times 64$ &  8.9770e-11   & \cellcolor{brown!8}    4.0304 &   4.3837e-09  & \cellcolor{brown!8} 3.0022  \\

\cellcolor{brown!8} $128\times 128$& 3.1763e-13   &\cellcolor{brown!8}    4.0856   & 5.4810e-10  &    \cellcolor{brown!8} 2.9996 \\
 \hline
\end{tabular}
\end{center}
\begin{center}
\caption{Convergence results in $p$-version refinement by fixing mesh size $30\times 30$.}\label{pversionref}
\addtolength{\tabcolsep}{9pt}
\begin{tabular}{ |c|c|c|c| c|}
 \hline
 \multicolumn{5}{|c|}{\cellcolor{red!15} Parameters: $\alpha=1.0,\, \beta=1.0,\, \mu=0.5$} \\
 \hline
\cellcolor{brown!8} $\underset{(n_i)}{\text{Poly. degree}}$ & $\|e\|_{L^2}$ & \cellcolor{brown!8} {$\%$}  $\frac{\|e\|_{L^2}}{\|u\|_{L^2}}$&  $\|e\|_{\mathcal{E}}$ & \cellcolor{brown!8} {$\%$}$\frac{\|e\|_{\mathcal{E}}}{\|u\|_{\mathcal{E}}}$ \\
 \hline
\cellcolor{brown!8}  1   &  1.0043e-04 &  16.9174  \cellcolor{brown!8}    &  3.7283e-03   & 14.4526\cellcolor{brown!8}  \\
 
\cellcolor{brown!8} 2   &  1.3906e-06 & 2.3427e-02 \cellcolor{brown!8}    & 2.9183e-05   & 1.1398e-01 \cellcolor{brown!8}\\

\cellcolor{brown!8} 3 &  9.0697e-08   & 1.5279e-03 \cellcolor{brown!8}    &  8.7052e-07 & 3.4002e-03\cellcolor{brown!8}   \\

\cellcolor{brown!8} 4 &  7.1810e-09 & 1.2096e-04 \cellcolor{brown!8}     &  4.6539e-08  &  1.8178e-04 \cellcolor{brown!8}   \\
 \hline
\end{tabular}
\end{center}

\begin{center}
\caption{Convergence in $hp$-version refinement.}\label{hptable}
\addtolength{\tabcolsep}{3pt}
\begin{tabular}{|c|c|c|c|c|c|c|}
 \hline
 \multicolumn{6}{|c|}{\cellcolor{red!15} Parameters: $\alpha=1.0,\, \beta=1.0,\, \mu=0.5$} \\
 \hline
\cellcolor{brown!8}  $\underset{(n_i)}{\text{Poly. degree}}$ & NDOF & \cellcolor{brown!8} $\|e\|_{L^2}$ & {$\%$} $\frac{\|e\|_{L^2}}{\|u\|_{L^2}}$ & \cellcolor{brown!8} $\|e\|_{\mathcal{E}}$  & {$\%$} $\frac{\|e\|_{\mathcal{E}}}{\|u\|_{\mathcal{E}}}$  \\ 
 \hline
 
\cellcolor{brown!8}  1   &  $175$ &  \cellcolor{brown!8} 5.0172e-03    &  94.2868  & \cellcolor{brown!8} 2.3242e-02 & 97.3054  \\
 
\cellcolor{brown!8}  2   & $1416$  &  \cellcolor{brown!8} 2.9929e-06   & 5.0282e-02   & \cellcolor{brown!8} 1.0241e-04 & 4.0002e-01 \\

\cellcolor{brown!8} 3 &  $5340$   & \cellcolor{brown!8}  3.0006e-8 &  5.0410e-04  & \cellcolor{brown!8} 1.9620e-06 & 7.6635e-03 \\

\cellcolor{brown!8} 4  & $14100$   & \cellcolor{brown!8}   1.32024e-09   &  2.2180e-05  &   \cellcolor{brown!8} 7.0856e-08  & 2.7676e-04 \\
 \hline
\end{tabular}
\end{center}
\end{table} 
\begin{figure} [H] 
    \centering
\begin{subfigure}[b]{0.5\textwidth}
\centering
\includegraphics[width=0.99\linewidth]{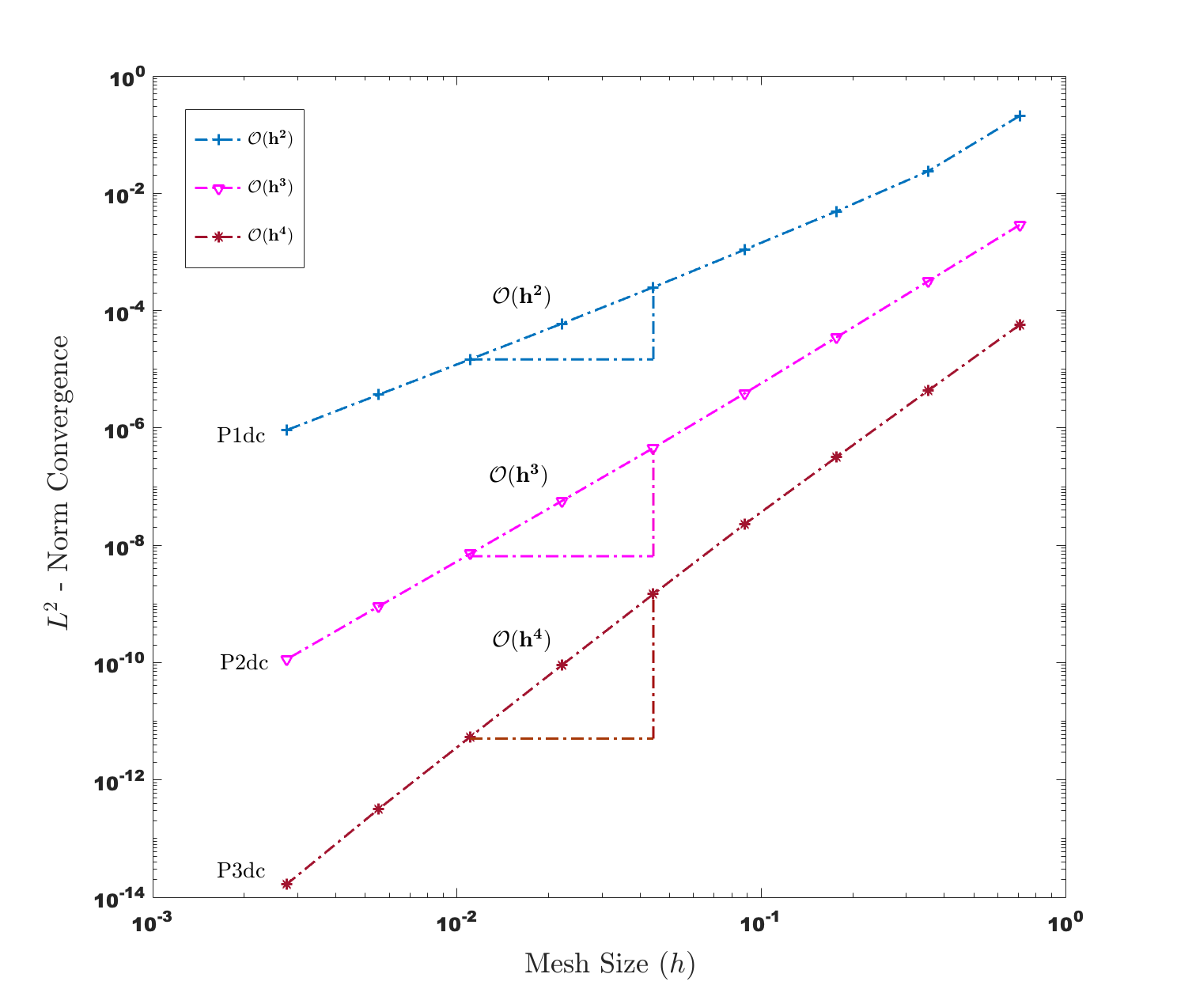} 
    \end{subfigure}
    \hfill
    \begin{subfigure}[b]{0.49\textwidth}
    \centering
 \includegraphics[width=1.0\linewidth]{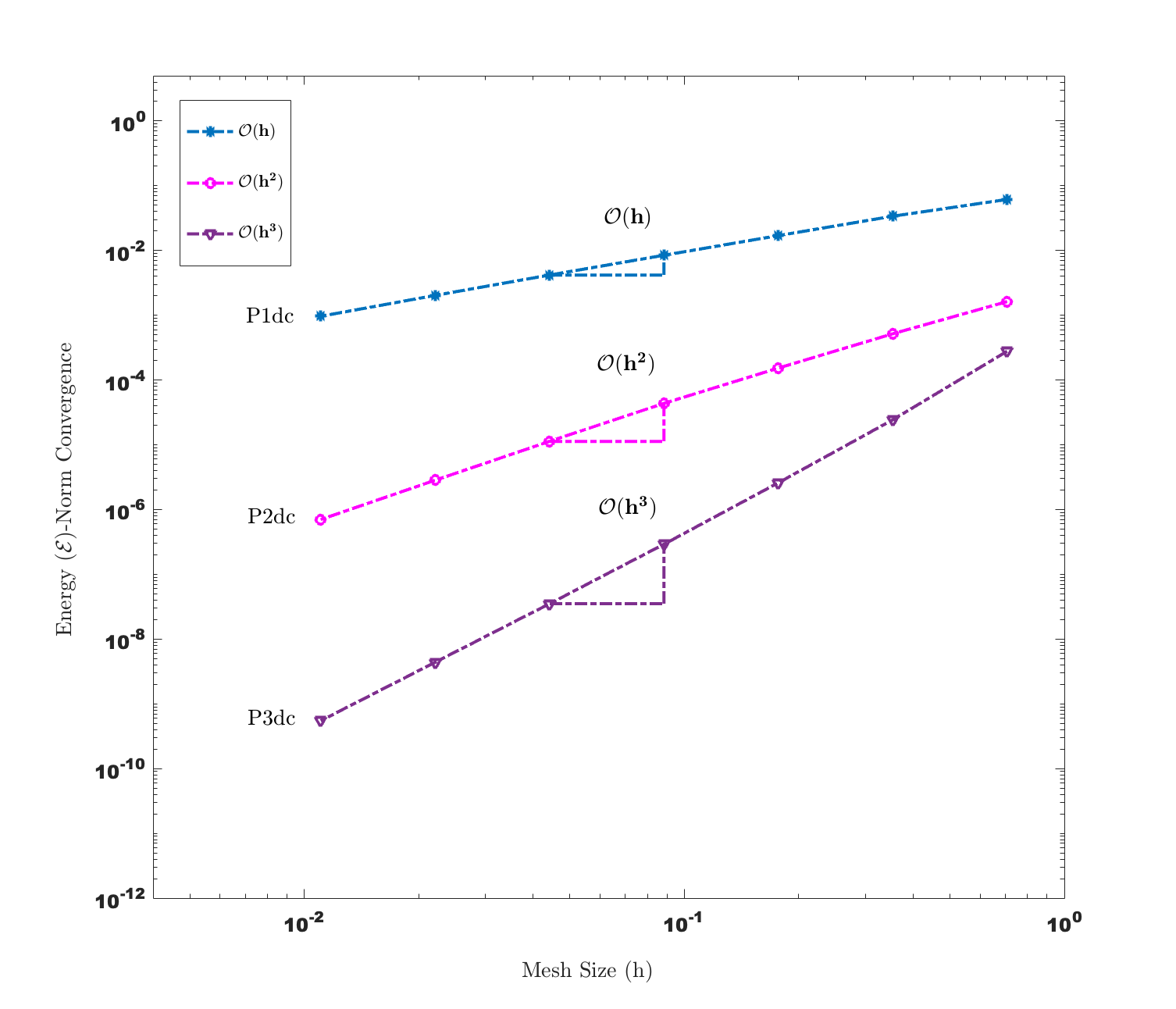} 
    \end{subfigure}
\caption{Plots of $L^2$-norm and energy ($\mathcal{E}$)-norm convergence in different polynomial spaces P1dc, P2dc, and P3dc, respectively} \label{Ex1L2H1eoc}
    \centering
\begin{subfigure}[b]{0.5\textwidth}
\centering
\includegraphics[width=0.99\linewidth]{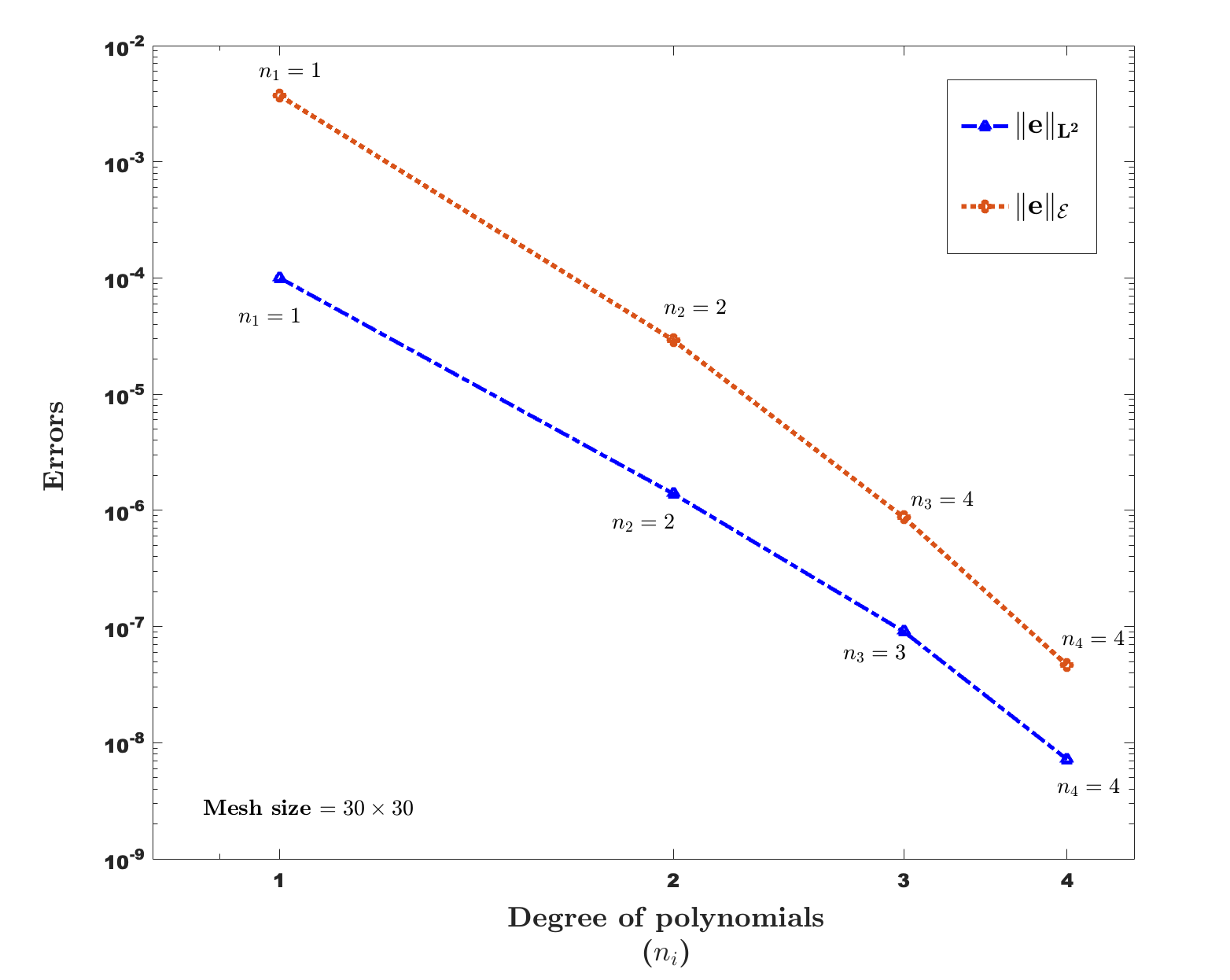} 
    \end{subfigure}
    \hfill
    \begin{subfigure}[b]{0.49\textwidth}
    \centering
 \includegraphics[width=1.0\linewidth]{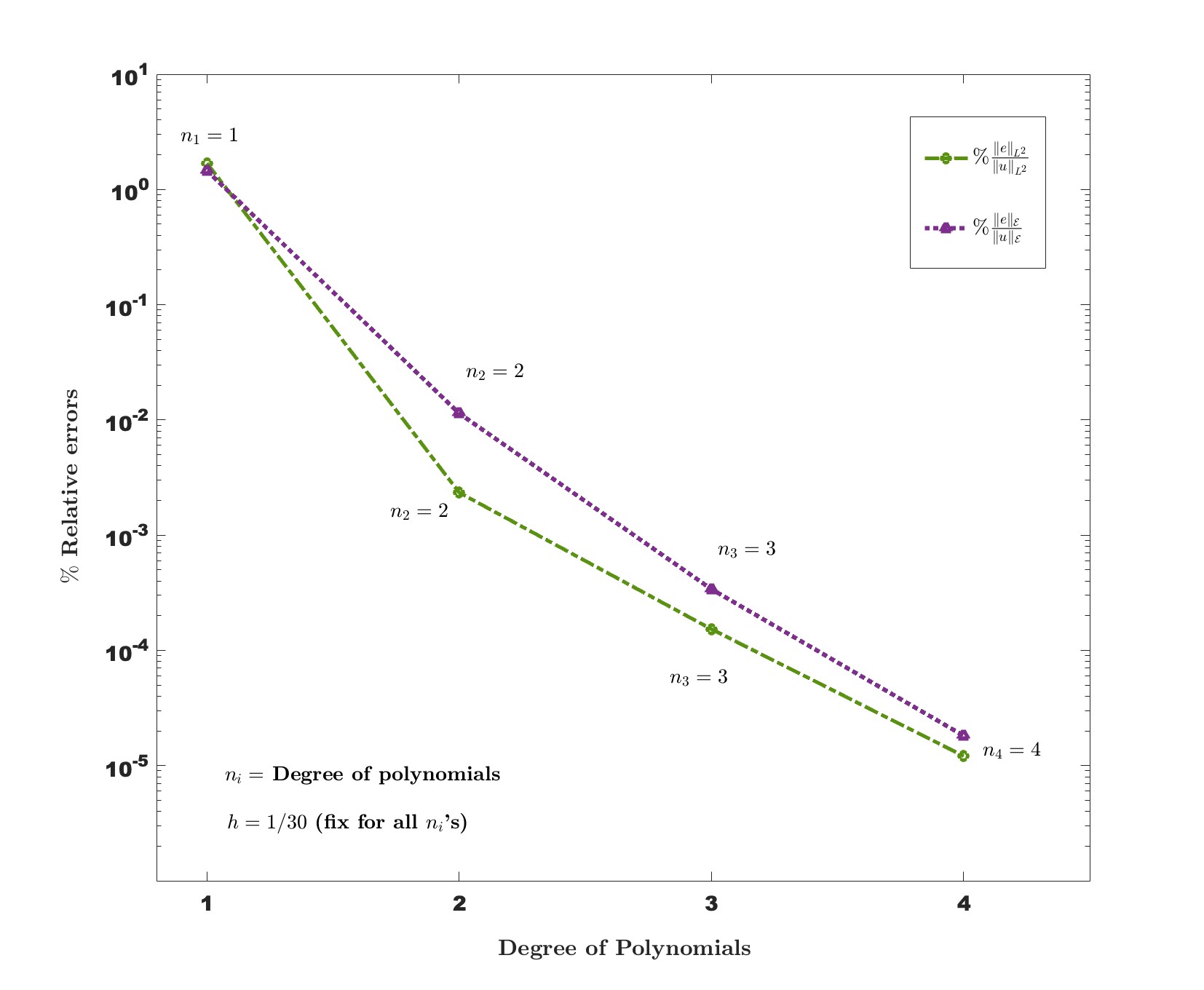} 
    \end{subfigure}
\caption{Convergence results for $p$-version refinement} \label{pversion} 
 \centering
\begin{subfigure}[b]{0.5\textwidth}
\centering
\includegraphics[width=0.99\linewidth]{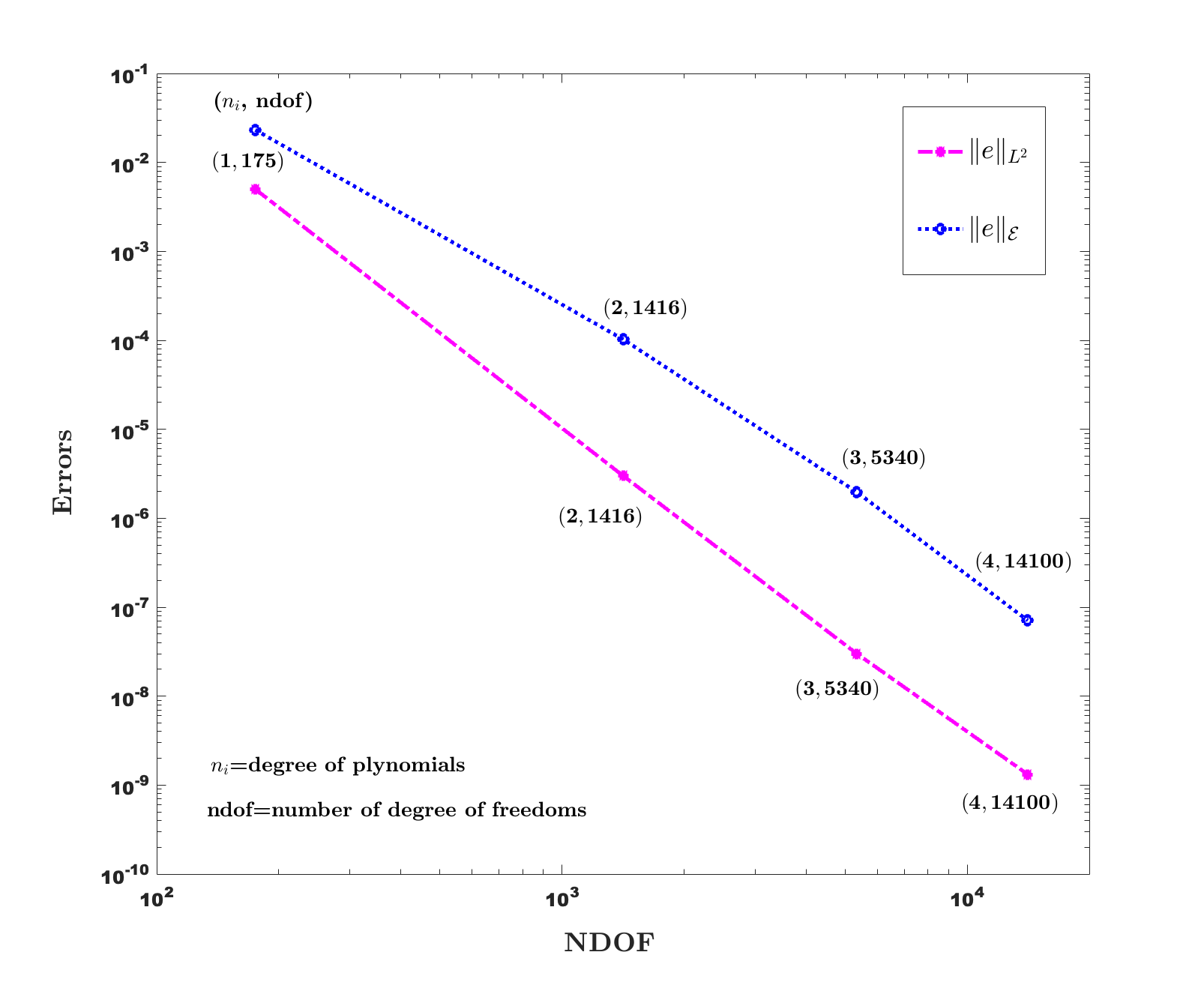} 
    \end{subfigure}
    \hfill
    \begin{subfigure}[b]{0.49\textwidth}
    \centering
 \includegraphics[width=1.0\linewidth]{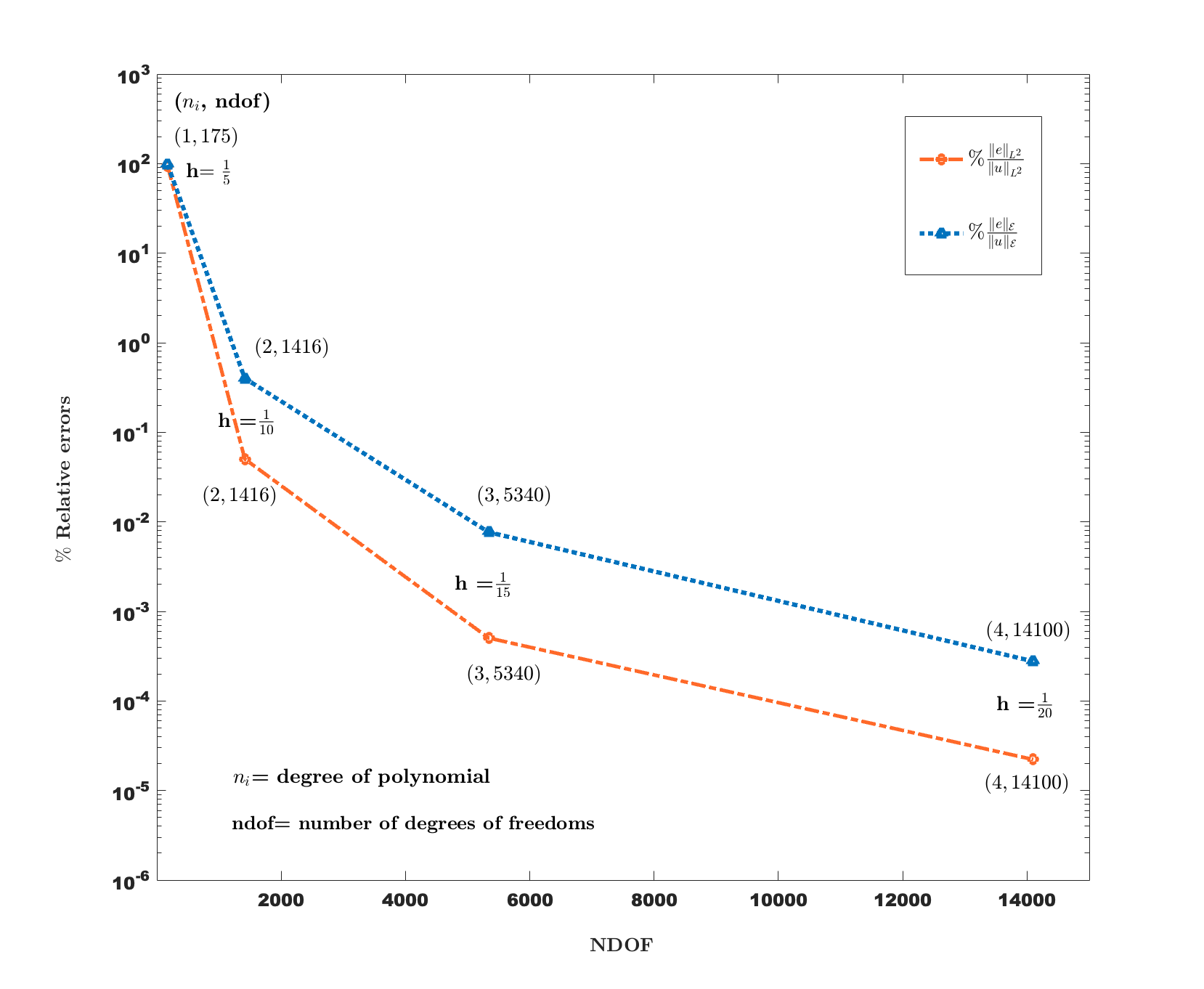} 
    \end{subfigure}
\caption{Convergence results for $hp$-version refinement} \label{hpversion} 
\end{figure}
\begin{exam}[Elastic domain with a single crack] \label{crackdomainex-2}
An elastic domain containing a single crack under anti-plane shear loading
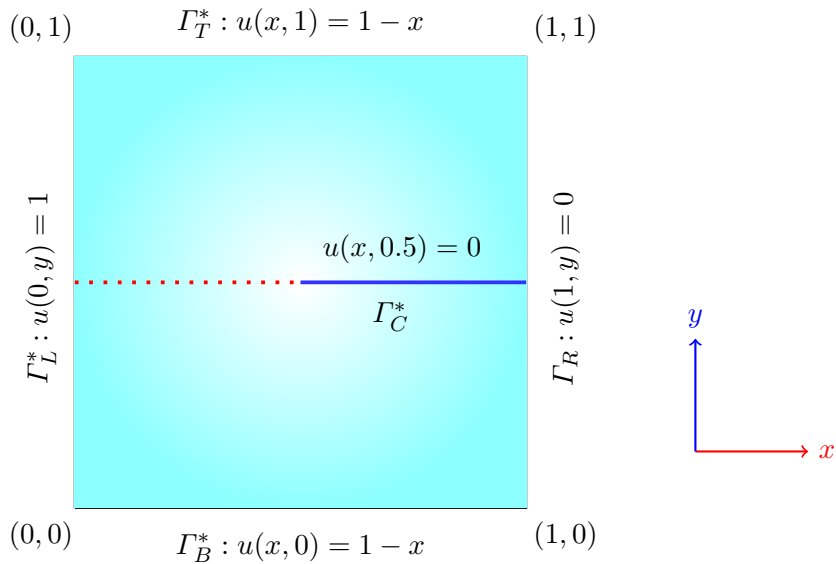
\begin{figure}[H]
\centering
\begin{tikzpicture}[scale=1.5]
\filldraw[fill=red!8, draw=black, thick] (0,0) -- (4,0) -- (4,4) -- (0,4) -- (0,0);
\shade[outer color=cyan! 45, inner color=white] (0,0) rectangle (4,4);
\draw [line width=0.5mm, blue!80]  (2,2) -- (4,2);
\draw [loosely dotted, line width = 0.5mm, red] (0,2) -- (2,2);
\node at (-0.3,-0.25){$(0,0)$};
\node at (4.35,4.25){$(1,1)$};
\node at (4.35,-0.25){$(1,0)$};
\node at (-0.3,4.25){$(0,1)$};
\node at (-0.3, 2.9)[anchor=east, rotate=90]{$\Gamma^{*}_L: u(0,y)=1$};
\node at (4.35, 2.9)[anchor=east, rotate=90]{$\Gamma_R: u(1,y)=0$};
\node at (2, -0.35){$\Gamma^{*}_{B}: u(x,0)=1-x$};
\node at (2, 4.3){$\Gamma^{*}_{T}: u(x,1)=1-x$};
\node at (2.8, 1.7){$\Gamma^{*}_{C}$};
\node at (2.9, 2.3){$u(x,0.5)=0$};
\def\xOrigin{5.5}
\def\yOrigin{0.5} 
\draw[->, color=red, thick] (\xOrigin, \yOrigin) -- (\xOrigin + 1, \yOrigin) node[right] {$x$};
\draw[->, color=blue, thick] (\xOrigin, \yOrigin) -- (\xOrigin, \yOrigin + 1) node[above] {$y$};
\end{tikzpicture}
\caption{A square domain containing a single edge crack. }
\end{figure}
\end{exam}
In this example, we test our DGFEM algorithm for a case of an elastic domain containing a single edge crack under anti-plane shear loading. Such a model is instrumental in describing the behavior of materials like brittle solids or polymers. Anti-plane shear conditions can cause fractures under torsional or twisting loads. The static anti-plane shear crack model describes how a material deforms and the stresses that arise when subjected to shear forces acting in a direction perpendicular to a plane. Within the linearized theory of elasticity, at the crack tip, both stresses and strains exhibit a singular behavior. It is interesting to test whether the model studied in this paper produces the same singularity in stresses and bounded strains in the neighborhood of the crack tip. 

The computational details for this example are as follows: the mesh is fixed at $h=1/30$, model parameters are taken as $\alpha=2.0,\, \beta=2.0$, material parameter, shear modulus, is fixed as $1.0$. The computations are done for different penalty parameters in the DG method, i.e., $\sigma =1.0,\, 10,\, 100,\, 1000,\, 10000$. The algorithm~$5.1$ is utilized in the computations. Figure \ref{Ex2bt=0.0} depicts the uniform mesh (left) and the adaptive mesh (right). One may observe that the adaptive meshes are well reflected near the crack tip, and such a mesh helps capture the singular behavior of the stress field. Figures~\ref{dgsolsigex2p1}--\ref{dgsolsigex2p3} show the DG solution at the different penalty parameters. The right of Figure~\ref{dgsolsigex2p3} depicts the DGFEM solution for the linear problem, i.e., for $\beta=0.0$, on the computational mesh.

Figure~\ref{Ex2T23E23betas} shows the effect of $\beta$ on the crack-tip stresses and strains. The computations are done for fixed $\alpha$. It is clear from the figure that the proposed model predicts stress concentration similar to the linear model. However, the strain growth is far slower in the nonlinear model with higher $\beta$ values.
 \begin{figure}[H]
    \centering
\begin{subfigure}[b]{0.5\textwidth}
\centering
\includegraphics[width=0.99\linewidth]{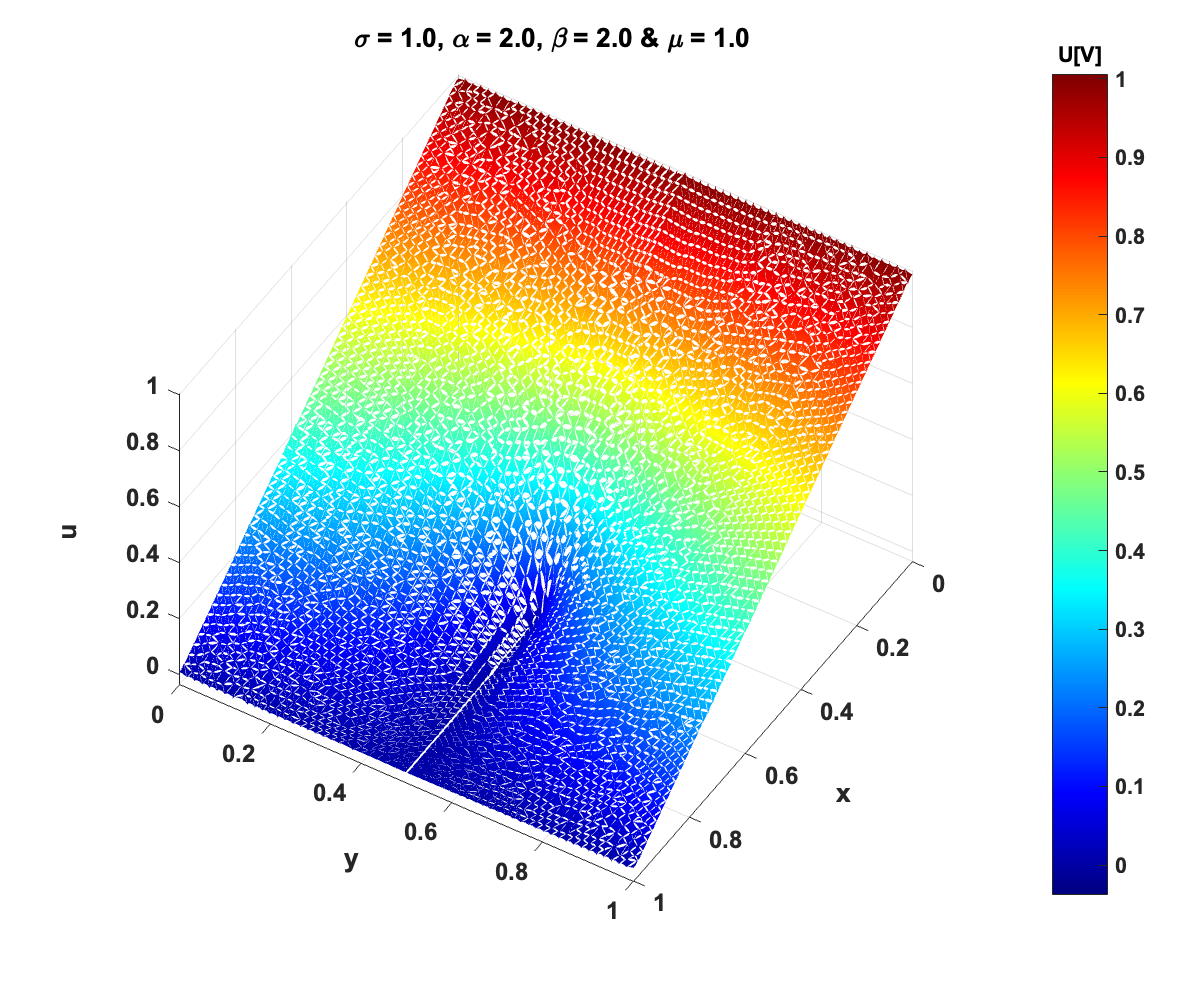} 
    \end{subfigure}
    \hfill
    \begin{subfigure}[b]{0.49\textwidth}
    \centering
 \includegraphics[width=0.99\linewidth]{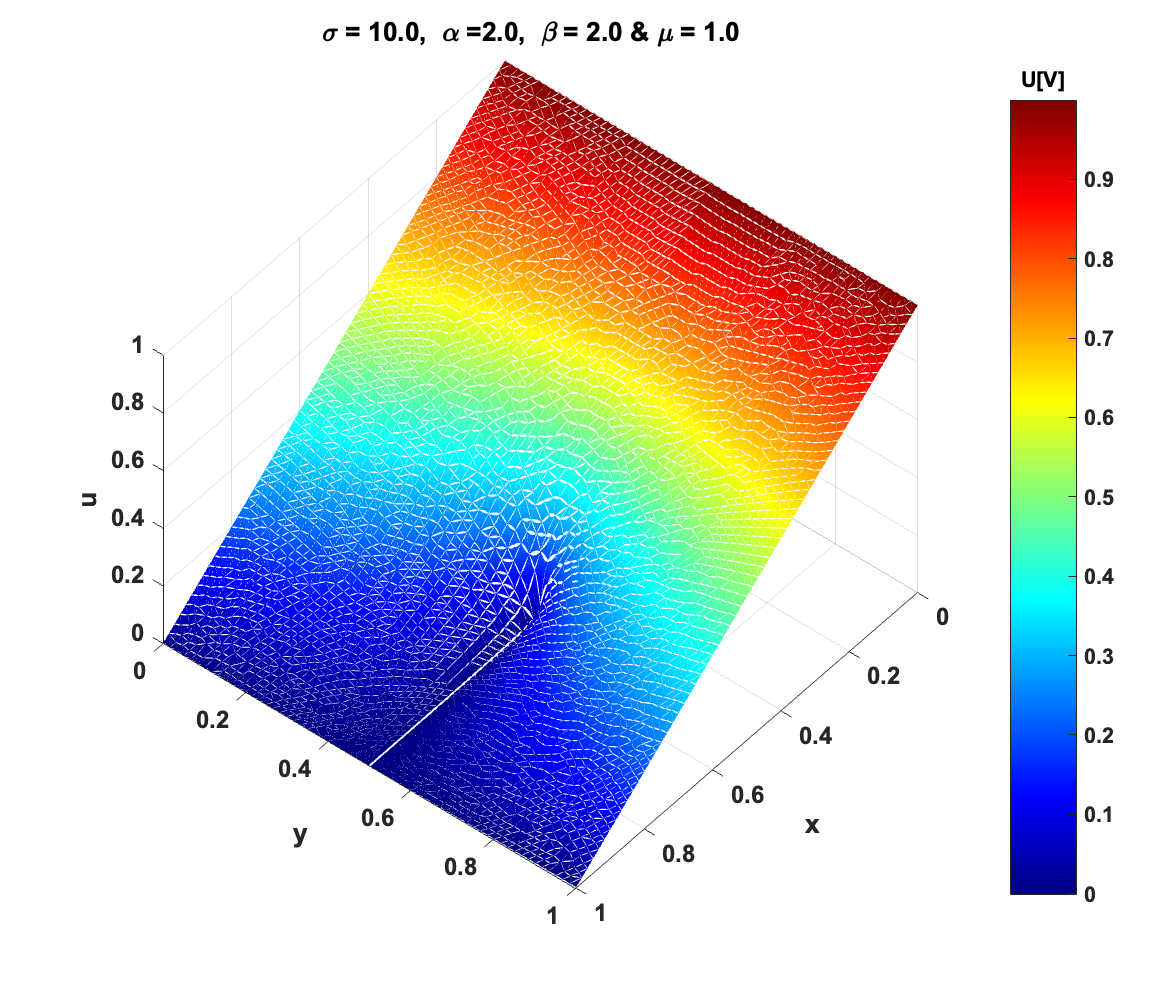} 
    \end{subfigure}
\caption{Profiles of DG solution with penalty parameters $\sigma=1.0 \, \& \, 10.0$.} \label{dgsolsigex2p1} 
\centering
\begin{subfigure}[b]{0.49\textwidth}
\centering
\includegraphics[width=0.99\linewidth]{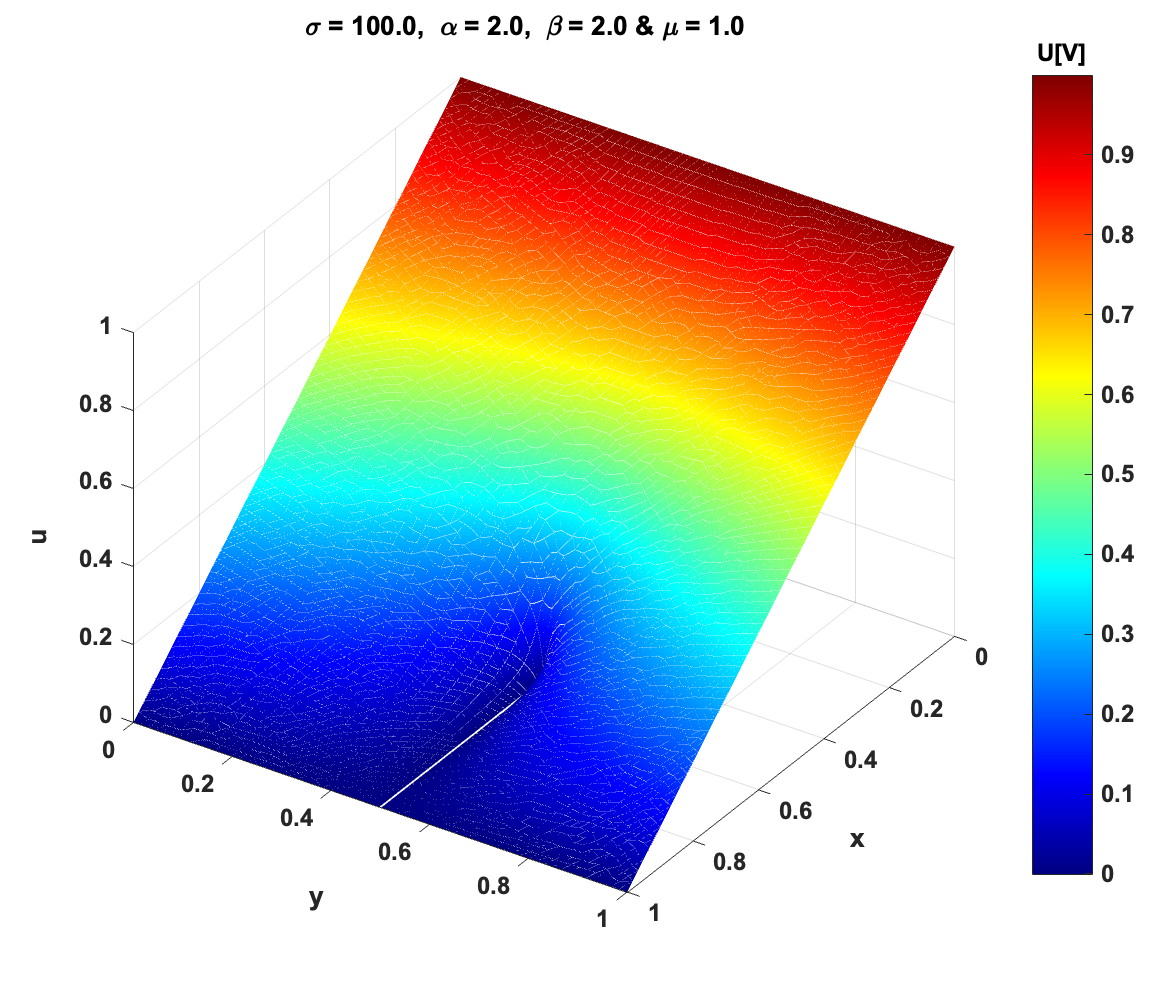} 
 \end{subfigure}
\hfill
\begin{subfigure}[b]{0.49\textwidth}
\centering
\includegraphics[width=0.99\linewidth]{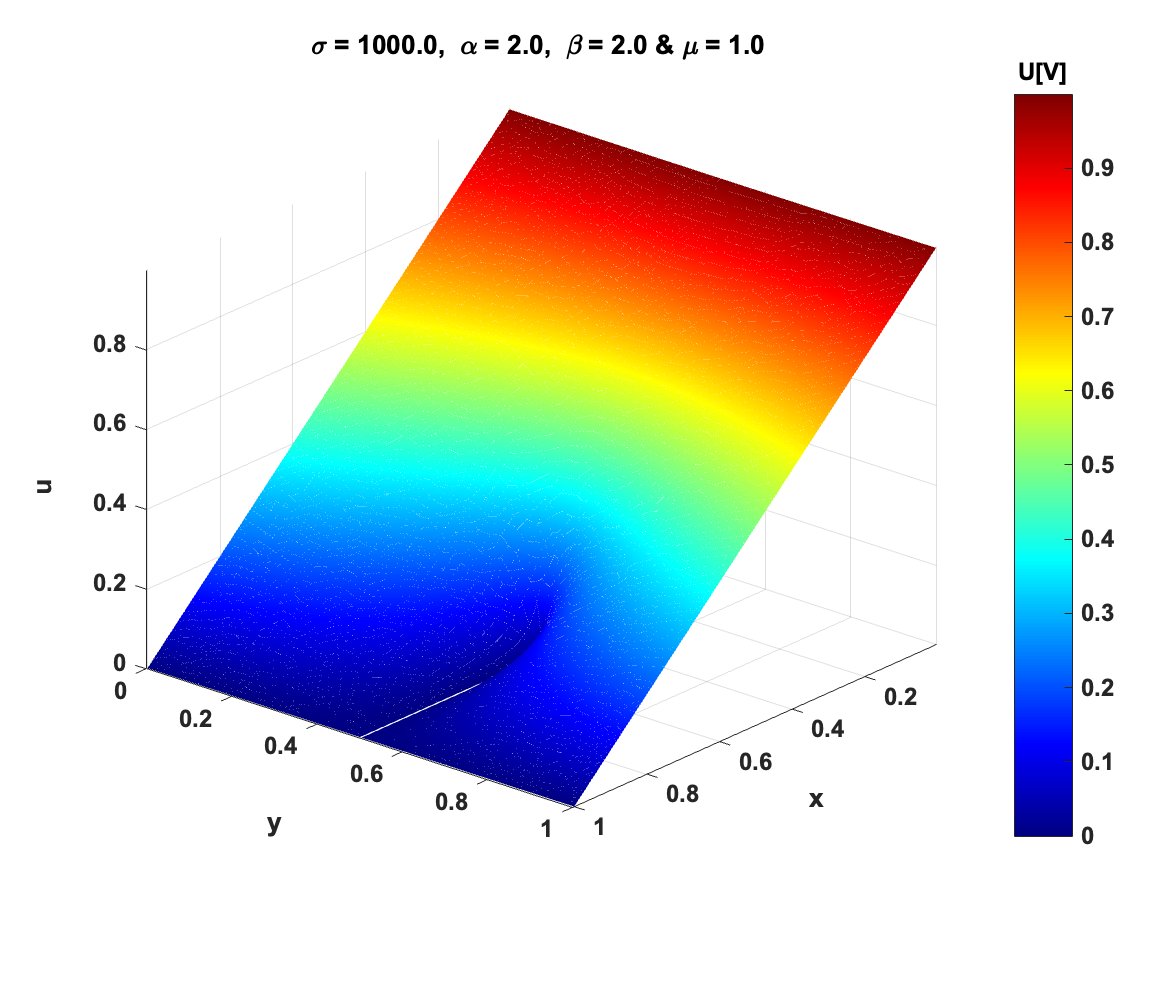} 
\end{subfigure}
\caption{Profiles of DG solution with penalty parameters $\sigma=10^2 \, \&\, 10^3$.} \label{dgsolsigex2p2} 
    \centering
\begin{subfigure}[b]{0.49\textwidth}
    \centering
 \includegraphics[width=0.99\linewidth]{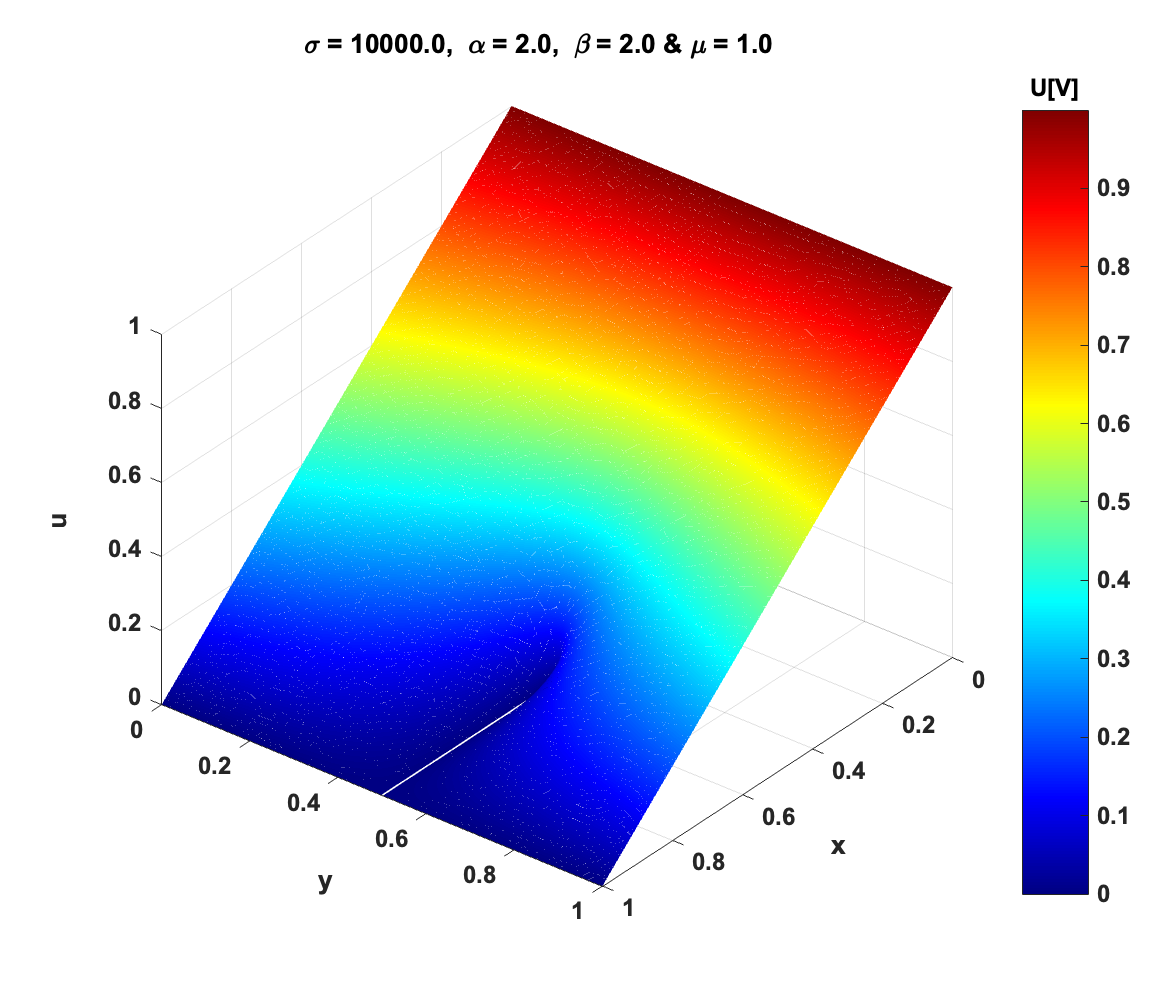} 
    \end{subfigure}
    \hfill
    \begin{subfigure}[b]{0.49\textwidth}
    \centering
 \includegraphics[width=1.0\linewidth]{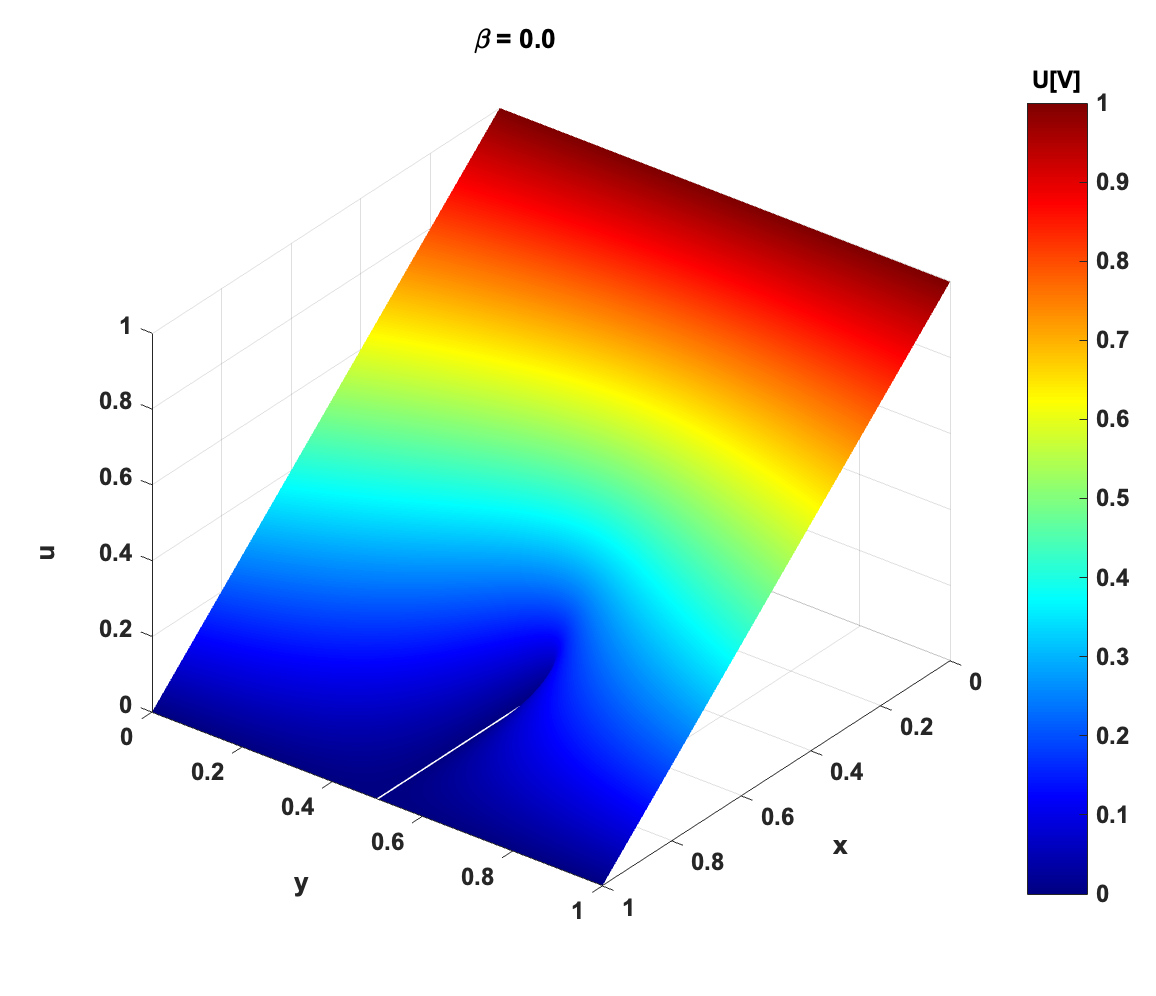} 
    \end{subfigure}
\caption{Profiles of DG solution for parameters $\sigma=10^4, \alpha=2.0,\, \beta=2.0, \, \mu=1.0$ and the approximate solution with $\beta=0.0$.}  \label{dgsolsigex2p3}
\end{figure}
\begin{figure}[H] 
    \centering
\begin{subfigure}[b]{0.49\textwidth}
\centering
\includegraphics[width=1.1\linewidth]{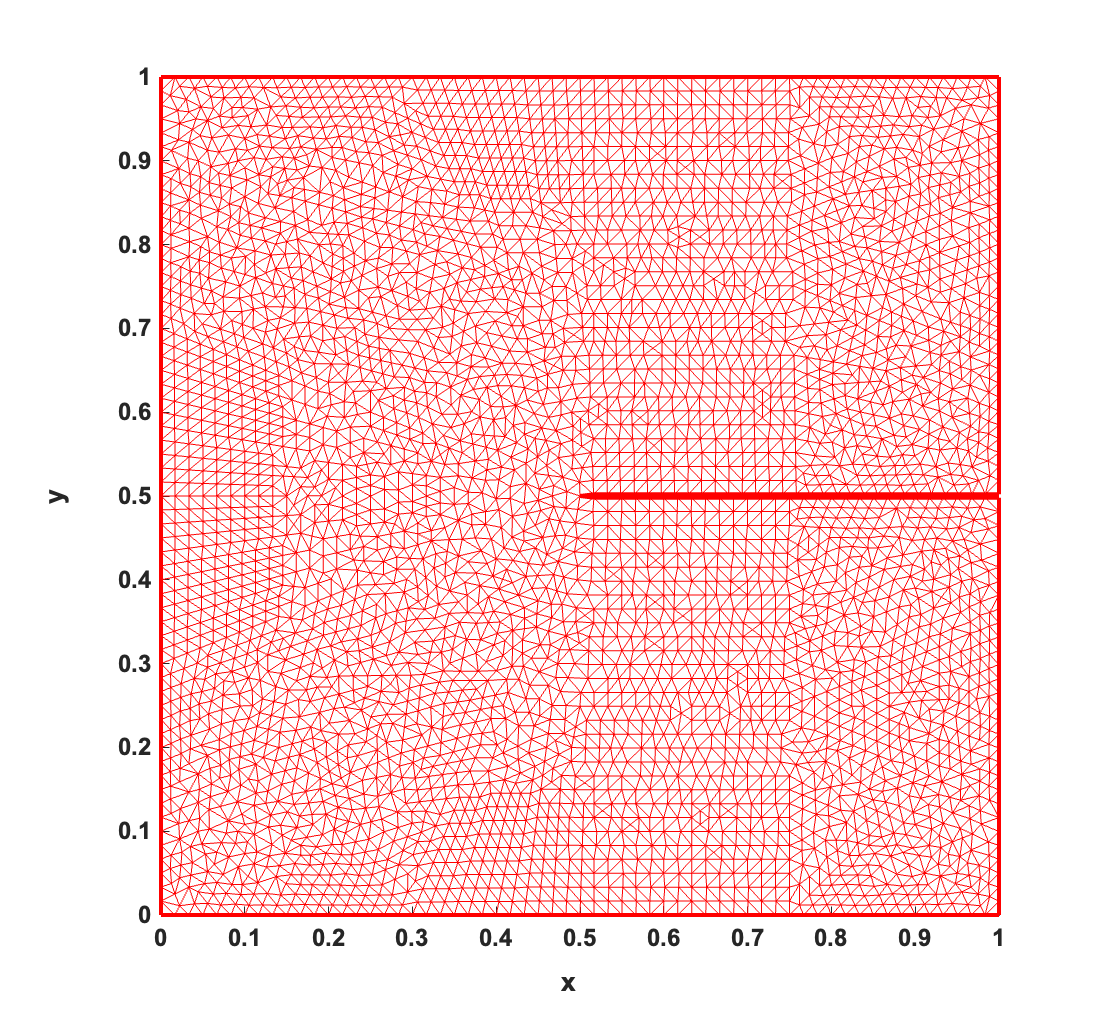} 
    \end{subfigure}
 \hfill
    \begin{subfigure}[b]{0.49\textwidth}
    \centering
 \includegraphics[width=1.12\linewidth]{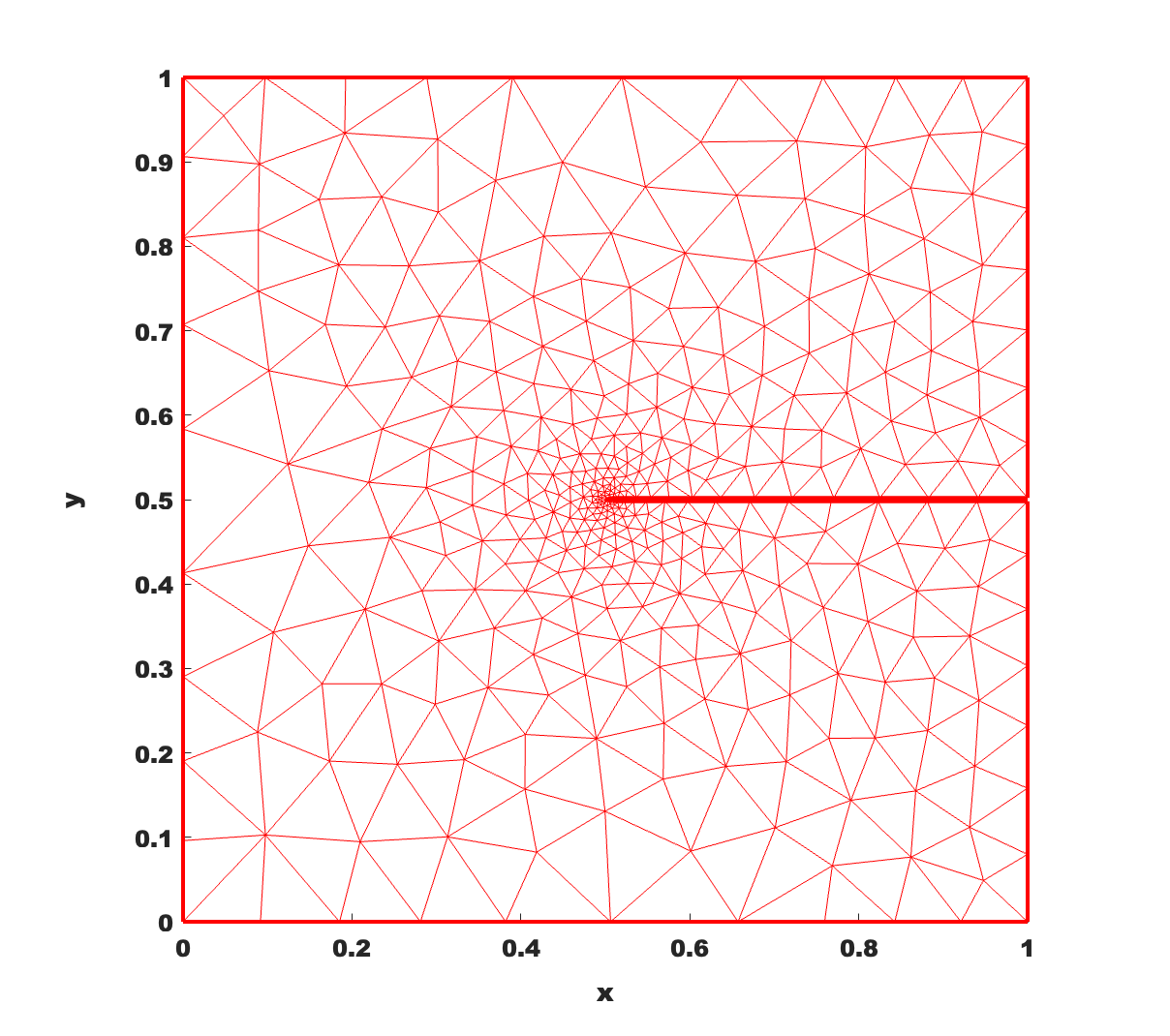} 
    \end{subfigure}
\caption{Uniform mesh with size $30\times 30$ and the corresponding Adaptive meshes.} \label{Ex2bt=0.0}\vspace{1.3cm}
\centering
\begin{subfigure}[b]{0.49\textwidth}
\centering
\includegraphics[width=1.1\linewidth]{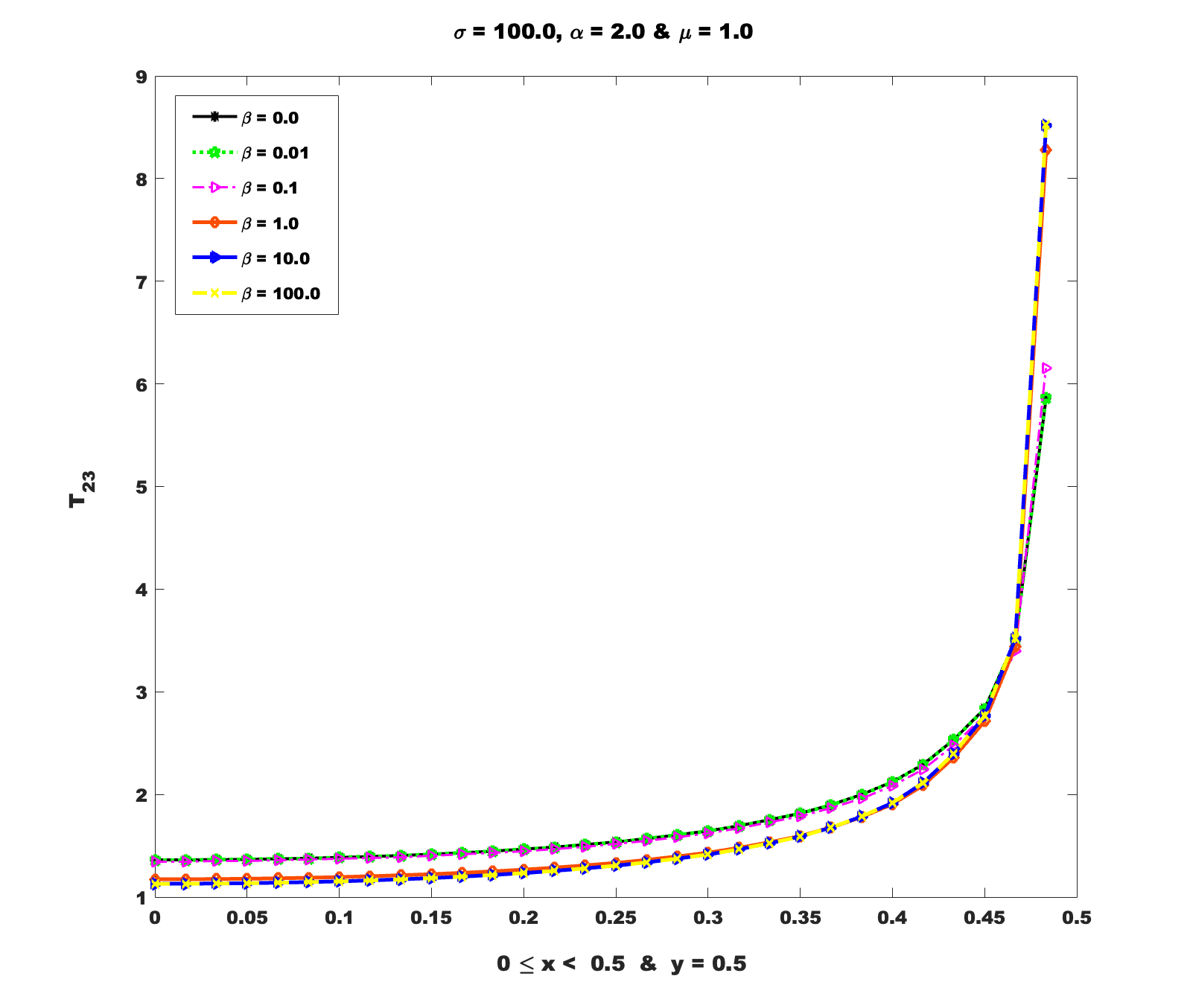} 
    \end{subfigure}
    \hfill
    \begin{subfigure}[b]{0.49\textwidth}
    \centering
 \includegraphics[width=1.1\linewidth]{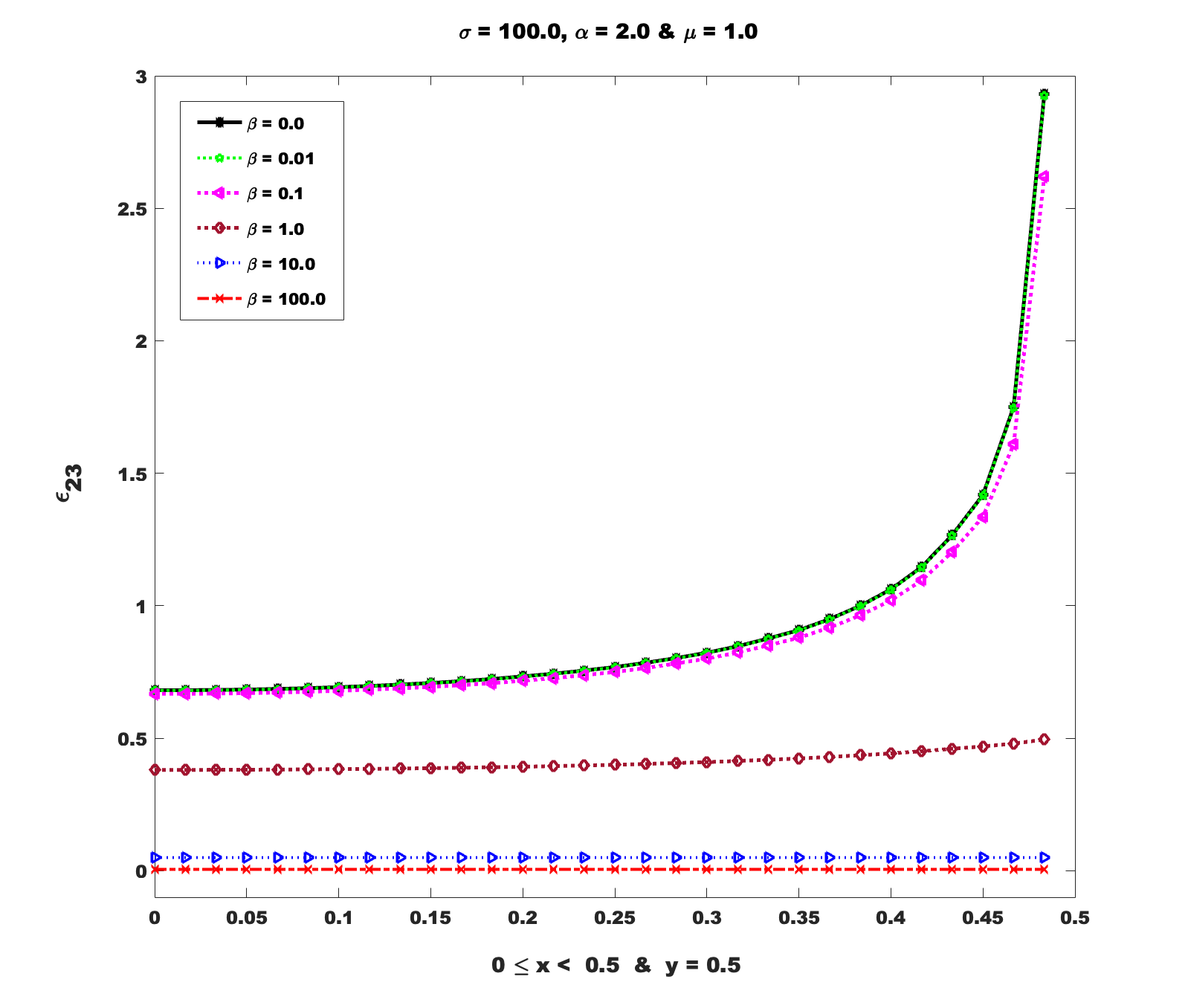} 
    \end{subfigure}
\caption{The stress  $T_{23}$  and strain $\epsilon_{23}$ for $\beta=0.0, \,0.01,\, 0.1,\, 1.0,\, 10\,\& \, 100$ computed on $0\leq x<0.5$ and $y=0.5$.}\label{Ex2T23E23betas}  \vspace{1.3cm}
 \end{figure}
\section{Conclusion} 
In this study, we have addressed the issue of devising an $hp$-local discontinuous Galerkin finite element to approximate the solution to a quasi-linear elliptic boundary value problem. Such a problem arises in modeling the response of a geometrically linear elastic body using an algebraically nonlinear constitutive relationship. The proposed response function is monotone and Lipschitz continuous. Our continuous DG formulation is well-posed, and further, for the discrete problem, we utilize the Ritz representation theory to show the existence of a unique solution. Then, the local polynomial degree $n_i$ for each triangle $\tau_i$, $\tau_i\in \mathcal{T}_h$, is enforced to approximate the solution variable for the discrete problem.
We also derive the apriori error estimates in the energy and $L^2$ norms. For the sufficient small $h$, $\mathtt{m}\geq 2$ and $\gamma=d-1$ with $d\geq2$, the derived error estimates are optimal in $h$  and suboptimal in $\mathtt{n}$ (total degree of polynomial).  The numerical test presented for a case of manufactured solution delineates the performance of the proposed locally adaptive DG scheme. The method presented in this paper can be easily used to characterize the crack-tip stress and strain in the material body under anti-plane shear containing a single crack. We have found that the stress is maximum along a line leading to the crack tip, but the strain growth is insignificant. Therefore, one can utilize the crack-tip stresses to define evolution criteria to study quasi-static and dynamic propagation problems in this class of new material models. Such an investigation is an automatic extension of the current research work. 
\section{Acknowledgement.}
This material is based upon work supported by the National Science Foundation under Grant No. 2316905. 
\bibliographystyle{plain}
\bibliography{references}
\end{document}